\RequirePackage[l2tabu, orthodox]{nag}
\documentclass[twoside]{amsart}

\usepackage[utf8]{inputenc}
\usepackage[T1]{fontenc}

\usepackage{amsmath,amsfonts,amssymb,amsxtra,setspace,xspace,graphicx,lmodern,psfrag,color,latexsym,bbm,comment,mathtools,amsthm, enumerate,enumitem,pifont,mathrsfs,caption,lmodern,microtype}
\usepackage[hidelinks,colorlinks=false]{hyperref}\hypersetup{urlcolor=blue, citecolor=black}

\usepackage[a4paper,includeheadfoot,margin=2.52 cm]{geometry}
\usepackage[acronym,nonumberlist,nogroupskip]{glossaries}
\usepackage[capitalise]{cleveref}
\usepackage{autonum}
\newtheorem{thm}{Theorem}[section]
\newtheorem{cor}[thm]{Corollary}
\newtheorem{lem}[thm]{Lemma}
\newtheorem{prop}[thm]{Proposition}

\newtheorem{rem}[thm]{Remark} 
\newtheorem{defn}[thm]{Definition}%[section]
\numberwithin{equation}{section}
\renewcommand{\glossarysection}[2][]{} %% removes glossaries from toc
\DeclarePairedDelimiter{\abs}{\lvert}{\rvert}
\DeclarePairedDelimiter{\norm}{\lVert}{\rVert}
\DeclarePairedDelimiter{\bra}{(}{)}
\DeclarePairedDelimiter{\pra}{[}{]}
\DeclarePairedDelimiter{\set}{\{}{\}}
\DeclarePairedDelimiter{\skp}{\langle}{\rangle}

\DeclareMathAlphabet{\mathup}{OT1}{\familydefault}{m}{n}
\newcommand{\dx}[1]{\mathop{}\!\mathup{d} #1}

\newcommand{\cC}{\ensuremath{\mathcal C}}
\newcommand{\cD}{\ensuremath{\mathcal D}}
\newcommand{\cE}{\ensuremath{\mathcal E}}
\newcommand{\cF}{\ensuremath{\mathcal F}}
\newcommand{\cG}{\ensuremath{\mathcal G}}
\newcommand{\cH}{\ensuremath{\mathcal H}}
\newcommand{\cI}{\ensuremath{\mathcal I}}
\newcommand{\cJ}{\ensuremath{\mathcal J}}

\newcommand{\cL}{\ensuremath{\mathcal L}}

\newcommand{\cT}{\ensuremath{\mathcal T}}

\newcommand{\cW}{\ensuremath{\mathcal W}}

\newcommand{\crF}{\ensuremath{\mathscr F}}

\newcommand{\stable}{\ensuremath{\mathup{s}}}
\newcommand{\unstable}{\ensuremath{\mathup{u}}}
\newcommand{\Leb}{\ensuremath{{L}}}
\newcommand{\SobH}{\ensuremath{{H}}}
\newcommand{\SobW}{\ensuremath{{\cW}}}
\newcommand{\HH}{\ensuremath{{\mathbb H}}}

\DeclareMathOperator*{\argmin}{arg\,min}

\newcommand{\N}{{\mathbb N}}
\newcommand{\R}{{\mathbb R}}

\newcommand{\T}{{\mathbb T}}
\newcommand{\Z}{{\mathbb Z}}
\newcommand{\B}{{\mathbb B}}
\newcommand{\Pro}{\mathcal{P}}
\newcommand{\Prac}{\mathcal{P}_{\mathup{ac}}}

\newcommand{\Sym}{\mathrm{Sym}}

\newcommand{\intom}[1]{\int #1 \dx{x}}

\newcommand{\iintom}[1]{\iint #1 \dx{x}\dx{y}}
\newcommand{\intT}[1]{\int #1 \dx{x}}
\newcommand{\infn}{{\mathbbm{1}}}

\DeclareMathOperator{\Ima}{Im}
\DeclareMathOperator{\codim}{codim}
\DeclareMathOperator{\ind}{ind}

\DeclareMathOperator{\dive}{\nabla \cdot}

\def\weak{\rightharpoonup}
\def\strong{\to}

\renewcommand{\tilde}{\widetilde}
\renewcommand{\hat}{\widehat}
\renewcommand{\bar}{\overline}
\newenvironment{tenumerate}[1][]
  {\enumerate[label=\textup(\alph*\textup),ref=(\alph*),#1]}
  {\endenumerate}
\def\XXint#1#2#3{{\setbox0=\hbox{$#1{#2#3}{\int}$}
     \vcenter{\hbox{$#2#3$}}\kern-.5\wd0}}

\theoremstyle{plain}

\title[The Mckean--Vlasov equation on the torus]{Long-time behaviour and phase transitions for the Mckean--Vlasov equation on the torus}

\author{J. A. Carrillo}
\author{R. S. Gvalani}
\author{G. A. Pavliotis}
\address{Department of Mathematics, Imperial College London, London SW7 2AZ}
\email{carrillo@imperial.ac.uk, rg1314@ic.ac.uk, g.pavliotis@imperial.ac.uk}
\author{A. Schlichting}
\address{Institut f\"ur Angewandte Mathematik, Universit\"at Bonn} %\\ Endenicher Allee 60, 53115 Bonn, Germany
\curraddr{Institut für Geometrie und Praktische Mathematik, RWTH Aachen} 
\email{schlichting@iam.uni-bonn.de, schlichting@igpm.rwth-aachen.de}
\thanks{JAC was partially supported by the EPSRC through grant number EP/P031587/1. RSG is funded by an Imperial College President's PhD Scholarship, partially through EPSRC Award Ref. 1676118. RSG also acknowledges the hospitality of RWTH Aachen University. GAP was partially supported by the EPSRC through grant numbers EP/P031587/1, EP/L024926/1, and EP/L020564/1. AS acknowledges the hospitality of Imperial College London.}

\date{}

\begin{document}

\begin{abstract}
We study the McKean-Vlasov equation 
\[
\partial_t \varrho= \beta^{-1} \Delta \varrho + \kappa \dive (\varrho \nabla (W \star \varrho)) \, ,
\]
with periodic boundary conditions on the torus. We first study the global asymptotic stability of the homogeneous steady state. We then focus our attention on the stationary system, and prove the existence of nontrivial solutions branching from the homogeneous steady state, through possibly infinitely many bifurcations, under appropriate assumptions on the interaction potential. We also provide sufficient conditions for  the existence of continuous and discontinuous phase transitions. Finally, we showcase these results by applying them to several examples of
interaction potentials such as the noisy Kuramoto model for synchronisation, the Keller--Segel model for bacterial chemotaxis, and the noisy Hegselmann--Krausse model for opinion dynamics. 
\end{abstract}
\maketitle

{\small
\tableofcontents

}

\section{Introduction}\label{S:intro}
Systems of interacting particles arise in a myriad of applications ranging from opinion dynamics ~\cite{hegselmann2002opinion}, granular materials ~\cite{BCCP98,carrillo2003kinetic,BGG2013} and mathematical biology ~\cite{keller1971model,BCM} to statistical mechanics ~\cite{martzel2001mean}, galactic dynamics~\cite{binney2008}, droplet growth~\cite{conlon2017}, plasma physics~\cite{bittencourt1986fund}, and synchronisation ~\cite{kuramoto1981rhythms}. Apart from being of independent interest, these systems find applications in a diverse range of fields such as: particle methods in numerical analysis~\cite{delmoral2010intro}, consensus-based methods for global optimisation~\cite{carrillo2018analytical}, and nonlinear filtering~\cite{crisan1997}. They
have also been studied in the context of multiscale analysis~\cite{gomes2018mean}, in the presence of
memory-like effects and in a non-Markovian setting~\cite{duong2018mean}, and in the discrete setting of graphs~\cite{EFLS16}.

In this paper, we analyse the partial differential equation (PDE) associated to the system of interacting stochastic differential equations (SDEs) on $\T^d$, the torus of side length $L>0$, of the following form
\begin{align}
dX^i_t = -\frac{\kappa}{ N} \sum\limits_{i \neq j}^N\nabla W(X^i_t -X^j_t) \dx{t} + \sqrt{2 \beta^{-1}} dB^i_t \, ,
\label{eq:lang}
\end{align}
where the $X^i_t \in \T^d,i= 1\dots N $ represent the positions of the $N$ ``particles'', $W$ is a periodic
interaction potential, and the $B^i_t, i = 1 \dots N$ represent $N$ independent $\T^d$-valued Brownian motions. The constants $\kappa,\beta>0$ represent the strength of interaction and inverse temperature respectively. Since one of the two parameters is redundant, we keep $\beta$ fixed for the rest of the paper. It is clear that what we have described is a set of interacting overdamped Langevin equations. Based on the choice of $W(x)$, one can then obtain models for numerous phenomena from the physical, biological, and social sciences. We refer to~\cite{review,pareschi2013interacting,muntean2014collective,MT2014} and the references therein for a comprehensive list of such models.

Systems of interacting diffusions have been studied extensively. They were first analysed by McKean (cf.~\cite{mckean1966class,mckean1967propagation}) who noticed an interesting relation between them and a class of nonlinear parabolic partial differential equations. In particular, it is well known (cf.~\cite{oelschlager1984martingale,sznitman1991topics}) that for this class of SDEs one can pass to the so-called mean field limit: if we consider the empirical measure defined as follows
\begin{align}
\varrho^{(N)} : = \frac{1}{N} \sum\limits_{i=1}^N \delta_{X^i_t}, \quad\text{ with }\quad \operatorname{Law}(X_0:=(X^1_0 \cdots X^N_0 )) =\prod\limits_{i=1}^N \varrho_0(x_i) \, ,
\end{align}
then, provided that $W$ is smooth, as $N \to \infty$, $\mathbb{E}\bra{\varrho^{(N)}}$ converges in the sense of weak convergence of probability measures to some measure $\varrho$ satisfying the following nonlocal parabolic PDE
\begin{align}
\label{eq:mv1}
\begin{aligned}
\partial_t \varrho & = \beta^{-1} \Delta \varrho + \kappa \nabla \cdot (\varrho \nabla W \star \varrho )\,, \\
\varrho(x,0) &=\varrho_0(x)\, . 
\end{aligned}
\end{align}
 The above equation is commonly referred to as the McKean--Vlasov equation, the latter name stemming from the fact that it also arises as the overdamped limit of the Vlasov--Fokker--Planck equation. Equation~\eqref{eq:mv1}
can also be thought of as a nonlinear Fokker--Planck equation for the following nonlinear SDE, commonly referred to as the McKean SDE,
\begin{align}
dX_t = -\kappa(\nabla W \star \varrho)(X_t,t) \dx{t} + \sqrt{2 \beta^{-1}} dB_t \, ,
\end{align}
where $\varrho= \textrm{Law}(X_t)$. The PDE~\eqref{eq:mv1} itself has a very rich structure associated to it - we have the following free energy functional
\begin{align}\label{eq:freeenergyintro}
\crF_{\kappa}(\varrho) & = \beta ^{-1}\int\limits_{\T^d} \varrho \log \varrho \dx{x} + \frac{\kappa}{2}
\iint_{\T^d \times \T^d} W(x-y) \varrho(y) \varrho(x) \dx{y} \dx{x} = \beta^{-1}S(\varrho)  + \frac{\kappa}{2} \mathcal{E}(\varrho,\varrho) \ , 
\end{align}
where $S(\varrho)$ and $\cE(\varrho,\varrho)$ represent the entropy and interaction energy associated with $\varrho$ respectively.
It is well known, starting from the seminal work in~\cite{jordan1998variational,otto2001geometry}, that this equation belongs to a larger class of dissipative PDEs including the heat equation, the porous medium equation, and the aggregation equation, which can be written in the form
\begin{align}
\partial_t \varrho = \nabla \cdot \bra*{\varrho \nabla \frac{\delta \crF}{\delta \varrho}} \, ,
\end{align}
for some free energy $\crF$ and are gradient flows for the associated free energy functional with respect to the $d_2$ transportation distance defined on probability measures having finite second moment, see ~\cite{carrillo2003kinetic,villani2003topics}. We refer the reader to~\cite{ambrosio2008gradient,santam} for more information on the abstract theory of gradient flows in the space of probability measures.

Our goals are to study some aspects of the asymptotic behaviour and the stationary states of the McKean--Vlasov equation for a wide class of interaction potentials. In terms of the asymptotic behaviour, we analyse the stability conditions for the homogeneous steady state $1/L^d$ and the rate of convergence to equilbrium. We extend the $\Leb^2$-decay results of~\cite{chazelle2017well} to arbitrary dimensions and arbitrary sufficiently nice interactions and also provide sufficient conditions for convergence to equilibrium in relative entropy. 

The rest of the paper is devoted to the analysis of the properties of non-trivial stationary states of the Mckean--Vlasov system, i.e., nontrivial solutions of 
\begin{align}
 \beta^{-1} \Delta \varrho + \kappa \nabla \cdot (\varrho \nabla W \star \varrho )=0 \, . \label{eq:smv1}
\end{align}
Previous results in this direction include those by Tamura ~\cite{tamura1984asymptotic}, who provided some criteria for the existence of local bifurcations on the whole space by using tools from nonlinear functional analysis, in particular, the Crandall--Rabinowitz theorem. Unfortunately, his analysis depends crucially on the unphysical assumption that the interaction potential is an odd function. One of the main results of the present work is a complete, quantitative, local bifurcation analysis under physically realistic assumptions. Dawson ~\cite{dawson1983critical} studied for the first time the existence of nontrivial stationary states for a particular double-well confinement and Curie--Weiss interaction on the line. The existence of nontrivial stationary states or the bifurcation of nontrivial solutions from the homogeneous steady state is usually referred as phase transition in the literature. We also mention that more recently several authors ~\cite{tugaut2014phase,DFL,BCCD} looked at the existence of phase transitions in the whole space with different confinement and interactions. The most related work to us in the literature is due to Chayes and Panferov ~\cite{chayes2010mckean}, who studied the problem on the torus and provided some criteria for the existence of continuous and discontinuous phase transitions. 

In addition to presenting an existence and uniqueness theory for the evolution problem, we extend considerably the results of both~\cite{tamura1984asymptotic} and~\cite{chayes2010mckean}. We provide explicit criteria based on the Fourier coefficients of the interaction potential $W$ for the existence of local bifurcations by studying the implicit symmetry in the problem. In fact, we show that for carefully chosen potentials it is possible to have infinitely many bifurcation points. Additionally, we extend the results of~\cite{chayes2010mckean} and provide additional criteria for the existence of continuous and discontinuous phase transitions.

\subsection{Statement of main results}
We only state simplified versions of our results in one dimension, so as to avoid the use of notation that will be introduced later. We only need to define the cosine transform,
$\tilde{W}(k):= (2/L)^{1/2}\intom{W(x) \cos\bra*{\frac{2 \pi k}{L} x}}$ for $k \in \Z, k>0$. We work with
classical solutions of~\eqref{eq:mv1} which are constructed in~\cref{thm:wellp}.
\begin{thm}(Convergence to equilibrium) \label{thm:m1}
Let $\varrho$ be a classical solution of the Mckean--Vlasov equation~\eqref{eq:mv1} with smooth initial data and smooth, even, interaction potential $W$. Then we have:
\begin{tenumerate}
\item If $0 < \kappa< \frac{2 \pi}{3 \beta L \norm{\nabla W}_\infty}$, then $\norm*{\varrho(\cdot ,t)- \frac{1}{L}}_2 \to 0$, exponentially, as $t \to \infty$\label{thm:m1a},
\item If $\tilde{W}(k) \geq 0$ for all $k \in \Z, k>0$, or $0 < \kappa< \frac{2 \pi^2}{ \beta L^2 \norm{\Delta W}_\infty}$, then $
\cH\bra*{\varrho(\cdot ,t)|\frac{1}{L}} \to 0$, exponentially, as $t \to \infty$ \label{thm:m1b},
\end{tenumerate}
where $\cH\bra*{\varrho(\cdot ,t)|\frac{1}{L}}:= \intom{\varrho(\cdot ,t) \log \bra*{\frac{\varrho(\cdot ,t)}{\varrho_\infty}}}$ denotes the relative entropy.
\end{thm}
The previous theorem implies that the uniform state can fail to be the unique stationary solution only if the interaction potential has a negative Fourier mode, i.e., the interaction potential is not $H$-stable. Thus, the concept of $H$-stability introduced by Ruelle ~\cite{ruelle1999statistical} is relevant for the study of the stationary McKean--Vlasov equation as noticed in ~\cite{chayes2010mckean}. We have the following conditions for the existence of bifurcating branches of steady states.

\begin{thm}(Local bifurcations)\label{thm:m2}
Let $W$ be smooth and even and let $(1/L,\kappa)$ represent the trivial branch of solutions. Then every $k^* \in \Z, k^*>0$ such that
\begin{enumerate}
\item $\operatorname{card}\set*{k \in \Z, k>0 : \tilde{W}(k)=\tilde{W}(k^*)}=1$ ,
\item $\tilde{W}(k^*) <0$,
\end{enumerate}
leads to a bifurcation point $(1/L,\kappa_*)$ of the stationary McKean--Vlasov equation through the formula
   \begin{align}
     \kappa_*=-\frac{(2L)^{1/2}}{\beta \tilde{W}(k^*)} \, .
   \end{align}
\end{thm}
We are also able to sharpen sufficient conditions for the existence of continuous or discontinuous bifurcating branches. The following theorem is a simplified version of the exact statements that are presented in~\cref{thm:dctp} and~\cref{thm:spgap}.
\begin{thm}(Discontinuous and continuous phase transitions)\label{thm:m3}
 Let $W$ be smooth and even and assume the free energy $\crF_{\kappa,\beta}$ defined in~\eqref{eq:freeenergyintro} exhibits a transition point, $\kappa_c<\infty$, in the sense of~\cref{defn:tp}. Then we have the following two scenarios:
\begin{tenumerate}
\item If there exist strictly positive $k^a,k^b,k^c \in \Z$ with $\tilde{W}(k^a)\approx \tilde{W}(k^b)\approx \tilde{W}(k^c)\approx\min_k\tilde{W}(k)<0$ such that 
$k^a=k^b +k^c$, then $\kappa_c$ is a discontinuous transition point.\label{thm:m3a}
\item Let $k^\sharp = \argmin_k \tilde{W}(k)$ be uniquely defined with $\tilde{W}(k^\sharp)<0$ and $\kappa_\sharp=\sqrt{2L}/(\beta \tilde{W}(k^\sharp))$. Let $W_\alpha$ denote the potential obtained by multiplying all the negative  Fourier modes $\tilde{W}(k)$ except $\tilde{W}(k^\sharp)$ by some $\alpha \in(0,1]$. Then if $\alpha$ is made small enough, the transition point $\kappa_c$ is continuous and $\kappa_c=\kappa_\sharp$.\label{thm:m3b}
\end{tenumerate}
\end{thm}

 The proof of the above theorem relies mainly on~\cref{prop:CharactTP} which states that if $\varrho_\infty$ is the unique minimiser of the free energy $\crF_\kappa$ at $\kappa=\kappa_\sharp$ then $\kappa_c=\kappa_\sharp$ is a continuous transition point; on the other hand if $\varrho_\infty$ is not the global minimiser of $\crF_\kappa$ at $\kappa=\kappa_\sharp$, then
 $\kappa_c<\kappa_\sharp$ and $\kappa_c$ is a discontinuous transition point. 
 
 We conclude the introduction with a figure to provide the reader with some more intuition about the spectral signature of continuous and discontinuous phase transitions. As it can be seen in Figure~\ref{fig:dcctp}, the results of \cref{thm:m3} essentially apply to two perturbative regimes. Figure~\ref{fig:dcctp}(a) shows the scenario for the existence of a discontinuous transition point in which there are multiple resonating/near-resonating dominant modes $k^a,k^b,k^c$ which satisfy the algebraic condition $k^a=k^b+k^c$ from~\cref{thm:m3}\ref{thm:m3a}. This condition allows us to construct a competitor state at $\kappa=\kappa_\sharp$ which has a lower value of $\crF_\kappa$ than $\varrho_\infty$ by controlling the sign of the higher order terms in the Taylor expansion of the free energy. The statement~\cref{thm:m3}\ref{thm:m3a} is then a direct consequence of~\cref{prop:CharactTP}.
 
 Figure~\ref{fig:dcctp}(b) shows the scenario in which there is one dominant negative mode and all other negative modes are restricted to a small neighbourhood of $0$. In this case, there exists a continuous transition point. The proof follows by showing that $\varrho_\infty$ is the unique minimiser of $\crF_\kappa$ at $\kappa=\kappa_\sharp$. For controlling the involved error terms, the neighbourhood needs to made by small, which is equivalent to making 
 $\alpha$ small in the statement of ~\cref{thm:m3}\ref{thm:m3b}. As it will become clear in~\autoref{S:thermodynamic}, the condition in~\cref{thm:m3}\ref{thm:m3b} is essentially an assumption on the size of the spectral gap of the linearised McKean--Vlasov operator. Again, applying~\cref{prop:CharactTP}, the result follows.
\begin{figure}[ht]
\centering
\begin{minipage}[c]{0.45\textwidth}
\centering
    \includegraphics[width=\linewidth]{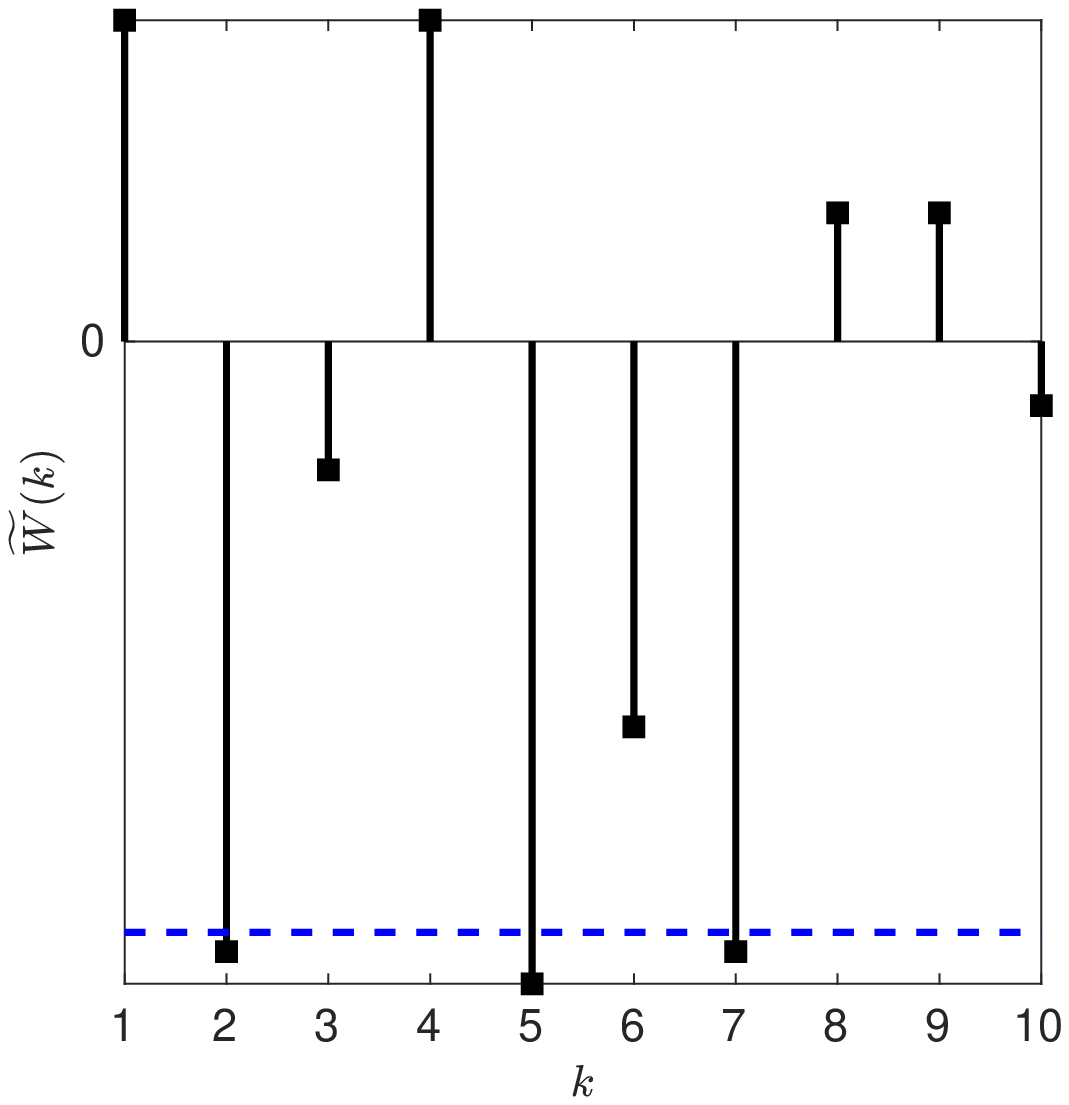}
     \caption*{(a)}
\end{minipage}
\begin{minipage}[c]{0.45\textwidth}
\centering
    \includegraphics[width=\linewidth]{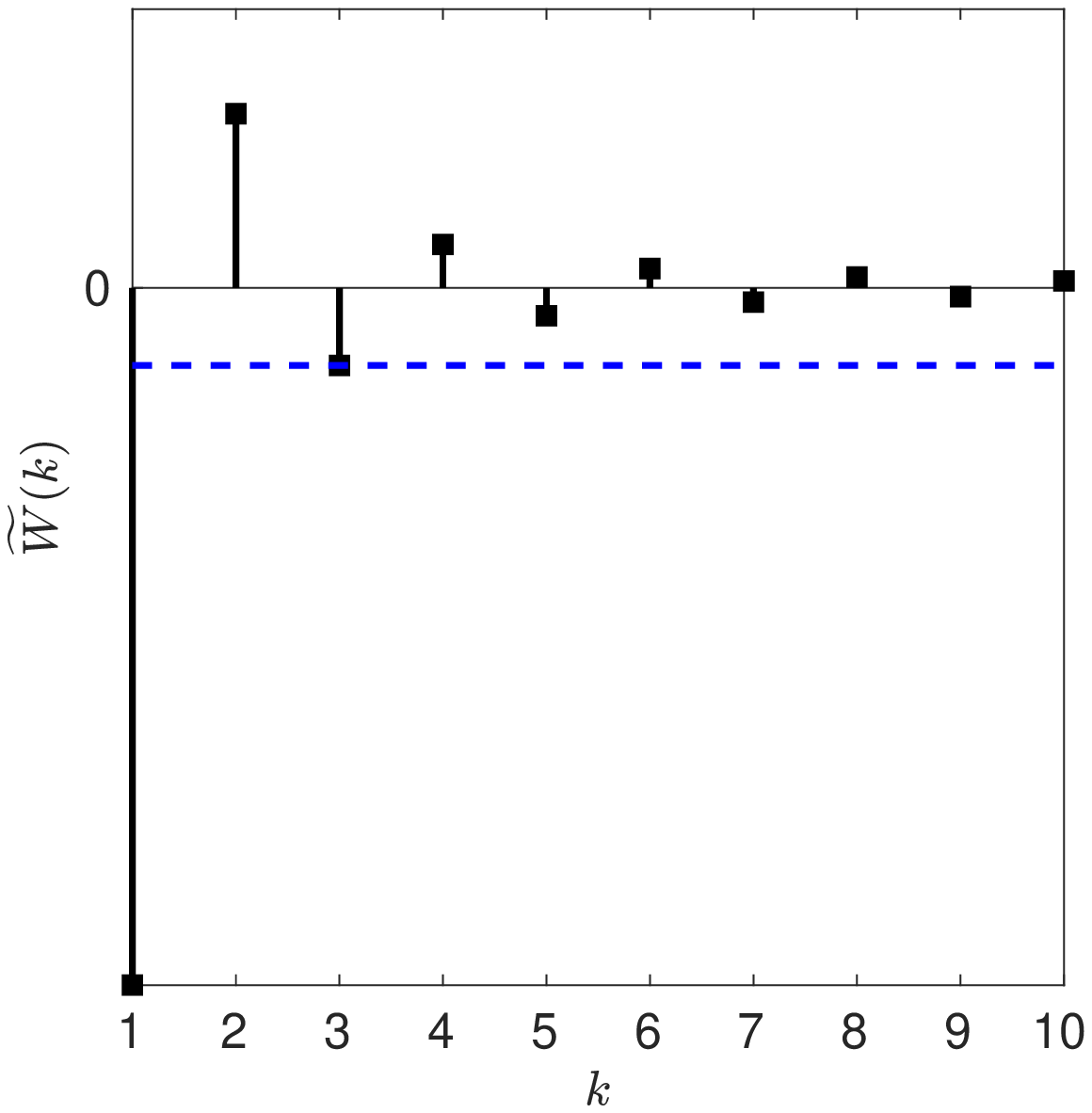}
    \caption*{(b)}
\end{minipage}
\caption{(a) The near-resonating modes scenario, in which the modes $k^a=7,k^b=5,k^c=2$ satisfy the algebraic condition $k^a=k^b+k^c$}; (b) The dominant mode scenario.
\label{fig:dcctp}
\end{figure}

This work provides a complete local and global bifurcation analysis for the Mckean--Vlasov equation on the torus.
This enables us to study phase transitions for several important models that have been introduced in the literature.
This is done in~\autoref{S:app}. In particular, we apply our results to the following examples: the noisy Kuramoto model for synchronisation, the Hegselmann--Krausse model for opinion dynamics, the Keller--Segel model for bacterial chemotaxis, the Onsager model for liquid crystal alignment, and the Barr\'e--Degond--Zatorska model for
interacting dynamical networks. As an example of the typical bifurcation diagram expected for this kind of system, 
we discuss the noisy Kuramoto model which has the interaction potential $W(x)=-(2/L)^{1/2}\cos(2 \pi x/L)$. For $\kappa$ sufficiently small,  the uniform state is the unique stationary solution. At some critical $\kappa=\kappa_c$ a clustered solution branches out from the uniform state and for all $\kappa>\kappa_c$ this clustered state is preferred solution, i.e., it is the global minimiser of the free energy, $\crF_{\kappa}$. The bifurcation diagram and a plot of the clustered solution can be seen in Figure~\ref{fig:kurbif}. The model is discussed in more detail in~\autoref{ss:kura}.
\begin{figure}[ht]
\centering
\begin{minipage}[c]{0.45\textwidth}
\centering
    \includegraphics[width=\linewidth]{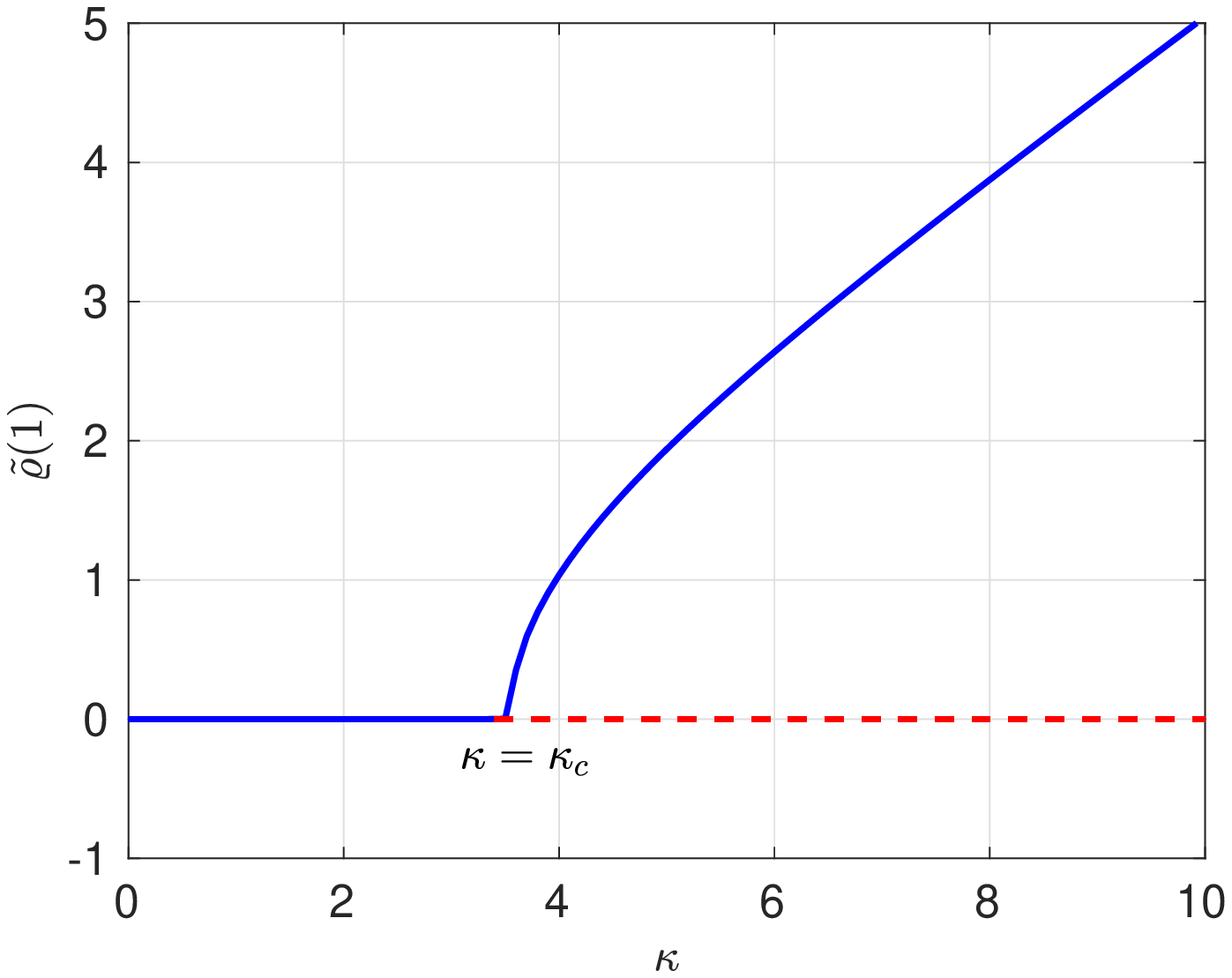}
    \caption*{(a)}
\end{minipage}
\begin{minipage}[c]{0.45\textwidth}
\centering
    \includegraphics[width=\linewidth]{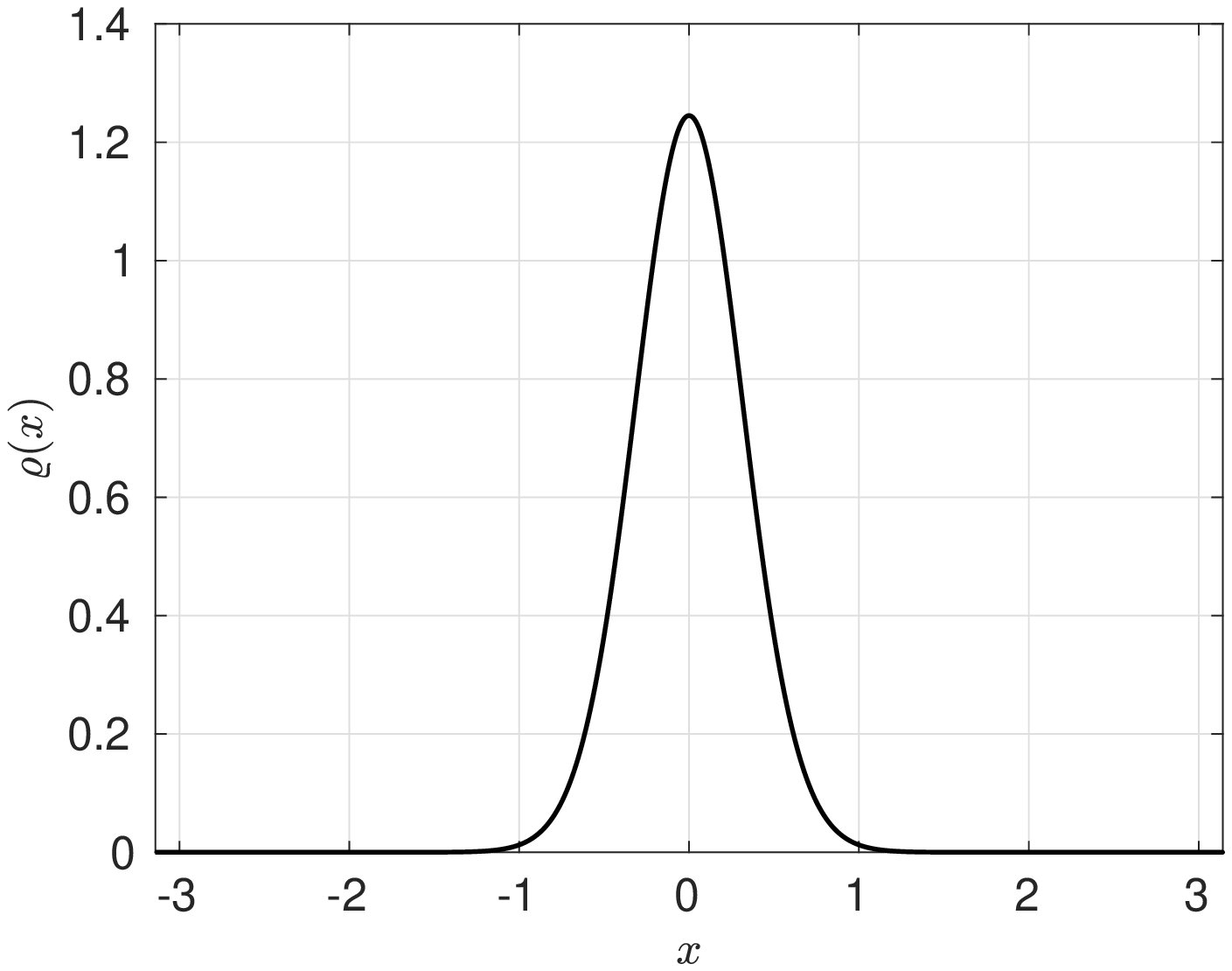}
     \caption*{(b)}
\end{minipage}
\caption{(a). The bifurcation diagram for the noisy Kuramoto system: the solid blue line denotes the stable branch
of solutions while the dotted red line denotes the unstable branch of solutions (b). An example of a clustered solution representing phase synchronisation of the oscillators}
\label{fig:kurbif}
\end{figure}
\subsection{Organisation of the paper}
The paper is organised in the following manner: In Section~\ref{S:not} we introduce the main notation and assumptions on the interaction potential $W$, state a basic existence and uniqueness theorem for classical solutions of the evolutionary problem and present a series of results about the
stationary problem and the associated free energy that we use for our later analysis. In Section~\ref{S:gas} we present the proof of~\cref{thm:m1}\ref{thm:m1b}, whereas the proof of ~\cref{thm:m1}\ref{thm:m1a} is similar to the argument in~\cite{chazelle2017well} and can be found in Version 1 of the arXiv manuscript. Additionally, we perform a linear stability analysis of the Mckean--Vlasov PDE about $1/L^d$. Section~\ref{S:lbt} is dedicated mainly to the the proof of~\cref{thm:m2}, including further details about the structure of the bifurcating branches and the structure of the global bifurcation diagram. In Section~\ref{S:thermodynamic} we give sufficient conditions for the existence of continuous and discontinuous phase transitions and we present the proofs of~\cref{thm:m3}\ref{thm:m3a} and~\cref{thm:m3}\ref{thm:m3b}, along with some supplementary results. In Section~\ref{S:app}, we apply our results to various models from the biological, physical and social sciences.
%%------------------------------------------------------------------------------------------------------%%
%%------------------------------------------------------------------------------------------------------%%

\section{Preliminaries}\label{S:not}
\subsection{Set up and notation}
Let $U=\R^{d}/ L \Z^d \hat{=} \left(-\frac{L}{2},\frac{L}{2}\right)^d \subset \R^d$ be the torus of size~$L>0$. We denote by $\N = \set{0,1,\dots}$ the nonnegative integers. Furthermore, we will denote by $\Pro(U)$ the space of Borel probability measures on $U$, by $\Prac(U)$ the subset of $\Pro(U)$ absolutely continuous w.r.t the Lebesgue measure, and by $\Prac^+(U)$ the subset of $\Prac(U)$ having strictly positive densities a.e. Additionally, $C^k(U)$ will denote the restriction to $U$ of all $L$-periodic and $k$-times continuously differentiable functions, $\cD(U)$ the space of test functions, and $\skp*{f,g}_\mu$ the $\Leb^2(U,\mu)$ inner product.

\subsection{Assumptions on \texorpdfstring{$W$}{W}}Throughout the subsequent discussion we will assume that $W(x)$ is at least integrable and coordinate-wise even, that is 
\begin{equation}
 \forall x\in \R^d  \ \forall i \in \set{1,\dots, d} : \qquad W(x_1, \dots ,x_i, \dots, x_d)=W(x_1, \dots ,-x_i, \dots, x_d) \, .
\end{equation}
For the evolutionary problem we will assume
\begin{equation}\label{ass:A}
  W \in \SobW^{2,\infty}(U) \, ,\tag{\textbf{A1}}
\end{equation}
while for the stationary problem we will assume
\begin{equation}\label{ass:B}
  W \in \SobH^1(U) \implies W \in \Leb^1(U) \qquad\text{and}\qquad W_- \in \Leb^\infty(U) \quad\text{with}\quad W_-(x)=\min \{0,W(x)\}  \tag{\textbf{A2}} \, ,
\end{equation}
where the $\Leb^p(U)$ with $1 \leq p \leq \infty$ represent the Lebesgue spaces and $\SobW^{k,p}(U)$ represent the periodic Sobolev spaces with $\SobH^k(U) = \SobW^{k,2}(U)$. Wherever required, weaker or stronger assumptions will be indicated in the text. As one may expect, the assumptions on $W(x)$ for the evolutionary and stationary problems to be the same, it is important to mention that these assumptions are in no way sharp and the aim of this paper is not to study low regularity theory for this class of PDEs.

For the space $\Leb^2(U)$ we define the orthonormal basis, $\{w_k\}_{k \in \Z^d}$, $k=(k_1 , k_2  , \ldots , k_d)$, as follows:
\begin{align}\label{e:def:wk}
w_k(x)= N_k\prod\limits_{i=1}^d w_{k_i}(x_i), \qquad\text{ where } \qquad
w_{k_i}(x_i)=
  \begin{cases}
    \cos\left(\frac{2 \pi k_i}{L} x_i\right) &  k_i>0, \\
    1 & k_i=0,  \\
    \sin\left(\frac{2 \pi k_i}{L} x_i\right)  & k_i<0, \\
  \end{cases}
\end{align}
and $N_k$ is defined as
\begin{align}
N_k:=\frac{1}{L^{d/2}}\prod\limits_{i=1}^d \left(2-\delta_{k_i , 0} \right)^{\frac{1}{2}}=: \frac{\Theta(k)}{L^{d/2}} \, ,\label{e:def:thetak}
\end{align}
where $\delta_{i,j}$ denotes the Kronecker delta. We then have the following form for the discrete Fourier transform of any $f \in \Leb^2(U)$
\begin{align}
\tilde{f}(k)&= \skp{f,w_k }, \qquad k \in \Z^d\, .
\end{align}
We denote by ``$\star$'' the convolution of any two functions, $f(x),g \in \Leb^2(U)$ and for $f(x)=W(x)$ we have the following representation in Fourier space:
\begin{align}
(W \star g)(y)= \sum\limits_{k \in \N^d}\tilde{W}(k) \frac{1}{N_k} \sum\limits_{\sigma \in \Sym(\Lambda)}\tilde{g}(\sigma(k))w_{\sigma(k)}(y)  \, .
\end{align}
 Here, we have used the fact that $W(x)$ is coordinate-wise even. $\Sym(\Lambda)$ represents the symmetry group of the product of two-point spaces $\Lambda=\{1,-1\}^d$, which acts on $\Z^d$ by pointwise multiplication, i.e., $(\sigma(k))_i=\sigma_i k_i, k \in \Z^d, \sigma \in \Sym(\Lambda)$. Another expression that we will use extensively in the sequel is the Fourier expansion
of the following bilinear form
\begin{align}\label{Fourier:Interaction}
\iint\limits_{U \times U} \! W(x-y)g(x) g(y) \dx{x} \dx{y} = \sum\limits_{k \in \N^d} \tilde{W}(k)\frac{1}{N_k}\sum\limits_{\sigma \in \Sym(\Lambda)}|\tilde{g}(\sigma(k))|^2 \, .
\end{align}
It will be useful to note that for any function $g(x)$ and $k \in \Z^d$ the sum $\sum_{\sigma \in \Sym(\Lambda)}|\tilde{g}(\sigma(k))|^2 $ is translation invariant, i.e., the value of the sum is the same for $g$ and $g_\tau(x) =g(x+\tau)$ for $\tau\in U$. In later sections we will also use the space $\Leb^2_s(U) \subset \Leb^2(U)$, which we define as the space of coordinate-wise even functions in $\Leb^2(U)$ given by 
\begin{equation}\label{eq:def:L2s}
  \Leb^2_s(U) = \set*{ f\in \Leb^2(U) : f(x_1,\dots,x_i,\dots,x_d)= f(x_1,\dots,-x_i,\dots, x_d) , i\in \set{1,\dots,d} , x\in U } \, . 
\end{equation}
It should be noted that any pointwise properties (like being coordinate-wise even) should be understood in a pointwise a.e. sense. The space $\Leb^2_s(U)$ is a closed subspace of $\Leb^2(U)$ and thus is a Hilbert space in its own right. It is also  easy to check that $\{w_k\}_{k \in \N^d} \subset
\{w_k\}_{k \in \Z^d}$ forms an orthonormal basis for $\Leb^2_s(U)$. If $g$ is assumed to be in $\Leb^2_s(U)$, then the above expressions reduce to,
\begin{align}
(W \star g)(y)&= \sum\limits_{k \in \N^d,k_i>0} \tilde{W}(k) \frac{1}{N_k} \tilde{g}(k)w_k(y)  \, ,\\
\iint\limits_{U \times U} \! W(x-y)g(x) g(y) \dx{x} \dx{y} &= \sum\limits_{k \in \N^d,k_i>0} \tilde{W}(k) \frac{1}{N_k} |\tilde{g}(k)|^2 \, .
\end{align}
In addition, the sign of the individual Fourier modes of $W$ is quite important in the subsequent analysis and we introduce the following definition.
\begin{defn}[$H$-stability]\label{def:Hstable}
  A function $W \in \Leb^2(U)$ is said to be $H$-stable, denoted by $W \in \HH_\stable$, if it has non-negative Fourier coefficients, i.e,
  \begin{align}
    \tilde{W}(k) \geq 0, \quad \forall k \in \mathbb{Z}^d  \ , 
  \end{align}
  where, $\tilde{W}_k=\skp{W,w_k}$. This is by~\eqref{Fourier:Interaction} equivalent to the condition that,
  \begin{align}
    \iint\limits_{U \times U} W (x-y) \eta(x) \eta(y) \dx{x} \dx{y} \geq 0, \qquad \forall \eta \in \Leb^2(U).
  \end{align}
  Thus every potential is decomposed into two parts $W(x)=W_\stable(x) + W_\unstable(x)$, where
  \begin{align}
        W_\stable(x)&= \sum\limits_{k \in \Z^d} \bra*{\skp{W,w_k}}_+  w_k(x)\qquad\text{and}\qquad  W_\unstable(x)=W(x)-W_\stable(x) \, .
  \end{align}
  Hereby, $(a)_+ = \max\set{0, a}$ (resp. $(a)_-= \min\set{0, a}$) denotes the positive (resp. negative) part for a real number $a\in \R$. We will denote a potential $W \in \Leb^2(U)$ which is not $H$-stable by $W \in \HH_s^c$.
 \end{defn}
 An immediate consequence of the identity~\eqref{Fourier:Interaction} is that $H$-stable potentials have nonnegative interaction energy. The above definition can be thought of as a continuous analogue of the notion of $H$-stability encountered in the  study of discrete systems (cf. \cite{ruelle1999statistical}). We refer to~\cite{canizo2015existence} for an example of the notion of $H$-stability applied to continuous systems. For the rest of
 the paper we will drop the subscript~$U$ under the integral sign and all integrals in space will be taken over $U$ unless specified otherwise. 

\subsection{Existence and uniqueness for the dynamics}
We present an existence and uniqueness result for the McKean--Vlasov equation and comment on the nontrivial parts of the proof. The proof is quite standard. Our result  is an extension of~\cite[Theorem 4.5]{chazelle2017well} since we consider all potentials $W$ satisfying Assumption ~\eqref{ass:A} in any dimension $d$, as opposed to~\cite[Theorem 4.5]{chazelle2017well} which deals with the  Hegselmann--Krause potential in one dimension. Additionally, we prove \emph{strict} positivity of solutions as opposed to the nonnegativity proved in~\cite{chazelle2017well}. We prove below the existence of classical solutions $\varrho(\cdot,t) \in C^2 (U)$ to the system
\begin{align} 
\label{eq:PDE}
\begin{alignedat}{3}
\frac{\partial \varrho}{\partial t}& = \beta^{-1}\Delta \varrho + \kappa \dive (\varrho\nabla W \star \varrho), \qquad &&(x,t) \in U \times (0,T] \ ; \\
\varrho(x,0) &=\varrho_0(x), \qquad && x \in U  \, . 
\end{alignedat}
\end{align}
\begin{thm}
Assume Assumption~\eqref{ass:A} holds, then for $\varrho_0  \in \SobH^{3+d}(U)\cap\Prac(U) $, there exists a unique classical solution $\varrho$ of \eqref{eq:PDE} such that $\varrho(\cdot,t) \in \Prac(U)\cap C^2(U)$ for all $t>0$. Additionally, $\varrho(\cdot,t)$ is strictly positive and has finite entropy, i.e,    $\varrho(\cdot ,t)>0$  and $S(\varrho(\cdot,t))< \infty$, for all $t>0$.
  \label{thm:wellp}
\end{thm}
The strategy of the proof is identical to that used in the proof of~\cite[Theorem 4.5]{chazelle2017well}. We construct a sequence of linear problems that approximate the McKean--Vlasov equation
\begin{align}
\frac{\partial \varrho_n}{\partial t} = \beta^{-1}\Delta \varrho_n + \kappa \dive (\varrho_n\nabla W \star \varrho_{n-1}) \quad &\textrm{ in } U \times (0,T] \ , \nonumber \\
\forall i \in \set{1,\dots, d} : \qquad \varrho_n(\cdot + L \mathbf{e}_i)=\varrho_n(\cdot) \quad &\textrm{ on } \partial U \times [0,T]\ , 
  \label{eq:nPDE} \\\nonumber
  \varrho=\varrho' \quad &\textrm{ in } U \times \{0\} \, ,
\end{align}
 which for smooth initial data, $\varrho'\in \Prac(U) \cap C^\infty(U)$ have unique smooth solutions. Similar apriori estimates to \cite{chazelle2017well}
 obtained using the $\SobW^{2,\infty}(U)$-regularity of $W$ allows us to 
 pass to the limit as $n \to \infty$ and recover weak solutions of the 
 McKean--Vlasov equation which are proved to be unique. Their regularity follows from bootstrapping and using the regularity of $W$ and the initial data.

 We now comment on the proof of strict positivity for classical solutions $\varrho(x,t)$ of~\eqref{eq:PDE}. The nonnegativity of the solutions
 follows from a similar argument to~\cite[Corollary 2.2]{chazelle2017well}. Consider now the ``frozen'' linearised version of the McKean--Vlasov equation, i.e.,
\begin{align}
\frac{\partial \vartheta}{\partial t}= \dive \left(\beta^{-1}\nabla \vartheta + \kappa \vartheta (\nabla W \star \varrho(x,t))\right)  
\end{align}
This is a linear parabolic PDE with uniformly bounded and continuous coefficients. Additionally, $\varrho(x,t)$ is a classical solution
to this PDE. Thus we have a Harnack's inequality of the following form (cf. 
~\cite[Theorems 8.1.1-8.1.3]{bogachev2015fokker} for sharp versions of this result)
\begin{align}
\sup_U \varrho(x,t_1) < C \inf_U \varrho(x,t_2) \ ,
\end{align}
for $0<t_1<t_2< \infty$ for some positive constant $C$.
 Since $\varrho(x,t)$ is nonnegative and $\norm{\varrho(x,t)}_{1}=1$ for all $0\leq t < \infty$, this implies that $\inf_U\varrho(x,t)$ is positive for any positive time. The fact that the entropy is finite follows from the fact that $\varrho(x,t)$ is positive and bounded above.

\subsection{Characterisation of the stationary solutions}\label{S:stationary}
In subsequent sections we will study the stationary solutions of the McKean--Vlasov equation \eqref{eq:PDE}, i.e, classical solutions $\varrho \in C^2 (U)$ of
\begin{align} 
\dive \left(\beta^{-1}\nabla\varrho + \kappa \varrho \nabla W \star \varrho \right)=0, \quad & x \in U \label{eq:sPDE} \, .
\end{align}
In this subsection we present standard results about the stationary McKean--Vlasov equation that will be useful for our later analysis. The main results in this section are~\cref{thm:wellpsPDE} which discusses the existence of solutions and their regularity,~\cref{prop:tfae} which connects stationary solutions to minimisers of the free energy, and~\cref{thm:dirmet} which discusses the existence of minimisers for the free energy.

 We start by discussing the existence and and regularity question for the stationary problem. The proof relies on the link between the stationary PDE and the fixed points of a nonlinear map as was discussed in~\cite{tamura1984asymptotic} and~\cite{dressler1987stationary}. 

 \begin{thm}[Existence, regularity, and strict positivity of solutions for the stationary problem] \label{thm:wellpsPDE} $ $\\
Consider the stationary McKean--Vlasov PDE  \eqref{eq:sPDE} such that Assumption \eqref{ass:B} holds. Then we have 
\begin{tenumerate}
\item There exists a weak solution, $\varrho \in \SobH^1(U) \cap \Prac(U)$ of \eqref{eq:sPDE} and any weak solution is a fixed point of the nonlinear map $\cT: \Prac(U) \to \Prac(U)$
\begin{align}\label{def:T}
\cT\!\varrho=\frac{1}{Z(\varrho,\kappa,\beta)}e^{-\beta \kappa W \star \varrho}, \quad\text{where}\quad Z(\varrho,\kappa,\beta)= \intom{e^{- \beta \kappa W \star \varrho}}\,.
\end{align}
\label{thm:ex}
\item Any weak solution $\varrho \in \SobH^1(U) \cap \Prac(U)$ is smooth and strictly positive ,i.e., $\varrho\in C^\infty(\bar{U})\cap
\Prac^+(U)$\,. \label{thm:re}
\end{tenumerate}
\end{thm}
% \begin{rem}
%  The proof of existence is, in fact, trivial because the uniform state $\varrho_\infty$ is always a solution of~\eqref{eq:wsPDE} However,  we include the proof based on this fixed point argument below as it can easily be generalised to the case $U=\R^d$, with a nonconvex confining potential, which lacks a candidate for the trivial state for a general interaction potential. Moreover the fixed point argument directly establishes the positivity of stationary states.
% \end{rem}
\begin{proof}
The weak formulation of \eqref{eq:sPDE} is
\begin{align}
-\beta^{-1}\intom{\nabla \varphi \cdot \nabla \varrho } -\kappa \intom{\varrho \nabla \varphi \cdot \nabla W \star \varrho}=0,
\quad \forall \varphi \in \SobH^1(U) \, ,
\label{eq:wsPDE}
\end{align}
where we look for solutions $\varrho \in \SobH^1(U) \cap \Prac(U)$. 
We have the following estimate on the map $\cT$ from~\eqref{def:T}
\begin{align}
\norm*{\cT\!\varrho}_2^2 \leq \norm*{\cT\!\varrho}_\infty \leq e^{\beta \kappa (\norm*{W_-}_\infty + \norm*{W}_1)} \, . \label{eq:linfT}
\end{align}
Thus it makes sense to search for fixed points of this equation in the set  $E:=\{\varrho \in \Leb^2(U) \cap \Prac(U):
\norm*{\varrho}_2^2 \leq e^{\beta \kappa (\norm*{W_-}_\infty + \norm*{W}_1)}\}$ as all fixed points must be in this set. It is easy to check that $E$ is a closed, convex subset of $\Leb^2(U)$. We can now redefine $\cT$ to act on $E$. Additionally, for any $\varrho \in E$, we have that
\begin{align}
\norm*{\cT\!\varrho}_{\SobH^1}^2 &= \norm*{\cT\!\varrho}_2^2 + \norm*{\nabla \cT\!\varrho}_2^2 
\leq \norm*{\cT\!\varrho}_\infty\left(1 + L^d\beta^2 \kappa^2 \norm*{\cT\!\varrho}_\infty^2\norm*{\nabla W}_2^2\right) \ ,\label{eq:H1T}
\end{align}
where we have used the fact that $W \in \SobH^1(U)$. Thus  using~\eqref{eq:linfT}, we have that $\cT(E) \subset E$ is uniformly bounded in $\SobH^1(U)$. By Rellich's compactness theorem, this implies that $\cT(E)$ is relatively compact in $\Leb^2(U)$, and therefore in $E$, since $E$ is closed. Furthermore, $\cT$ is Lipschitz continuous, i.e., we have for $\varrho_1,\varrho_2 \in E:$ $\norm*{\cT\!\varrho_1 -\cT\!\varrho_2}_2 \leq C\norm*{\varrho_1-\varrho_2}_2$, for some positive constant $C$. By the Schauder fixed point theorem, there exists a fixed point of $\varrho \in E$ of $\cT$ which by \eqref{eq:H1T} is in  $\SobH^1(U)$. Plugging this expression into the weak form of the PDE~\eqref{eq:wsPDE} we obtain~\ref{thm:ex}. Also note that fixed points of $\cT$ are bounded from below by $e^{-\beta \kappa (\norm*{W_-}_\infty + \norm*{W}_1 \norm*{\cT\!\varrho}_\infty )}$, proving the positivity of them.

Before proceeding to the proof of~\ref{thm:re}, we argue that every weak solution in $\SobH^1(U)\cap \Prac(U)$  is a fixed point of the  nonlinear map, $\cT$.
Consider the ``frozen'' version of the weak form in \eqref{eq:wsPDE},
\begin{align}
-\beta^{-1}\intom{\nabla \varphi \cdot \nabla \vartheta } -\kappa \intom{\vartheta \nabla \varphi \cdot \nabla W \star \varrho}=0 \ , \quad \forall \varphi \in \SobH^1(U) \, , \label{eq:fwf}
\end{align}
where $\varrho \in \SobH^1(U)\cap\Prac(U)$ is a weak solution of \eqref{eq:sPDE} and $\vartheta$ is the unknown function. The above equation is the weak form of a uniformly elliptic PDE whose associated bilinear form is coercive in the weighted space, $\SobH^1_0(U, \cT\!\varrho)$ where $\SobH^1_0(U)=\SobH^1(U) /\ \R$.
%(after transforming to the corresponding stationary backward Kolmogorov equation).
To see this, set $\vartheta(x)=h(x) \cT\!\varrho$. We then obtain the following integral formulation of the transformed PDE,
\begin{align}
-\beta^{-1}\intom{\nabla \varphi \cdot \nabla h \; \cT\!\varrho }=0 \ , \qquad \forall \varphi \in \SobH^1(U) \, .
\end{align}
Let $h_1$ and $h_2$ be two weak solutions of the above equation. By choosing $\varphi =h_1 -h_2 = h$, we obtain a unique weak solution to~\eqref{eq:fwf} up to normalisation. Here, we also used that $\cT\!\varrho$ has full support, since it is bounded from below.
% \begin{align}
% -\beta^{-1}\intom{|\nabla (h_1 -h_2)|^2 \; \cT\!\varrho }=0 \,.
% \end{align}
% Since $\cT\!\varrho$ by the lower boundedness has full support, the weak solution to~\eqref{eq:fwf} is unique up to normalisation. 
Hence, if it is chosen to be a probability measure, it is unique. We observe that $\vartheta=\cT\!\varrho$ is such a weak solution, as is $\varrho$. This implies that any weak solution must be such that $\varrho=\cT\!\varrho$. 

We obtain regularity of solutions by observing that if $f \in \SobH^m(U), g \in \SobH^n(U)$,
then we have that $f \star g \in \SobW^{m+n,\infty}(U)$. Then we use a bootstrap
argument, i.e., $W \in \SobH^1(U), \varrho \in \SobH^1(U)$ implies that $\varrho
=\cT\!\varrho \in \SobW^{2,\infty}(U)$. This implies that $W \star \varrho \in \SobW^{3,\infty}(U)$
and so on and and so forth. Thus we have that $\varrho \in \SobH^m(U) \cup \SobW^{m,\infty}(U)$ for
any $m \in \N$.  The strict positivity follows from the lower bound on $\cT\!\varrho$.
\end{proof}
We already know that associated with this PDE we have a free energy functional $\crF_{\kappa}: \Prac^+(U) \to \R$ defined on the space $\Prac^+(U)$ of strictly positive absolutely continuous probability measures on $U$ by
\begin{align}\label{eq:FreeEnergy}
  \crF_{\kappa}(\varrho) & = \beta ^{-1}\int \varrho \log \varrho \dx{x} + \frac{\kappa}{2}
\iint W(x-y) \varrho(y) \varrho(x) \dx{y} \dx{x} \\
& = S_{\beta}(\varrho)  + \frac{\kappa}{2} \mathcal{E}(\varrho,\varrho) \ . \nonumber
\end{align}
% where $\beta,\kappa>0$, and we take $W \in \SobW^{2,\infty}$. 
Since, we regard $\beta$ as a fixed parameter, we omit it in the subscript on $\cF_\kappa$. 

The free energy $\crF_{\kappa}$ is a Lyapunov function for the evolution and its negative derivative along the flow is given by the entropy dissipation functional $\mathcal{J}_\kappa: \Prac^+(U)  \to \R^+ \cup \set{+\infty}$ with

\begin{align}\label{e:def:dissipation}
\mathcal{J}_{\kappa}(\varrho)= 
\begin{cases}\intT{\abs*{\nabla \log \frac{\varrho}{\exp\bra{- \beta \kappa W \star \varrho}}}^2 \varrho} \ , & \varrho \in \Prac^+(U) \cap \SobH^1(U) \\
+ \infty \ , & \textrm{otherwise} \, . 
\end{cases}
\end{align}

% It is obtained by computing the time rate of change of the free energy along the flow of the PDE.  i.e., $- \frac{\dx{}}{\dx{t}}\crF_{\kappa}(\varrho(x,t)) $. 
This follows from rewriting~\eqref{eq:PDE} as $\partial_t \varrho = \dive \bra*{ \varrho \bra*{\beta^{-1} \nabla \log \varrho + \nabla W \star \varrho}}$ and differentiating the free energy functional along the flow
\begin{equation}
  \frac{\dx{}}{\dx{t}} \crF_{\kappa}(\varrho) =  \int \bra*{ \beta^{-1} \log \varrho + \kappa W \star \varrho} \partial_t \varrho \dx{x}  = - \int \abs*{ \beta^{-1} \nabla \log \varrho + \kappa \nabla W \star \varrho}^2 \varrho \dx{x} = -\cJ_{\kappa}(\varrho(t)) \leq 0. 
\end{equation}
Finally we have the Gibbs state map $F_{\kappa}:\Prac(U) \to \Prac(U)$. This equation encodes the stationary states as fixed points of the nonlinear mapping $\cT$ from~\eqref{def:T} 
\begin{align}\label{eq:fixedPoint}
F_{\kappa}(\varrho)=\varrho-\cT\!\varrho= \varrho -\frac{1}{Z(\varrho,\kappa,\beta)}e^{-\beta \kappa W \star \varrho} \, , \qquad\text{where}\qquad  Z(\varrho,\kappa,\beta)= \intT{e^{- \beta \kappa W \star \varrho}} \ .
\end{align}
The identification of stationary states \eqref{eq:sPDE}, critical points of $\crF_{\kappa}$ and $\mathcal{J}_\kappa$, and zeros of $F_{\kappa}$ is given by the following proposition.
\begin{prop} \label{prop:tfae}
Assume $W(x)$ satisfies Assumption~\eqref{ass:B} and fix $\kappa >0$. Let $\varrho \in \Prac^+(U)$. Then the following statements are equivalent:
\begin{enumerate}
  \item $\varrho$ is a classical solution of the stationary McKean--Vlasov equation \eqref{eq:sPDE}.
  \item $\varrho$ is a zero of the map $F_\kappa(\varrho)$.
  \item $\varrho$ is a critical point of the free energy $\crF_\kappa(\varrho)$.
  \item $\varrho$ is a global minimiser of the entropy dissipation functional $\cJ_\kappa (\varrho)$.
\end{enumerate}
\end{prop}
\begin{proof}\begin{itemize}[leftmargin=4.5em]
 \item[(1)$\Leftrightarrow$(2):] Observe that $\varrho$ is a zero of $F_\kappa(\varrho)$ if and only if it is a fixed point of $\cT$. Thus by part (a) of~\cref{thm:wellpsPDE} we have the desired equivalence. 
\item[(2)$\Rightarrow$(3):] The main observation for this is that
zeroes of $F_\kappa$ represent solutions of the Euler--Lagrange equations for
$\crF_\kappa$. Let $\varrho, \varrho_1 \in \Prac^+(U)$, we define the standard convex interpolant, $\varrho_s=(1-s)\varrho+s \varrho_1$, $s \in (0,1)$ such that $\crF(\varrho),\crF(\varrho_1)< \infty$. Then we have the following form of the Euler--Lagrange equations (which are well-defined for $\varrho,\varrho_1 \in \Prac^+(U)$),
\begin{align}
\left.\frac{\dx{}}{\dx{s}}\crF_\kappa(\varrho_s)\right\rvert_{s=0}=
\intT{\left(\beta^{-1}\log \varrho + \kappa W \star \varrho\right) \eta} =0 \ ,
\label{eq:eullag}
\end{align}
where $\eta = \varrho_1-\varrho$. Now if $\varrho$ is a zero of $F_\kappa$ it is easy to check that the above expression is zero
for any $\varrho_1 \in \Prac^+(U)$.
\item[(3)$\Rightarrow$(2):] On the other hand assume that $\varrho$ is a critical point. If the integrand $\beta^{-1}\log \varrho + \kappa W \star \varrho$ in \eqref{eq:eullag} is not constant a.e., we can find without loss of generality a set $A \in \mathcal{B}(U)$ of nonzero Lebesgue measure such that
\begin{align}
A := \left\{x \in U : \left(\beta^{-1}\log \varrho + \kappa W \star \varrho\right) > \int \bra*{\beta^{-1}\log \varrho + \kappa W \star \varrho} \varrho  \dx{y} \right\} \, .
\end{align}
We are now free to choose $\varrho_1 \in \Prac^+(U)$ to be
\begin{align}
\varrho_1 = \frac{1}{L^d}\bra[\big]{(1- \varepsilon) \chi_A(x) + \varepsilon \chi_A^c(x)} \ ,
\end{align}
for some $\varepsilon>0$. For this choice of $\varrho_1$, we have,
\begin{align}
\left.\frac{\dx{}}{\dx{s}}\crF_\kappa(\varrho_s)\right\rvert_{s=0}&=(1-\varepsilon) a + \varepsilon b \, ,  \\
\text{where } \qquad a &= \frac{1}{L^d} \int_{A}\bra*{ \left(\beta^{-1}\log \varrho + \kappa W \star \varrho\right) - \int \bra*{\beta^{-1}\log \varrho + \kappa W \star \varrho} \varrho  \dx{y}  }\dx{x} \, ,\\
\text{and }\qquad b &=\frac{1}{L^d} \int_{A^c}\bra*{ \left(\beta^{-1}\log \varrho + \kappa W \star \varrho\right) - \int \bra*{\beta^{-1}\log \varrho + \kappa W \star \varrho} \varrho  \dx{y}  }\dx{x} \, .
\end{align}
From our choice of the set $A$, it is clear that $a >0$ and $b \leq 0$. Since $\varepsilon$ can be made arbitrarily
small, $(1-\varepsilon)a + \varepsilon b$ can be made positive. Thus we have derived a contradiction since $\varrho$ is a critical point of $\crF_\kappa$ and therefore it must satisfy the Euler--Lagrange equations in \eqref{eq:eullag}. Thus the integrand must be constant a.e. from which we obtain (3)$\Rightarrow$(2). 

\item[(2)$\Rightarrow$(4):] Clearly, $\cJ_\kappa$ is nonnegative. Thus if $\cJ_\kappa(\varrho)=0$ for some $\varrho \in \Prac^+(U)$ then it is necessarily a global minimiser. Plugging in $\varrho$ for some zero of $F_\kappa$ finishes this implication.
\item[(4)$\Rightarrow$(2):] Now, any global minimiser $\varrho$ of
$\cJ_\kappa(\varrho)$ must satisfy $\cJ_\kappa(\varrho)=0$ since $\cJ_\kappa(\varrho_\infty)=0$. From the expression for $\cJ_\kappa(\varrho)$ and the fact that $\varrho$ has full support this is possible only if
\begin{align}
\nabla \frac{ \log \varrho}{e^{-\beta \kappa W \star \varrho}} =0, \quad \mathrm{a.e.} 
\end{align}
Thus, we have that, $\varrho - C e^{-\beta \kappa W \star \varrho}=0$, a.e., for some constant, $C>0$, which is given precisely by $Z(\varrho,\kappa,\beta)$ since $\varrho \in \Prac(U)$. Thus we have that $\varrho$ is a zero of $F_\kappa(\varrho)$ and the reverse implication, (4)$\Rightarrow$(2).
\qedhere
\end{itemize}
\end{proof}
% \begin{cor}
% Assume $W$ satisfies Assumption~\eqref{ass:B} and fix $\kappa \in (0,\infty)$. Let $\varrho \in \Prac^+(U) \cap \SobH^1(U)$. Then the following statements are equivalent:
% \begin{enumerate}
%   \item $\varrho$ is a classical solution of the stationary McKean--Vlasov equation~\eqref{eq:sPDE}.
%   \item $\varrho$ is a global minimiser of the entropy dissipation functional $\cJ_\kappa (\varrho)$.
% \end{enumerate}
% \end{cor}
% \begin{proof}\cre{
% \begin{itemize}[leftmargin=4.5em]
% \item[(1)$\Rightarrow$(2):] Clearly, $\cJ_\kappa$ is nonnegative. Thus if $\cJ_\kappa(\varrho)=0$ for some $\varrho \in \Prac^+(U)$ then it is necessarily a global minimiser. Plugging in $\varrho$ for some zero of $F_\kappa$ and using
% ~\cref{prop:tfae} finishes this implication.
% \item[(2)$\Rightarrow$(1):] Now, any global minimiser $\varrho$ of
% $\cJ_\kappa(\varrho)$ must satisfy $\cJ_\kappa(\varrho)=0$ since $\cJ_\kappa(\varrho_\infty)=0$. From the expression for $\cJ_\kappa(\varrho)$ and the fact that $\varrho$ has full support this is possible only if
% \begin{align}
% \nabla \frac{ \log \varrho}{e^{-\beta \kappa W \star \varrho}} =0, \quad \mathrm{a.e.} 
% \end{align}
% Thus, we have that, $\varrho - C e^{-\beta \kappa W \star \varrho}=0$, a.e., for some constant, $C>0$, which is given precisely by $Z(\varrho,\kappa,\beta)$ since $\varrho \in \Prac(U)$. Thus we have that $\varrho$ is a zero of $F_\kappa(\varrho)$ and by
% ~\cref{prop:tfae} the reverse implication, (2)$\Rightarrow$(1). \qedhere
% \end{itemize}}
% \end{proof}
The following lemma is taken from~\cite{chayes2010mckean} in which it is shown that for any unbounded $\varrho \in \Prac(U)$ there exists  a bounded $\varrho^\dagger \in\Prac(U)$ having a lower value of the free energy. 
%For the convenience of the reader, we include the entire proof.
\begin{lem}[\cite{chayes2010mckean}]
Assume  that $W$ satisfies Assumption~\eqref{ass:B} and  fix $\kappa \in (0, \infty)$. Then there exists  a positive constant $B_0<\infty$ such that for all $\varrho \in \Prac(U)$ with  $\norm*{\varrho}_{\infty} > B_0$  there  exists some $ \varrho^{\dagger} \in
\Prac(U)$ with $\norm*{\varrho^{\dagger}}_{\infty}\leq B_0 $ satisfying
\begin{align}
  \crF_{\kappa}(\varrho^{\dagger}) < \crF_{\kappa}(\varrho) \, .
\end{align}
\label{lem:1}
\end{lem}
The next lemma shows that minimisers of $\crF_\kappa(\varrho)$ over $\Prac(U)$ are attained in $\Prac^+(U)$.
\begin{lem}\label{lem:pos}
Assume $W(x)$ satisfies Assumption~\eqref{ass:B} and let $\varrho \in \Prac(U) \setminus \Prac^+(U)$. Then, there exists $\varrho^+ \in \Prac^+(U)$
such that,
\begin{align}
\crF_\kappa(\varrho^+) < \crF_\kappa(\varrho) \, .
\end{align}
\end{lem}
\begin{proof}
Let $\B_0:=\{ x \in U:\varrho(x)=0\}$, then from assumption $\varrho \notin \Prac^+(U)$, it follows $|\B_0|>0$. We define the competitor state
\begin{align}
\varrho_\epsilon(x)= \frac{1}{1 + \epsilon |\B_0|}\bra*{\varrho(x) + \epsilon \chi_{\B_0}(x)} \in \Prac^+(U) 
\end{align}
and show that for $\epsilon>0$ sufficiently small $\varrho_\epsilon$ has smaller free energy. We first compute its entropy
\begin{align}
S(\varrho_\epsilon)  %&=  \intom{\varrho_\epsilon \log \varrho_\epsilon} \\
&= \frac{1}{1 + \epsilon |\B_0|} \intom{\bra{\varrho + \epsilon \chi_{\B_0}} \log \bra{\varrho + \epsilon \chi_{\B_0}}} - 
\log(1 +\epsilon |\B_0|)  \\
&<  S(\varrho) -\frac{\epsilon |\B_0|}{1 +\epsilon |\B_0|} S(\varrho) + \frac{\epsilon |\B_0| \log \epsilon }{1 +\epsilon |\B_0|} <  S(\varrho) -\frac{\epsilon |\B_0|}{1 +\epsilon |\B_0|}\bra*{ S(\varrho_\infty) -  \log \epsilon } \, ,
\end{align}
where we have used the fact that $S(\varrho)> S(\varrho_\infty), \forall \varrho \in \Prac(U), \varrho \neq \varrho_\infty$.
For computing the interaction term, we use the fact that $\cE(\varrho,\varrho)> - \norm*{W_-}_\infty$ to estimate
\begin{align}
\frac{\kappa}{2} \cE(\varrho_\epsilon,\varrho_\epsilon) &=\frac{\kappa}{2}\iintom{W(x-y) \varrho_\epsilon(x)\varrho_\epsilon(y) } \\
& < \frac{\kappa}{2}\cE(\varrho,\varrho)  +  \frac{\kappa}{2}\bra*{\frac{1}{(1 + \epsilon |\B_0|)^2}-1} \cE(\varrho,\varrho)  + \kappa \norm*{W}_1 \frac{\epsilon}{(1 + \epsilon |\B_0|)^2} + \frac{\kappa}{2}\norm*{W}_1|\B_0| \frac{\epsilon^2}{(1 + \epsilon |\B_0|)^2} \\
& \leq \frac{\kappa}{2}\cE(\varrho,\varrho)  +\frac{\kappa}{2}\bra*{\frac{\epsilon |\B_0|(2 + \epsilon |\B_0|)}{(1 + \epsilon |\B_0|)^2}} \norm*{W_-}_\infty + \frac{\epsilon |\B_0|}{1 +\epsilon |\B_0|} C_1\\
&< \frac{\kappa}{2}\cE(\varrho,\varrho)  + \frac{\epsilon |\B_0|}{1 +\epsilon |\B_0|} (C_1 +  C_2)
\, ,
\end{align}
where $C_1,C_2<\infty$ depend on $W$ and $\B_0$ and we have chosen $\epsilon$ sufficiently small. Combining the two expressions together we obtain,
\begin{align}
\crF_\kappa(\varrho_\epsilon) < \crF_\kappa(\varrho) +\frac{\epsilon |\B_0| }{1 + \epsilon |\B_0|}\bra*{\beta^{-1} \log \epsilon -\beta^{-1} S(\varrho_\infty) + (C_1+C_2) \epsilon}  \, .
\end{align} 
Thus for $\epsilon$ sufficiently small but positive the logarithmic term will dominate and give us the required result.
\end{proof}
\begin{thm}[Existence of a minimiser~\cite{chayes2010mckean}]\label{thm:dirmet}
  Assume $W(x)$ satisfies Assumption~\eqref{ass:B}. For $\kappa \in (0,\infty)$ the free energy $\crF_{\kappa}(\varrho)$ has a smooth minimiser $\varrho_{\kappa} \in C^\infty(U) \cap \Prac^+(U)$.
  \label{thm:1}
\end{thm}
\begin{proof}
We start by noticing that we can control the entropy and interaction energy from below:
\begin{equation}
  S(\varrho) \geq \log \varrho_\infty \qquad\text{and}\qquad \mathcal{E}(\varrho,\varrho) \geq -\norm*{W_-}_\infty \,, \label{eq:elb}
\end{equation}
where the bound on the entropy follows from Jensen's inequality and the bound on the interaction energy follows from Assumption~\eqref{ass:B}.
Since by~\eqref{eq:elb}, $\crF_{\kappa}(\varrho)$ is bounded from below over $\Prac(U)$, there exists a minimising sequence $\{\varrho_j\}_{j=1}^\infty \subset  \Prac(U)$. Furthermore, by~\cref{lem:1}, the minimising sequence can be chosen such that $\{\varrho_j\}_{j=1}^\infty \subset  \Leb^2(U)$ with $\norm*{\varrho_j}_2 \leq B_0^{\frac{1}{2}}$, where $B_0$ 
is the constant from~\cref{lem:1}. Thus, there exists a subsequence which we continue to denote by $\{\varrho_j\}_{j=1}^\infty$ 
such that $\varrho_j \weak \varrho_{\kappa}$ weakly in $\Leb^2(U)$.
Clearly we have that $\intom{\varrho_\kappa}=1$. It is also easy to see that $\varrho_\kappa \geq 0$, a.e. Thus $\varrho_\kappa \in \Leb^2(U) \cap \Prac(U)$. The lower semicontinuity of $S(\varrho)$ follows from standard results (cf. ~\cite{jost1998calculus}, Lemma 4.3.1).  
\begin{comment}For $k, l \in \N^d$, we write $k<l$ if $k_i < l_i$ for $i=\set{1,\dots, d}$. Taking the Fourier transform, we obtain for any $k_0 \in \N^d$
\begin{align}
|\mathcal{E}(\varrho_\kappa,\varrho_\kappa)-\mathcal{E}(\varrho_j,\varrho_j)| \leq \sum_{k \leq k_0} \frac{\tilde{W}(k)}{N_k}\sum_{\sigma \in \Sym(\Lambda)}\bra*{ |\tilde{\varrho_\kappa}(\sigma(k))|^2  - |\tilde{\varrho_j}(\sigma(k))|^2 } + C \min_{i\in\set{1,\dots,d}}\max_{k: k_i>k_0} \tilde{W}(k) \, ,
\end{align}
where the constant $C$ depends only on $L$ and $k_0$. The second term can be made arbitrarily small by choosing $k_0$ sufficiently large by the Riemann--Lebesgue lemma. The first term can be made arbitrarily small for any $k_0$ due to the weak convergence result by choosing $j$ sufficiently large. 
\end{comment}
Consider now the interaction energy term. For $W \in \Leb^1(U)$, the interaction energy is weakly continuous in $\Leb^2(U)$~\cite[Theorem 2.2, Equation (9)]{chayes2010mckean}. This implies that the free energy $\crF_\kappa(\varrho)$ has a minimiser $\varrho_\kappa$ over $\Prac(U)$. A direct consequence of this and~\cref{lem:pos} is that the minimisation problem is well-posed in $\Prac^+(U)$ since the minimiser $\varrho_\kappa$  must be attained in $\Prac^+(U)$. We can then use~\cref{thm:wellpsPDE} together with~\cref{prop:tfae} to argue that any such minimiser must be smooth.
\end{proof}
% We briefly remind the reader of the standard notion of convexity. A proper functional $\crF:K \to (-\infty,\infty]$ defined over some convex subset $K$ of a topological vector space $E$ is said to be convex if
% It is said to be strictly convex if the above inequality is strict.  A direct consequence of this definition is that a 
% strictly convex functional has unique minimisers provided they exist.
\begin{prop} \label{prop:con}
Assume $W $ satisfies Assumption~\eqref{ass:B} such that $W_\unstable$ is bounded from below, where $W_\unstable$ is the unstable part defined in~\cref{def:Hstable}. Then, for $\kappa \in \bra*{ 0,\kappa_{\mathrm{con}}}$, where $\kappa_{\mathrm{con}}:=\beta^{-1} \norm*{W_{\unstable-}}_\infty^{-1}$, the functional $\crF_\kappa(\varrho)$ is strictly convex on $\Prac(U)$, that is for all $s\in(0,1)$ holds
\begin{align}
\crF_{\kappa}\bra[\big]{(1-s) \varrho_1 + s \varrho_2} < (1-s) \crF_{\kappa}(\varrho_1) +s \crF_{\kappa}(\varrho_2)  \qquad \forall \varrho_1,\varrho_2 \in \Prac(U) \ . 
\end{align}
\end{prop}
\begin{proof}
For $\varrho_1,\varrho_2 \in \Prac^+(U)$, let $\varrho(s)=(1-s) \varrho_1 + s \varrho_2, s \in (0,1)$ and $\eta=\varrho_2-\varrho_1$. Then we have
\begin{align}
\frac{\dx{}^2 }{\dx{s^2}} \crF_\kappa(\varrho_s)&=\beta^{-1}\intom{\frac{\eta^2}{\varrho_s}} + \kappa \intom{(W \star \eta) \eta} =\beta^{-1}\intom{\frac{\eta^2}{\varrho_s^2}\varrho_s} + \kappa \intom{((W_\stable +W_\unstable)\star \eta) \eta}\,.
\end{align}
Now, we apply Jensen's inequality and use the fact that $W_\stable\in \HH_\stable$ which gives us 
\begin{align}
\frac{\dx{}^2 }{\dx{s^2}} \crF_\kappa(\varrho_s)& \geq \beta^{-1}\left(\intom{|\eta|}\right)^2 + \kappa \intom{(W_\unstable \star \eta) \eta}\,.
\end{align}
Finally we bound $W_\unstable(x)$ from below to obtain,
\begin{align}
\frac{\dx{}^2 }{\dx{s^2}} \crF_\kappa(\varrho_s)& \geq\left(\beta^{-1}- \kappa \norm*{W_{\unstable-}}_\infty\right) \left(\intom{|\eta|}\right)^2 \, ,
\end{align}
showing the desired statement.
\end{proof}
\begin{rem}
It follows from the above result that if $W_\unstable\equiv 0$, i.e., $W\in \HH_\stable$, then $\crF_\kappa(\varrho)$ is strictly convex for all $\kappa \in (0,\infty)$. 
\end{rem}
%%------------------------------------------------------------------------------------------------------%%
%%------------------------------------------------------------------------------------------------------%%
\section{Global asymptotic stability}\label{S:gas}
\subsection{Trend to equilibrium in relative entropy}

In this section, we will use the free energy as defined in~\eqref{eq:FreeEnergy} to study the global asymptotic stability of the uniform state for the system~\eqref{eq:PDE}. By introducing the relative entropy
\begin{align}\label{e:def:RelEnt}
  \mathcal{H}(\varrho |\varrho_{\infty}) =  \intom{\varrho \log\left(\frac{\varrho}{\varrho_\infty}\right)}  \ ,
\end{align}
we observe the following identity between the free energy gap and the relative entropy
\begin{align}\label{FreeEnergyExcess}
  \crF_\kappa(\varrho)-\crF_\kappa(\varrho_{\infty}) = \beta^{-1}\mathcal{H}(\varrho | \varrho_\infty) + \frac{\kappa}{2}\mathcal{E}(\varrho-\varrho_\infty,\varrho-\varrho_\infty) \ . 
\end{align}
By directly differentiating the relative entropy~\eqref{e:def:RelEnt}, we obtain the rate of change of the relative entropy
\begin{align}
    \frac{\dx{}\mathcal{H}(\varrho | \varrho_\infty)}{\dx{t}}= - \beta^{-1}\intom{|\nabla\log \varrho|^2 \varrho} -\kappa \intom{\nabla( W \star \varrho) \cdot \nabla \varrho} \, .  \label{eq:relen}
\end{align}

\begin{comment}
We now state two preliminary facts that are useful in the study of convergence to equilibrium. On the torus, the heat semigroup is hypercontractive and satisfies a logarithmic Sobolev inequality~\cite{emery1987simple} of the form
\begin{align}
  \mathcal{H}(\varrho | \varrho_\infty ) \leq \frac{L^2}{4 \pi^2}\intom{|\nabla \log \varrho|^2 \varrho } \ ,
  \label{eq:LSI} 
\end{align}
where the term on the right hand side is the Fisher information of the measure $\varrho$. It will be convenient to compare the $L^1$ norm with the relative entropy, which is granted by the Csisz\'ar--Kullback--Pinsker (CKP) inequality~\cite{bolley2005weighted}
\begin{align}
% \frac{1}{\sqrt{d}\, L}W_1(\varrho,\varrho_\infty) \leq
\norm*{\varrho - \varrho_\infty}_{1} \leq \sqrt{2 \mathcal{H}(\varrho|\varrho_\infty)} \label{eq:CKP} \, . 
\end{align}
These preliminary inequalities allow us to obtain exponential convergence to equilibrium in relative entropy.
\end{comment}
\begin{prop}[Exponential stability and convergence in relative entropy]\label{prop:review}
  Let $\varrho_0 \in \Prac(U)\cap \SobH^{3+d}(U)$ with
$S(\varrho_0) <\infty$ and  $W \in \SobW^{2,\infty}(U)$. Then the classical solution $\varrho$ of \eqref{eq:PDE} is exponentially stable in relative entropy and it holds
\begin{equation}
  \cH(\varrho(\cdot,t) | \varrho_\infty) \leq \exp\pra*{\bra*{-\frac{4 \pi^2}{\beta L^2} +2 \kappa \norm*{\Delta W_\unstable}_{\infty}} t } \ \cH(\varrho_0 | \varrho_\infty) \, . 
\end{equation}
Especially, in the cases $W \in \HH_\stable$ for any $\beta,\kappa >0$ and if $W\notin \HH_\stable$ for $\beta \kappa < \frac{2\pi^2}{L^2 \norm*{\Delta W_\unstable}_{\infty}}$ it holds that we have exponentially fast convergence to the uniform state in relative entropy for any initial condition $\varrho_0 \in \Prac(U)\cap \SobH^{3+d}(U)$.
\end{prop}
\begin{proof}[Proof of~\cref{thm:m1}\ref{thm:m1b}]
  We know the solution $\varrho$ is classical, thus $\mathcal{H}(\varrho(\cdot,t) | \varrho_\infty) \in C^1(0,\infty)$. Using~\eqref{eq:relen}, we obtain with another integration by parts
   \begin{align}
     \frac{\dx{}}{\dx{t}} \mathcal{H}(\varrho|\varrho_\infty) = - \beta^{-1}\intom{|\nabla\log \varrho|^2 \varrho} + \kappa \intom{\Delta W \star \varrho\  \varrho},
     \quad \forall t \in (0,\infty) \, . 
     \label{eq:reldec}
   \end{align}
   The first term is the Fisher information and can be controlled by a log-Sobolev inequality of the form
\begin{align}
  \mathcal{H}(\varrho | \varrho_\infty ) \leq \frac{L^2}{4 \pi^2}\intom{|\nabla \log \varrho|^2 \varrho } \ ,
  \label{eq:LSI} 
\end{align}
   Now, we rewrite the interaction term in its Fourier series by~\eqref{Fourier:Interaction}, estimate it in terms of the unstable modes and transform it back to position space
   \begin{align}
     \intom{\Delta W \star \varrho\  \varrho} &= -\frac{ 4\pi^2}{L^2} \sum_{k \in \N^d} \abs{k}^2 \tilde{W}(k)\frac{1}{N_k}\sum_{\sigma \in \Sym(\Lambda)}|\tilde{g}(\sigma(k))|^2 \\
     &\leq - \frac{ 4\pi^2}{L^2} \sum_{k \in \N^d} \abs{k}^2 \tilde{W}_\unstable(k)\frac{1}{N_k}\sum_{\sigma \in \Sym(\Lambda)}|\tilde{g}(\sigma(k))|^2 \\ &= \int \Delta W_\unstable \star \varrho \ \varrho \dx{x} \, .
   \end{align}
   Now, we use the fact that $\Delta W_\unstable$ has mean zero to replace $\varrho$ by $\varrho- \varrho_\infty$ and estimate
   \begin{align}
     \int \Delta W_\unstable \star \varrho \ \varrho \dx{x} &\leq \norm*{\Delta W_\unstable \star (\varrho -\varrho_\infty)}_\infty \norm*{ \varrho - \varrho_\infty}_1 \leq \norm*{\Delta W_\unstable}_\infty \norm*{ \varrho - \varrho_\infty}_1^2 \,.
   \end{align}
   The above term can be controlled using the CKP inequality in the following
   way
   \begin{align}
% \frac{1}{\sqrt{d}\, L}W_1(\varrho,\varrho_\infty) \leq
\norm*{\varrho - \varrho_\infty}_{1} \leq \sqrt{2 \mathcal{H}(\varrho|\varrho_\infty)} \label{eq:CKP} \, . 
\end{align}
   In combination with~\eqref{eq:LSI} and~\eqref{eq:CKP}, we obtain the bound
   \[
     \frac{\dx{}}{\dx{t}} \mathcal{H}(\varrho|\varrho_\infty) \leq \bra*{- \frac{4 \pi^2}{\beta L^2} + 2\kappa\norm*{\Delta W_\unstable}_\infty} \cH(\varrho | \varrho_\infty) \, . 
   \]
   Finally, by Gronwall's inequality, we have the desired result.
\end{proof}
\begin{rem}
  For the case of the noisy Hegselmann--Krausse model studied in ~\cite{chazelle2017well}, we have $W(x)=\int_{0}^y \phi(|x|)x \dx{y}$ with $\phi(|x|)=\infn_{|x|\leq R}$. We can estimate by the same arguments $\norm*{W_\unstable''(x)}_{\infty}\leq \norm*{W''(x)}_{\infty}=R$. Thus for 
  $\kappa<\frac{2 \pi^2}{\beta L^2}$, we have exponential convergence to equilibrium. See~\autoref{S:Hegselmann} for a detailed analysis of this model. 
\end{rem}
\begin{rem}
  By the improved entropy defect estimate of \cref{lem:entdef}, the above statement could be slightly improved under more specific assumptions on the unstable modes of the potential. For the moment, we want to keep the presentation as concise as possible and refer to \autoref{S:thermodynamic} for the details.
\end{rem}
\subsection{Linear stability analysis}\label{ssec:lsa}
We start this subsection by linearising the stationary Mckean--Vlasov equation around some stationary solution, $\varrho_\kappa$.
We obtain the following linear integrodifferential operator:
\begin{align}
  \cL w := \beta^{-1} \Delta w+ \kappa\dive \bra[\big]{ \varrho_\kappa \nabla (W \star  w)}
  + \kappa \dive\bra[\big]{w \nabla(W \star \varrho_\kappa)}\, .  
\end{align}
If we pick $\varrho_\kappa$ to be the uniforms state $\varrho_\infty$ the above expression reduces to
\begin{align}
  \cL w := \beta^{-1} \Delta w+ \kappa\varrho_\infty   \Delta(W \star  w) \, .
\end{align}
 We are now interested in studying the spectrum of this operator over mean zero $\Leb^2(U)$ functions, $\Leb^2_0(U)$. From the classical theory for symmetric elliptic operators, it follows that the eigenfunctions of this system form an orthonormal basis in $\Leb^2_0(U)$ given by $\{ L^{-\frac{d}{2}} e^{i \frac{2 \pi}{L} k' \cdot x}\}_{k' \in \Z^d\setminus\set{\mathbf{0}}}$ with the eigenvalues given by
 \begin{align}
\lambda_{k'}= \bra*{-\beta^{-1}\bra*{\frac{2 \pi |k'|}{L}}^2 - \kappa L^{-d/2} \bra*{\frac{2 \pi |k'|}{L}}^2 \hat{W}(k') } \, ,
 \end{align}
where $\hat{W}(k')= L^{-\frac{d}{2}}\intom{W(x)  e^{-i \frac{2 \pi}{L} k' \cdot x}}$.  One can check that we have the following relationship 
\begin{align}
\hat{W}(k') = \frac{1}{\Theta(k)} \tilde{W}(k) \,,  \qquad k_\ell= |k'_\ell|, \ k \in \N^d \, ,
\end{align}
where $\Theta(k)$ is as defined in~\eqref{e:def:thetak}. To obtain the above expression we have used the fact that $W$ is coordinate-wise even, which implies that
\begin{align}
\intom{W(x)  e^{-i \frac{2 \pi}{L} k' \cdot x}}&=\intom{W(x)  e^{-i \frac{2 \pi}{L} k' \cdot x}}\\
&=\intom{W(x)  \prod\limits_{\ell=1}^d \bra*{\cos\bra*{\frac{2 \pi k'_\ell x}{L} } + i \sin\bra*{\frac{2 \pi k'_\ell x}{L} }  }} \\
&= \intom{W(x)  \prod\limits_{\ell=1}^d \bra*{\cos\bra*{\frac{2 \pi k'_\ell x}{L} }  }} \, .
\end{align}
Thus, we have the following expression for the value of the parameter $\kappa_\sharp$ at which the first eigenvalue of~$\cL$
crosses the imaginary axis:
\begin{align}
\kappa_\sharp = -\frac{L^{d/2}\Theta(k)}{\beta \min_{k \in \N^d\setminus\set{\mathbf{0}}} \tilde{W}(k)} \, 
\label{eq:koc}.
\end{align}
We will refer to $\kappa_\sharp$ as the \textbf{point of critical stability}. We denote by $k^\sharp$ the critical wave number (if it is unique)
and define it as:
\begin{align}
k^\sharp:= \argmin_{k \in \N^d\setminus\set{\mathbf{0}}} \tilde{W}(k)  \, 
\label{eq:poc}.
\end{align} 
%%------------------------------------------------------------------------------------------------------%%

%%------------------------------------------------------------------------------------------------------%%
%%------------------------------------------------------------------------------------------------------%%

%%------------------------------------------------------------------------------------------------------%%
%%------------------------------------------------------------------------------------------------------%%
\section{Bifurcation theory}\label{S:lbt}
For the local bifurcation analysis, it is convenient to rewrite the fixed point equation~\eqref{eq:fixedPoint} of the nonlinear mapping~\eqref{def:T} by making the parameter $\kappa \in (0,\infty)$ explicit. Hence, in this section we consider the nonlinear map $F:\Leb^2_s(U) \times \R^+ \to \Leb^2_s(U)$ defined as
\begin{align}
  F(\varrho,\kappa)=F_\kappa(\varrho) = \varrho - \frac{1}{Z} e^{-\beta \kappa W \star \varrho}, \qquad\text{where}\qquad Z =\intT{e^{-\beta \kappa W \star \varrho}} \ ,
  \label{eq:kirkwoodmonroe}
\end{align}
where  $\beta > 0$ is fixed, and $W \in \Leb^2_s(U)$ with $\Leb^2_s(U)$, the space of coordinate-wise even and square integrable functions as defined in~\eqref{eq:def:L2s}.

 The purpose of this section is to study the bifurcation problem:
\begin{align}
 F(\varrho,\kappa)=0 \, .
 \end{align} 
 Any zero of $F(\varrho,\kappa)$ is also a coordinate-wise even fixed point of $\cT:\Prac \to \Prac$. The converse is true if~$W$ satisfies Assumption~\eqref{ass:B}. We do not make this assumption for the whole section as we want the bifurcation theory to be valid for more singular potentials, e.g., the Keller--Segel model which we treat in a later section. It is also clear that the map
 $F(\varrho,\kappa)$ is translation invariant on the whole space $\Leb^2_s(U)$, i.e., if $\varrho$ is a zero of $F(\varrho,\kappa)$ then so is any translate $\varrho(\cdot - y)$ of $\varrho(\cdot)$ for any $y\in U$. This is the motivation for the restriction of $F$ to the space $\Leb^2_s(U)$. We will further justify our choice of the space $\Leb^2_s(U)$ in~\cref{lem:lbfour}. 
 
 The first result is an easy consequence of the characterisation of stationary solutions from~\autoref{S:stationary}, but could be also derived by standard contraction mapping argument on the map $F$ as done in
~\cite[Theorem 4.1]{tamura1984asymptotic} and ~\cite[Theorem 3]{messer1982statistical}.   
\begin{prop}
Assume $W(x)$ satisfies Assumption~\eqref{ass:B}. Then, for $\kappa$ sufficiently small, the uniform state $\varrho_\infty$ is the only solution of $F(\varrho,\kappa)=0$.
\begin{proof}
 \cref{prop:con} implies that $\crF_\kappa(\varrho)$ is
 strictly convex for $\kappa< \kappa_{\mathrm{con}} = \beta^{-1} \norm*{W_\unstable}_\infty^{-1}$. Hence, using~\cref{thm:dirmet}, it has a unique minimiser and exactly one critical point. This implies from \cref{prop:tfae} that $F(\varrho,\kappa)$ has a unique solution. 
\end{proof}
\end{prop}
We use the trivial branch of solutions $F(\varrho_\infty,\kappa)=0,  \kappa \in(0,\infty)$ with $\varrho_\infty \equiv 1/L^d$ to centre the map and define for any $u\in \Leb^2_s(U)$
\begin{equation}\label{eq:hatF}
  \hat{F}(u,\kappa)=F(u + \varrho_\infty,\kappa) \, . 
\end{equation}
In this way, we have $\hat{F}(0,\kappa)=0$. We compute the Fr\'echet
derivatives of this map for variations $w_1,w_2, w_3\in \Leb_s^2(U)$
\begin{align}
  D_\varrho(\hat{F}(0,\kappa))[w_1] &= w_1 + \beta \kappa \varrho_\infty (W \star w_1) - \beta \kappa \varrho_\infty^2 \intT{(W \star w_1)(x)} \, , \label{eq:dr} \displaybreak[0]\\
  D_\kappa(\hat{F}(0,\kappa)) &=0  \, ,\label{eq:dk} \displaybreak[0]\\
  D^2_{\varrho \kappa}(\hat{F}(0,\kappa))[w_1] &= \varrho_\infty (W \star w_1) - \varrho_\infty^2 \intT{(W \star w_1) (x)} -\varrho_\infty^2 W \star D_\varrho ( \hat{F}(0,\kappa))[w_1] \, , \label{eq:drk} \displaybreak[0]\\
  D^2_{\varrho\varrho}(\hat{F}(0,\kappa))[w_1,w_2]&= \beta \kappa (w_2- D_\varrho( \hat{F}(0,\kappa))[w_2])(W \star w_1) 
  \label{eq:drr}  \\
  &-\beta \kappa \varrho_\infty  (w_2- D_\varrho ( \hat{F}(0,\kappa)) [w_2]) \intT{W \star w_1(x)}   \nonumber \\ 
  &-\beta \kappa \varrho_\infty \intT{W \star w_1(x) (w_2- D_\varrho ( \hat{F}(0,\kappa)) [w_2])(x)}  \, ,  \nonumber \displaybreak[0]\\
  D^3_{\varrho \varrho \varrho}\hat{F}(0,\kappa)[w_1,w_2,w_3] &= -\beta \kappa D^2_{\varrho \varrho}\hat{F}(0,\kappa)[w_2,w_3](W 
  \star w_1) \label{eq:drrr}  \\
  &+ \beta \kappa \varrho_\infty (D^2_{\varrho \varrho}\hat{F}(0,\kappa)[w_2,w_3]) \intT{(W \star w_1)(x)} \nonumber \\ 
    &- \beta \kappa (w_2 - D_\varrho \hat{F}(0,\kappa)[w_2]) \intT{(W \star w_1)(x) (w_3 - D_\varrho \hat{F}(0,\kappa)[w_3])(x)}
    \\ \nonumber 
    &- \beta \kappa (w_3 - D_\varrho \hat{F}(0,\kappa)[w_3]) \intT{(W \star w_1)(x) (w_2 - D_\varrho \hat{F}(0,\kappa)[w_2])(x)}
    \\ \nonumber
    &+ \beta \kappa \varrho_\infty  \intT{(W \star w_1)(x)(D^2_{\varrho \varrho}\hat{F}(0,\kappa)[w_2,w_3])(x)} \, .
\end{align}
We have the following characterisation of the local bifurcations of $\hat{F}$:
\begin{thm} \label{thm:c1bif}
  Consider $\hat{F}: \Leb^2_s(U) \times \R^+ \to \Leb^2_s(U)$ as defined in~\eqref{eq:kirkwoodmonroe} with $W \in \Leb^2_s(U)$.
  Assume there exists $k^* \in \N^d$, such that:
   \begin{enumerate}
     \item $\operatorname{card}\set*{k : \frac{\tilde{W}(k)}{\Theta(k)}=\frac{\tilde{W}(k^*)}{\Theta(k^*)}}=1$\,.
     \item $\tilde{W}(k^*) <0$\,.
   \end{enumerate}
   Then $(0,\kappa_*)\in \Leb^2_s(U) \times \R^+ $ is a bifurcation point of $\hat{F}(\varrho,\kappa)=0$ where
   \begin{align}
     \kappa_*=-\frac{L^{\frac{d}{2}} \Theta(k^*)}{\beta \tilde{W}(k^*)} \label{eq:thetak}\, .
   \end{align}
   In addition, there exists a branch of solutions of the form
   \begin{align}
     \varrho_{*}(s)&= \varrho_\infty + s w_{k^*} + r(s w_{k^*},\kappa(s)) \, , \label{eq:branchstructure} 
   \end{align}
where $w_{k^*} \in \Leb^2_s(U)$ defined in~\eqref{e:def:wk}, $s \in (-\delta,\delta)$ for some $\delta>0$, and $\kappa:(-\delta,\delta) \to V$ is a twice continuously differentiable function in a neighbourhood $V$ of $\kappa_*$ with $\kappa(0)=\kappa_*$. Moreover,  it holds $\kappa'(0)=0$, $\kappa''(0)=\frac{2\beta \kappa_*}{3 \varrho_\infty}>0$, and $\varrho_*$ is the only nontrivial solution in a neighbourhood of $(0,\kappa_*)$ in $\Leb^2_s(U) \times \R$.

Specifically, the error $r:\operatorname{span} [w_{k^*}] \times V \to (\operatorname{span} [w_{k^*}])^{\perp}\subset \Leb^2_s(U)$ is a map satisfying
\begin{gather}\label{eq:rprop}
 \forall s \in (-\delta,\delta): \quad r(s w_{k^*},\kappa(s)) \in \Leb^{\infty}(U) \qquad  \text{with}\qquad r(0,0)=0  \, , \\
\text{and}\qquad \lim_{|s| \to 0} \frac{\norm*{r(s w_{k^*},\kappa(s))}_2}{|s|}=0 \, .\nonumber
\end{gather}
\end{thm}
\begin{proof}[Proof of~\cref{thm:m2}]
The proof of this theorem relies on the Crandall--Rabinowitz theorem~\cite{crandall1971bif}, which for the convenience of the reader is included in Appendix \ref{S:lyapschmid}. Before we proceed it is convenient to rewrite $D_\varrho\hat{F}$ from~\eqref{eq:dr} as
\begin{align}\label{eq:hatFhatT}
  D_\varrho(\hat{F}(0,\kappa))&= I - \kappa \hat{T}  \ ,
\end{align}
where $\hat{T}: \Leb^2_s(U) \to \Leb^2_s(U)$ is defined for $w\in \Leb^2_s(U)$ by
\begin{align}\label{def:Th}
  (\hat{T} w)(x) = \beta \bra*{ -\varrho_\infty (W \star w)(x) + \varrho_\infty^2 \int (W \star w)(y) \dx{y} } \, . 
\end{align}
Using the above expression one checks that the linear operator $\hat{T}$ is Hilbert--Schmidt with $\norm*{\hat{T}}_{\mathrm{HS}}^2 = \sum_{k \in \mathbb{N}^d}\norm*{\hat{T}w_k}_2^2 < \infty$, where $\{w_k\}_{k \in \N^d}$ is the orthonormal basis of $\Leb^2_s(U)$ as defined earlier. Thus, $I - \kappa \hat{T}$ is Fredholm by ~\cite[Corollary 4.3.8]{davies2007linear}. Since the index of a Fredholm operator is homotopy invariant (cf. Theorem 4.3.11 ~\cite{davies2007linear}), we show that the mapping $\kappa \mapsto (I-  \kappa\hat{T})$ is norm-continuous:
\begin{align}
  \norm*{I- \kappa_1 \hat{T} -I + \kappa_2 \hat{T}} = |\kappa_2 -  \kappa_1| \norm*{\hat{T}} \, .
\end{align}
Thus, the index satisfies $\ind\bra{I-\kappa\hat{T}}= \ind\bra*{I}=0$. We diagonalize $I -\kappa\hat{T}$ with respect to $\set{w_k}_{k\in \N^d}$
\begin{align} \label{eq:diagT}
  (I-\kappa\hat{T})w_k(x)=
  \begin{cases}
      \qquad 1  &,   \forall i=1 \dots d: k_i =0 \, ,  \\
    \left( 1 + \beta \kappa \frac{\tilde{W}_k}{L^{d/2} \Theta(k)}\right) w_k(x) &, \text{ else.} 
  \end{cases}  
\end{align}
Now it is easy to see that if Condition (1) in the statement of the theorem is satisfied, then $\dim \ker (I - \kappa \hat{T})=1$ for $\kappa=  \kappa_*$. Indeed, if Condition~(1) is satisfied, we have $\ker(I-\kappa_* \hat{T})= \mathrm{span}[w_{k^*}]$ and Condition~(2) ensures that $\kappa_*$ is positive.

Thus Condition~(1) of Theorem \ref{thm:cr} is satisfied. Since $\Ima (I -  \kappa \hat{T})$ is closed we have that $\Ima(I-\kappa \hat{T})=\ker (I - \kappa \hat{T}^*)^{\perp}$, with $\hat{T}^*$ denoting the adjoint. It is easy to check that if $v_0 \in \ker (I-\kappa \hat{T})$, $v_0 \not\equiv 0 $ then $v_0 \in \ker (I-\kappa \hat{T}^*) $. Then, by differentiating \eqref{eq:hatFhatT} in $\kappa$ and using $v_0 \in \ker (I-\kappa \hat{T})$ , we get the identity
\begin{align}
  \skp*{D^2_{\varrho\kappa}(\hat{F}(0, \kappa))[v_0],v_0}&= -\skp*{\hat{T} v_0,v_0}
  =-\kappa^{-1} \norm*{v_0}_2^2 \neq 0 \, ,  
\end{align}
since $v_0 \not\equiv 0$ by assumption.This implies that $D^2_{\varrho \kappa}(\hat{F}(0, \kappa))[v_0] \notin \ker (I-\kappa \hat{T}^*)^{\perp}$. Thus condition (2) of \cref{thm:cr} is also satisfied. Thus we can now apply \cref{thm:cr} and use \eqref{eq:dk} to obtain \eqref{eq:branchstructure}. 

Before proceeding, it is useful to characterize $\Ima (I -\kappa_* \hat{T})$. By using \eqref{eq:diagT}, we can see that we have the following orthogonal decomposition of $\Leb^2_s(U)$ into
\begin{align}
\Leb^2_s(U)= \mathrm{span}[w_{k^*}] \oplus \Ima(I -\kappa_* \hat{T} ) \, . 
\end{align}

Using the identity ~\cite[(I.6.3)]{kielhofer2006bifurcation} it follows that $\kappa'(0)=0$ provided that $D^2_{\varrho \varrho}\hat{F}(0,\kappa)[w_{k^*},w_{k^*}] \in \Ima (I -\kappa_* \hat{T} )$. Thus it is sufficient to check that 
\begin{align}
\skp*{D^2_{\varrho \varrho}\hat{F}(0,\kappa)[w_{k^*},w_{k^*}],w_{k^*}}= \skp*{\beta \kappa \tilde{W}(k^*)\left[w_{k^*}^2\left(\frac{L}{2}\right)^{d/2}-\left(\frac{1}{2L}\right)^{d/2}\right], w_{k^*}} =0 \ ,
\end{align}
where we have used \eqref{eq:drr} and the fact that $\intT{w_{k^*}^3}=0$. Thus we conclude that $\kappa'(0)=0$. Likewise, from~\cite[(I.6.11)]{kielhofer2006bifurcation}, we also have that 
\begin{align}
\kappa''(0)&= -\frac{\skp*{D^3_{\varrho \varrho \varrho}\hat{F}(0,\kappa_*)[w_{k^*},w_{k^*},w_{k^*}],w_{k^*}}}{3\skp*{D^2_{\varrho \kappa}\hat{F}(0,\kappa_*)[w_{k^*}],w_{k^*}}} = \frac{2 \beta \kappa_* \tilde{W}(k^*)(L/2)^{d/2}}{3 \varrho_\infty
\tilde{W}(k^*)(L/2)^{d/2}} = \frac{2\beta \kappa_*}{3\varrho_\infty} >0 \ ,
\end{align}
where we have used \eqref{eq:drk} and \eqref{eq:drrr}. The first two properties of \eqref{eq:rprop} follow from \cref{thm:cr}. To prove the third property in \eqref{eq:rprop}, we observe that
\begin{align}
\lim_{|s| +|\kappa(s)-\kappa_*| \to 0} \frac{\norm*{r_1(s\hat{v_0},\kappa(s))}_2}{|s| +|\kappa(s)-\kappa_*|}=0 \, . 
\end{align}
Since $\kappa'(0)=0$, we also have $\lim_{|s|\to 0}\frac{|\kappa(s)-\kappa_*|}{|s|}=0$.
Thus, we conclude
\begin{align}
\lim_{|s| \to 0} \frac{\norm*{r(s w_{k^*},\kappa(s))}_2}{|s|}&= \lim_{|s| \to 0} \frac{\norm*{r(s w_{k^*},\kappa(s))}_2}{|s| +|\kappa(s)-\kappa_*|}\left(\lim_{|s|\to 0}\frac{|s| +|\kappa(s)-\kappa_*|}{|s|}\right) =0 \ ,
\end{align}
where we have used the fact from~\cref{thm:cr} that $\kappa$ is continuously differentiable. This completes the proof.
\end{proof}
The statment of~\cref{thm:c1bif} becomes more transparent in one dimension:
\begin{cor}
Fix $U=(-L/2,L/2)$ and consider $\hat{F}: \Leb^2_s(U) \times \R^+ \to \Leb^2_s(U)$ as defined in~\eqref{eq:kirkwoodmonroe} with $W \in \Leb^2_s(U)$. Assume that there exists $k^* \in \N$, such that:
   \begin{enumerate}
     \item $\operatorname{card}\set*{k : \tilde{W}(k)=\tilde{W}(k^*)}=1$\,.
     \item $\tilde{W}(k^*) <0$\,.
   \end{enumerate}
   Then $(0,\kappa_*)$ is a bifurcation point of $\hat{F}(\varrho,\kappa)=0$, where
   \begin{align}
     \kappa_*=-\frac{ (2L)^{\frac{1}{2}} }{\beta \tilde{W}(k^*)}  ; 
   \end{align}
   that is, there exists a branch of solutions having the following form:
   \begin{align}
     \varrho_{*}(s)&= \frac{1}{L}+ s \sqrt{\frac{2}{L}}\cos\bra*{\frac{2 \pi k^* x }{L}} + o(s), \quad s \in(-\delta,\delta) \, , 
   \end{align}
   with all the other properties of the branch being the same as~\cref{thm:c1bif}.
\end{cor}
\begin{rem}
It should also be noted that one can obtain the existence of bifurcations with higher-dimensional kernels as well, i.e, when $\dim(\ker (\hat{T})) >1$. Since $\hat{T}$ is self adjoint, for any eigenvalue its algebraic and geometric multiplicities are the same. From~\cite[Theorem 28.1]{deimling1985nonlinear} it follows that any characteristic values (the reciprocals of the eigenvalues of $\hat{T}$) of odd algebraic multiplicity correspond to a bifurcation point. This implies that we could replace Condition (1) in~\cref{thm:c1bif} with $\operatorname{card}\set*{k : \frac{\tilde{W}(k)}{\Theta(k)}=\frac{\tilde{W}(k^*)}{\Theta(k^*)}}=m$, where $m$ is odd. However, it is not easy to obtain detailed information about the structure and regularity of the bifurcating branches in this case. 
\end{rem}
\begin{rem}~\label{rem:inj}
  Condition (1) of \cref{thm:c1bif} is in particular satisfied for an interaction potential $W \in \Leb^2_s(U)$ if the map $\tilde{W} : \N^d \to \R$ is injective. In this case, every $k_\alpha \in \N^d$ such that $\tilde{W}(k) <0$, corresponds to a unique bifurcation point $ \kappa_\alpha$ of  $F(\varrho,\kappa)$ through the relation~\eqref{eq:thetak}.  For example consider the interaction potential $W(x)=x^2/2$. In this case $\tilde{W}$  is injective and therefore the system has infinitely many bifurcation points. On the other hand, when $W(x)=-w_k(x)$ for some $k \in \N^d$, the system has only one bifurcation point. 
\end{rem}
\begin{rem} \label{rem:l2pi}
In dimensions higher than one, the space $\Leb^2_s(U)$ may not be small enough for our purposes, i.e., it is
possible that the potential may have additional symmetries. For instance, the potential could be exchangeable, that is $W(x)=W(\Pi(x))$ for all possible permutations~$\Pi$ of the $d$ coordinates. 
In this case it is easy to check that $\skp*{W,w_k}=
\skp*{W,w_{\Pi(k)}}$ for all $k \in \N^d$. We can then define the equivalence relation, $k \sim k'$ if $k' = \Pi(k)$ for some permutation $\Pi$ and write $\pra*{k}$ for the corresponding equivalence class. Thus, the consequence of $W(x)$ having this symmetry is that the value $\tilde{W}(k)/\Theta(k)$ is constant on $\pra{k}$. This implies that kernel of $D_\varrho\hat{F}$ is can never be one-dimensional. We can quotient out this symmetry by defining the space $\Leb^2_{\operatorname{ex}}(U)=\operatorname{span}\set{w_{\pra{k}}}$, where
$\{w_{\pra{k}}\}$ is an orthonormal basis defined by
\begin{align}
w_{\pra{k}}= \frac{1}{\sqrt{\sharp\pra{k}}}\sum_{\ell \in \pra{k}} w_{\ell}(x), \quad k \in \N^d \, ,
\end{align}
where $\sharp\pra{k}$ denotes the cardinality of the equivalence class $\pra{k}$. Then $\hat{F}:\Leb^2_{\operatorname{ex}}(U) \times \R^+ \to \Leb^2_{\operatorname{ex}}(U)$ is a well-defined mapping. Then, the results of \cref{thm:c1bif} carry over to $\hat{F}$ defined this way for $W \in \Leb^2_{\operatorname{ex}}(U)$ and the corresponding orthonormal basis $\{w_{\pra{k}}\}_{k\in \N}$. In this case the conditions read as follows
   \begin{enumerate}
     \item $\operatorname{card}\set*{ \pra*{k} : \frac{\tilde{W}([k])}{\Theta([k])} = \frac{\tilde{W}([k^*])}{\Theta([k^*])}} = 1 $ \, ,
     \item $\tilde{W}([k^*]) <0$ \, ,
   \end{enumerate}
with $\tilde{W}([k])=\tilde{W}(k),\Theta([k])=\Theta(k)$ for any $k \in \pra{k}$. The bifurcation point is given by
   \begin{align}
     \kappa_*=-\frac{L^{\frac{d}{2}} \Theta(\pra{k^*})}{\beta \tilde{W}(\pra{k^*})} \,.
   \end{align}
\end{rem}
\begin{rem}
Consider the following interaction potential
\begin{align}
W_s(x)= -\sum_{k=1}^\infty \frac{1}{|k|^{2s}}w_k(x), \quad s \geq 1 \,.
\end{align}
It is straightforward to check that $W_s(x)$ belongs to $\SobH^s(U)$ and thus to $C(\bar{U})$. Additionally, $W_s(x) \to -w_1(x)$ uniformly as $s \to \infty$. One can check now that, for any $s>1$, $W_s(x)$ satisfies the conditions of~\cref{thm:c1bif} for all
$k \in \N, k \neq 0$ and thus the trivial branch of the system has infinitely many bifurcation points. However, as mentioned in~\cref{rem:inj}, the system $W(x)=-w_1(x)$ has only one bifurcation point. This can be explained by the fact that as $s\to \infty$
all bifurcation points of $W_s(x)$ except one are pushed to infinity. This example illustrates however that two potentials may ``look'' similar but their associated  bifurcation structure may be entirely different. Therefore, approximating potentials, even uniformly,
by some dense subset, may not reveal all the information about the bifurcation structure of the limiting system.
\end{rem}
If we now assume that $W$ satisfies assumption~\eqref{ass:B} we can see that the zeros of $F(\varrho,\kappa)$ are  fixed points of the map $\cT$ which by~\cref{prop:tfae} are equivalent to smooth solutions of the stationary McKean--Vlasov equation. \cref{thm:c1bif} also provides us information about the structure of the branches, i.e., if $w_k(x)$ is the mode such that $k\in \N^d$ satisfies the conditions of \cref{thm:c1bif}, then to leading order the nontrivial solution is of the form $\varrho_\infty + s w_k(x)$. One may think of this as a ``proto-cluster'', with the nodes of $w_k(x)$ corresponding to the  positions of the peaks and valleys of the cluster. 

So far the analysis in this section has been local. We conclude this section by providing a characterisation of the global structure of the bifurcation diagram for $\hat{F}$ as defined in~\eqref{eq:hatF}.
\begin{prop}
Let  $V$ be an open neighbourhood of $(0,\kappa_*)$ in $\Leb^2_s(U) \times \R$, where $(0,\kappa_*)$ is a bifurcation point of the map $\hat{F}$ in the sense of~\cref{thm:c1bif}. We denote by $\cC_V$  the set of nontrivial solutions of $\hat{F}(\varrho,\kappa)=0$ in $V$ and by $\cC_{V,\kappa_*}$ the connected component of $\overline{\cC_V}$ containing $(0,\kappa_*)$. Then $\cC_{V,\kappa_*}$ has at least one of the following two
properties:
\begin{enumerate}
\item $\cC_{V,\kappa_*} \cap \partial V \neq \emptyset$\,.
\item $\cC_{V,\kappa_*}$ contains an odd number of characteristic values of $\hat{T}$, $(0,\kappa_i) \neq (0,\kappa_*)$, which have odd algebraic multiplicity\,.
\end{enumerate} 

\end{prop}
\begin{proof}
The proof follows from the direct application of the so-called Rabinowitz alternative ~\cite[Theorem 29.1]{deimling1985nonlinear} which we have included as~\cref{thm:rabalt} for the convenience of the reader. It is easy to check that the map $\hat{F}$ can be written in the following form,
\begin{align}
\hat{F}(\varrho,\kappa)= \varrho - \kappa \hat{T}\varrho + G(\varrho,\kappa) \, ,
\end{align}
with $\hat{T}$ as defined in~\eqref{def:Th}, and
\begin{align}
G(\varrho,\kappa)= \varrho_\infty - \frac{1}{Z}e^{-\beta \kappa W \star \varrho} + \kappa \hat{T} \varrho \, .
\end{align}
We now need to show that $G$ is completely continuous and $o\bra*{\norm*{\varrho}_2}$ uniformly in $\kappa$ as $\norm*{\varrho}_2 \to 0 $.
For the first result, it is enough to show that $G$ is compact since $\Leb^2_s(U)$ is reflexive. We establish the following estimate: 
\begin{align}
\norm*{G(\varrho_1,\kappa) - G(\varrho_2,\kappa)}_2 &\leq\frac{1}{Z(\varrho_2)} \norm*{e^{-\beta \kappa W \star \varrho_2} -e^{-\beta \kappa W \star \varrho_1}}_2 + \frac{\norm*{e^{-\beta \kappa W \star \varrho_1}}_\infty}{Z(\varrho_2) Z(\varrho_2)} L^{d/2}\abs*{ Z(\varrho_2)-Z(\varrho_1)}  +  \\ &+\kappa\norm*{\hat{T}(\varrho_2-\varrho_1)}_2 \\
& \leq \frac{\beta \kappa}{L^{d/2}}e^{\beta \kappa \norm*{W}_2 \norm*{\varrho_2}_2  }\bra*{1+ e^{2\beta \kappa \norm*{W}_2 \norm*{\varrho_1}_2  }} \norm*{W \star (\varrho_2-\varrho_1)}_\infty  \\ & +\frac{2\beta \kappa}{L^{d/2}} \norm*{W \star (\varrho_2-\varrho_1)}_\infty \,.
\end{align}
Now setting $\varrho_2=\varrho$ and $G(\varrho_1,\kappa)=\tau G(\varrho,\kappa)=  G(\tau \varrho,\kappa)$(with $\tau f(x+\tau)$) in the above expression we obtain
\begin{align}
\norm*{G(\varrho,\kappa) - \tau G( \varrho,\kappa)}_2& \leq C_\kappa  \norm*{W \star\varrho - \tau W \star\varrho }_\infty \,  \label{eq:l2comp}.
\end{align}
Similarly we can also deduce the following estimate by bounding $W \star (\varrho_2-\varrho_1)$ from above:
\begin{align}
\norm*{G(\varrho_1,\kappa) - G(\varrho_2,\kappa)}_2 &\leq C_\kappa \norm*{W}_2 \norm*{\varrho_1-\varrho_2}_2 \, . \label{eq:spohn}\\
\end{align}
In the above two expressions, $C_\kappa$ is a constant which tends to 0 as $\kappa \to 0$. Setting $\varrho_2=0$ in~\eqref{eq:spohn}, it follows that $G$ is a bounded map on $\Leb^2(U)$. Together with this and~\eqref{eq:l2comp}, and using the fact that the convolution is uniformly continuous, one can check that that $G(A)$ satisfies the conditions of the Kolmogorov--Riesz theorem, where $A$ is any bounded subset of $\Leb^2_s(U)$. Thus $G$ is compact. The fact that $G$ is $o\bra*{\norm*{\varrho}_2}$  follows by Taylor expanding $e^{-\beta \kappa W \star \varrho}/Z$. 

One can now check that if condition (1) of~\cref{thm:c1bif} is satisfied for some $k \in \N^d$, the associated eigenvalue $\kappa^{-1}$(which could be negative) of $\hat{T}$ is simple, i.e., it has algebraic multiplicity one. This implies that all bifurcation points predicted by \cref{thm:c1bif} are associated with simple eigenvalues of $\hat{T}$. Thus, we can apply \cref{thm:rabalt} to complete the proof. 
\end{proof}
%%------------------------------------------------------------------------------------------------------%%
%%------------------------------------------------------------------------------------------------------%%
\section{Phase transitions for the McKean--Vlasov equation}\label{S:thermodynamic}
%\textcolor{red}{\textbf{TO DO: Connection with statistical mechanics definitions}}
%
We know from~\cref{prop:con} that $\varrho_\infty$ is the unique minimiser of the free energy for $\kappa$ sufficiently small. We are interested in studying under what criteria there is a change in the qualitative structure of the set of minimisers of $\crF_\kappa$. For the rest of this section we will assume that $W$ satisfies Assumption~\eqref{ass:B}, i.e, $W \in \SobH^1(U)$ and bounded below. We build on and extend the notions introduced by ~\cite{chayes2010mckean}. The first definition introduces what we mean by a transition point.
\begin{defn}[Transition point] \label{defn:tp}
A parameter value $\kappa_c >0$ is said to be a \emph{transition point} of $\crF_\kappa$ if it satisfies the following conditions:
\begin{enumerate}
\item For $0<\kappa < \kappa_c$, $\varrho_\infty$ is the unique minimiser of $\crF_\kappa(\varrho)$\,.
\item For $\kappa=\kappa_c$, $\varrho_\infty$ is a minimiser of $\crF_\kappa(\varrho)$\,.
\item For $\kappa>\kappa_c$, there exists some $ \varrho_\kappa \in \Prac^+(U)$, not equal to $\varrho_\infty$, such that $\varrho_\kappa$ is a minimiser of $\crF_\kappa(\varrho)$\,.
\end{enumerate}
\end{defn}
In the present work, we are only interested in the first transition point by increasing $\kappa$ starting from~$0$, also called the lower transition point. To convince the reader that the above definition makes sense we include the following result from~\cite{chayes2010mckean}.
\begin{prop}[{\cite[Proposition 2.8]{chayes2010mckean}}]~\label{prop:2.8cp}
Assume $W \in \HH_\stable^c$ and suppose that for some $\kappa_T <\infty$ there exists $\varrho_{\kappa_T} \in \Prac^+(U)$ not equal to $\varrho_\infty$ such that:
\begin{align}
\crF_{\kappa_T}(\varrho_{\kappa_T}) \leq \crF_{\kappa_T}(\varrho_\infty)\, .
\end{align}
Then, for all $\kappa>\kappa_T$, $\varrho_\infty$ no longer minimises the free energy. 
% By a variant of the same argument we have that at $\kappa=\kappa_c$, where $\kappa_c$ is transition point in the sense of~\cref{defn:tp}, $\varrho_\infty$ is still a minimiser of the free energy.
\end{prop}
In addition, the following result from~\cite{gates1970van} shows that $H$-stability of the potential is a necessary and sufficient condition for the nonexistence of a transition point. 
\begin{prop}[{\cite{gates1970van}}] \label{prop:tpex}
$\crF_\kappa$ has a transition point at some $\kappa=\kappa_c<\infty$ if and only if $W \in \HH_\stable^c$. Additionally for $\kappa>\kappa_\sharp$, with $\kappa_\sharp$ the point of critical stability as defined in~\eqref{eq:koc} in~\autoref{ssec:lsa}, $\varrho_\infty$ is not the minimiser of $\crF_\kappa$.
\end{prop}  
From this result it follows directly that if the system possesses a transition point $\kappa_c$, $\varrho_\infty$ can no longer be a minimiser beyond this point. We are also interested in understanding how this transition occurs. In the infinite-dimensional setting it is not always possible to obtain a well-defined order parameter for the system characterizing first and second order phase transitions in the sense of statistical physics. For this reason, it may be better to define such transitions in terms of discontinuity in some norm or metric.
\begin{defn}[Continuous and discontinuous transition point] \label{defn:ctp}
A transition point  $\kappa_c >0$ is said to be a \emph{continuous transition point} of $\crF_\kappa$ if it satisfies the following conditions:
\begin{enumerate}
\item For $\kappa=\kappa_c$, $\varrho_\infty$ is the unique minimiser of $\crF_\kappa(\varrho)$\,.
\item Given any family of minimisers, $\{\varrho_\kappa|\kappa> \kappa_c \}$, we have that
\begin{align}
\limsup_{\kappa \downarrow \kappa_c} \norm*{\varrho_\kappa-\varrho_\infty}_1=0 \, .
\end{align}
\end{enumerate}
A transition point $\kappa_c$ which is not continuous is said to be \emph{discontinuous}.
\end{defn}

We now include a series of results from~\cite{chayes2010mckean} that we need
for our subsequent analysis.
\begin{prop}[\cite{chayes2010mckean}]\label{prop:cpni}
$\min\limits_{\varrho \in \Prac(U)} \crF_\kappa(\varrho) - \frac{1}{2}\kappa \cE(\varrho_\infty,\varrho_\infty)$ is nonincreasing in $\kappa$.
\end{prop}
\begin{prop}[\cite{chayes2010mckean}]\label{prop:cpdc}
Assume $W \in \HH_s^c$ and that condition (2) of~\cref{defn:ctp} is violated. 
Then there exists a discontinuous transition point $\kappa_c < \infty$ and 
some $\varrho_{\kappa_c} \neq \varrho_\infty$ such that $\crF_{\kappa_c}(\varrho_{\kappa_c})=\crF_{\kappa_c}(\varrho_\infty)$.
\end{prop}
\begin{prop}[\cite{chayes2010mckean}]\label{prop:cpc}
Assume $W \in \HH_s^c$ and that the free energy $\crF_\kappa$ exhibits
a continuous transition point at some $\kappa_c < \infty$. Then it follows that $\kappa_c=\kappa_\sharp$.
\end{prop}

By combining certain properties of transition points with the previous analysis on critical stability in~\autoref{ssec:lsa}, we obtain more streamlined sufficient conditions for the identification of transition points, which is the basis for the proof of \cref{thm:m3}, or more precisely~\cref{thm:dctp} and~\cref{thm:spgap}. 
\begin{prop}\label{prop:CharactTP}
  Let $\crF_\kappa$ have a transition point at some $\kappa_c<\infty$ and let $\kappa_\sharp$ denote the point of critical stability defined in~\autoref{ssec:lsa}. 
  \begin{tenumerate}
   \item If $\varrho_\infty$ is the unique minimiser of $\crF_{\kappa_\sharp}$, then $\kappa_c = \kappa_{\sharp}$ is a continuous transition point.  \label{prop:CharactTP:cont}
   \item If $\varrho_\infty$ is not a global minimiser of $\crF_{\kappa_\sharp}$, then $\kappa_c < \kappa_{\sharp}$ and $\kappa_c$ is a discontinuous transition point. \label{prop:CharactTP:discont}
  \end{tenumerate}
\end{prop}
\begin{rem}
The statements of~\cref{prop:CharactTP}\ref{prop:CharactTP:cont} and ~\cref{prop:CharactTP}\ref{prop:CharactTP:discont}
are only necessary conditions for the characterisation of transition points. In particular, they are not logical complements of each other, i.e., $\varrho_\infty$ could be a
global minimiser of $\crF_{\kappa_\sharp}$ without being the unique one or vice versa.
\end{rem}
\begin{proof} 
  A consequence of the assumption in the first statement~\ref{prop:CharactTP:cont} of the proposition is that $\varrho_\infty$ is the unique minimiser for all $\kappa\leq \kappa_\sharp$. Indeed, from~\cref{prop:cpni}, we know that  $\min_{\varrho \in \Prac(U)}\crF_\kappa \leq \crF_{\kappa_c}(\varrho_\infty)$ for $\kappa \leq \kappa_c$. Thus, if $\varrho_\infty$ is the unique minimiser at some $\kappa=\kappa_c$, it must be a minimiser for all $\kappa \leq \kappa_c$.  In fact, using~\cref{prop:2.8cp} we can assert that $\varrho_\infty$ is the unique minimiser of $\crF_\kappa$ for all $\kappa \leq \kappa_c$. Indeed, if this were not the case then there exists
  some $\varrho_{\kappa_T}\in \Prac^+(U)$ not equal to $\varrho_\infty$ such that $\crF_{\kappa_T}(\varrho_{\kappa_T}) = \crF_{\kappa_T}(\varrho_\infty)$ for some $\kappa_T<\kappa_\sharp$. ~\cref{prop:2.8cp} then tells us that $\varrho_\infty$ can no longer be a minimiser for any $\kappa>\kappa_T$, which is a contradiction. It follows that conditions (1) and (2) from \cref{defn:tp} are satisfied. That condition (3) is satisfied follows directly from~\cref{prop:tpex}. This implies that $\kappa_\sharp$ satisfies the three conditions of being a transition point.
  
  Now, we have to verify condition (2) of \cref{defn:ctp} (condition (1) is already satisfied from the statement of the proposition). Assume condition (2) doesn't hold, i.e., there exists a family of minimisers  $\{\varrho_\kappa|\kappa> \kappa_c \}$ of $\crF_\kappa(\varrho)$ such that  $\limsup_{\kappa \downarrow \kappa_c} \norm*{\varrho_\kappa-\varrho_\infty}_1\neq0$. Then we know from~\cref{prop:cpdc} that there exists some $\varrho_{\kappa_c} \in \Prac^+(U)$ not equal to $\varrho_\infty$  such that it is a minimiser of the free energy $\crF_\kappa(\varrho)$ at $\kappa=\kappa_c$. Applied in the present setting with $\kappa_c=\kappa_\sharp$, we would deduce that $\varrho_\infty$ is no longer the unique minimiser of $\crF_{\kappa_\sharp}(\varrho)$, in contradiction to statement~\ref{prop:CharactTP:cont} of the proposition. Thus both conditions (1) and (2) of~\cref{defn:ctp} are satisfied from which it follows that $\kappa_c=\kappa_\sharp$ is a continuous transition point.

  To prove the second statement~\ref{prop:CharactTP:discont} of the proposition, let $\varrho$ be such that $\crF_{\kappa_\sharp}(\varrho) < \crF_{\kappa_\sharp}(\varrho_\infty)$. Then for any $\kappa$ close enough to $\kappa_\sharp$, we also have $\crF_{\kappa}(\varrho) < \crF_{\kappa}(\varrho_\infty)$. Hence by a combination of \cref{prop:2.8cp} and \cref{prop:tpex} there exists a transition point $\kappa_c < \kappa_\sharp$ and, in particular $\kappa_\sharp$, cannot be a transition point. From~\cref{prop:cpc}, we have the fact that if $\kappa_c$ is a continuous transition point of $\crF_{\kappa}$, then necessarily $\kappa_c =\kappa_\sharp$. This implies that $\kappa_c< \kappa_\sharp$ cannot be a continuous transition point.  
\end{proof}
Before proceeding to present the main results of this section, we remind the reader that for
the rest of the paper $\kappa_c$ denotes a transition point, $\kappa_\sharp$ denotes the point of critical stability, and $\kappa_*$ denotes a bifurcation point. 
\subsection{Discontinuous transition points}
We provide below a characterisation of potentials which exhibit discontinuous transition points, which proves~\cref{thm:m3}\ref{thm:m3a}.
\begin{defn}\label{def:del}
Assume $W\in \HH_\stable^c$ and let $K^{\delta}:=\left\{k' \in \N^d\setminus\set{\mathbf{0}}: \frac{\tilde{W}(k')}{\Theta(k')}\leq \min_{k \in \N^d\setminus\set{\mathbf{0}}} \frac{\tilde{W}(k)}{\Theta(k)} +\delta  \right\}$ for some $\delta \geq 0$. We define $\delta_*$ to be the smallest value, if it exists, of $\delta$ for which the following condition is satisfied:

\begin{equation}\label{eq:c1}
  \text{there exist } k^a,k^b, k^c \in K^{\delta_*} , \text{ such that } k^a=k^b + k^c \,  \tag{\textbf{C1}}. 
\end{equation}
\end{defn}
\begin{thm}\label{thm:dctp}
Let $W(x)$ be as in~\cref{def:del}. Then if $\delta_*$ exists and is sufficiently small, $\crF_\kappa$ exhibits a discontinuous transition point at some $\kappa_c<\kappa_\sharp$. 
\end{thm}
\begin{proof}
We know already from~\cref{prop:tpex} that the system possesses a transition point $\kappa_c$. We are going to use \cref{prop:CharactTP}~\ref{prop:CharactTP:discont} and construct a competitor $\varrho \in \Prac^+(U)$ which has a lower value of the free energy than $\varrho_\infty$ at $\kappa=\kappa_\sharp$. Let
\begin{align}
\varrho= \varrho_\infty\bra[\bigg]{1 + \epsilon \sum_{k \in K^{\delta_*} }w_{k}} \in \Prac^+(U) \ ,
\end{align}
for some $\epsilon >0$, sufficiently small. We denote by $|K^{\delta_*}|$ the cardinality of $K^{\delta_*}$, which is necessarily finite as $W \in \Leb^2(U)$. Expanding about $\varrho_\infty$ we obtain
\begin{align}
\beta^{-1}S(\varrho)&= \beta^{-1}\left(S(\varrho_\infty) +\frac{|K^{\delta_*}|}{2}\varrho_\infty \epsilon^2 -\frac{\varrho_\infty}{3}\intom{\epsilon^3\bra[\bigg]{\sum_{k \in K^{\delta_*} }w_{k}}^3 } + o(\epsilon^3) \right) \\
\textrm{and} \qquad
\frac{\kappa_\sharp}{2}\mathcal{E}(\varrho,\varrho)&\leq\frac{\kappa_\sharp}{2}\mathcal{E}(\varrho_\infty,\varrho_\infty)
+ \frac{\kappa_\sharp\epsilon^2 |K^{\delta_*}| \varrho_\infty^2}{2} \min_{k \in \N^d \setminus\set{\mathbf{0}}} \frac{\tilde{W}(k)}{\Theta(k)} L^{d/2}
+\frac{\kappa_\sharp\epsilon^2 |K^{\delta_*}| \delta_*}{2 L^{3d/2}} \, .
\end{align} 
Using the fact that $\kappa_\sharp \min\limits_{k \in \N^d\setminus\set{\mathbf{0}}}\frac{\tilde{W}(k)}{\Theta(k)}=-\beta^{-1}L^{d/2}\ $, we obtain,
\begin{align}
\crF_{\kappa_\sharp}(\varrho)&\leq\crF_{\kappa_\sharp}(\varrho_\infty) -\frac{\epsilon^3 \varrho_\infty}{3 \beta}\intom{\bra[\bigg]{\sum\limits_{k \in K^{\delta_*} }w_{k}}^3 } - \frac{\epsilon^2 \delta_* \varrho_\infty |K^{\delta_*}| }{2\beta} \bra*{\min\limits_{k \in \N^d \setminus\set{\mathbf{0}}} \frac{\tilde{W}(k)}{\Theta(k)}}^{-1} + o(\epsilon^3) \, .
\end{align}
Setting $\epsilon=\delta_*^{\frac{1}{2}}$ (if $\delta_* >0$, otherwise we stop here), we obtain
\begin{align}
\crF_{\kappa_\sharp}(\varrho)&\leq\crF_{\kappa_\sharp}(\varrho_\infty) -\frac{\delta_*^{\frac{3}{2}}\varrho_\infty}{3 \beta}\intom{\bra[\bigg]{\sum\limits_{k \in K^{\delta_*} }w_{k}}^3 } + \frac{\delta_*^{2}\varrho_\infty |K^{\delta_*}| }{2\beta} \abs*{\min\limits_{k \in \N^d \setminus\set{\mathbf{0}}} \frac{\tilde{W}(k)}{\Theta(k)}}^{-1} + o(\delta_*^{\frac{3}{2}})  \, .
\label{eq:defup}
\end{align} 
One can now check that under condition~\eqref{eq:c1}, it holds that
\[
  \intom{
\bra[\bigg]{\sum\limits_{k \in K^{\delta_*}}w_{k}}^3 } > a>0 \, ,
\]
 where the constant $a$ is independent of $\delta_*$. Indeed, the cube of the sum of $n$ numbers $a_i$, $i=1, \dots, n$ consists of only three types of terms, namely: $a_i^3$, $a_i^2 a_j$ and $a_i a_j a_k$. Setting the $a_i=w_{s(i)}$, with $s(i) \in K^{\delta_*}$, one can check that the first type of term will always integrate to zero.  The other two will take nonzero and in fact positive values if and only if condition~\eqref{eq:c1} is satisfied.  This follows from the fact that 
 \[
 \int_{-\pi}^{\pi} \,\cos({\ell x})\cos(m x) \cos(nx) \! \dx{x}= \frac{\pi}{2}\bra*{\delta_{\ell+m,n}+\delta_{m+n,\ell} + \delta_{n+\ell,m}}\, .
\]
 Thus, for $\delta_*$ sufficiently small considering the fact that $|K^{\delta_*}| \geq 2$ and is nonincreasing as $\delta_*$ decreases, $\varrho$ has smaller free energy and $\varrho_\infty$ is not a minimiser at $\kappa=\kappa_\sharp$.
\end{proof}
\begin{rem}
The case of the above result for $\delta_*=0$ can be thought of as the pure resonance case. In this case the set $K^0$ will denote the set of all resonant modes. Similarly, the above result for $\delta_*$ small but positive can be thought of as the near resonance case.
\end{rem}
The corollary below tells us that if we have a have a sequence of potentials whose Fourier modes grow closer to each other then it will eventually  have a discontinuous transition point, as long as the potentials do not lose mass too fast.
\begin{cor}~\label{cor:gamma}
Let $\{W^n\}_{n \in \N} \in \HH_\stable^c$ be a sequence of interaction potentials such that $\delta_*(n) \to 0$ as $n \to \infty$,
where $\delta_*$ is as defined in~\cref{def:del}. Assume further that  for all $n$ greater than some $N \in \N$, there exists a constant $C>0$ such that $\abs*{\min\limits_{k \in \N^d\setminus\set{\mathbf{0}}} \frac{\tilde{W^n}(k)}{\Theta(k)}} \geq C \delta_*(n)^{ \gamma}$ for some $\gamma<1/2$. Then for $n$ sufficiently large, the associated free energy $\cF^n_\kappa(\varrho)$ possesses a discontinuous transition point at some $\kappa_c^n < \kappa_\sharp^n$.
\end{cor}
\begin{proof}
We return to estimate~\eqref{eq:defup} from the proof of~\cref{thm:dctp}
\begin{align}
\crF^n_{\kappa_\sharp}(\varrho)&\leq\crF^n_{\kappa_\sharp}(\varrho_\infty) -\frac{\delta_*^{\frac{3}{2}}\varrho_\infty}{3 \beta}\intom{\bra[\bigg]{\sum\limits_{k \in K^{\delta_*} }w_{k}}^3 } + \frac{\delta_*^{2}\varrho_\infty |K^{\delta_*}| }{2\beta} \abs*{\min\limits_{k \in \N^d\setminus\set{\mathbf{0}} } \frac{\tilde{W^n}(k)}{\Theta(k)}}^{-1} + o(\delta_*^{\frac{3}{2}})  \, ,
\end{align} 
 where we have suppressed the dependence of $\delta_*$ on $n$. We also note that the error term is independent of the potential $W^n$. Using our assumption on the potential (for $n>N$), we have
 \begin{align}
\crF^n_{\kappa_\sharp}(\varrho)&<\crF^n_{\kappa_\sharp}(\varrho_\infty) -\frac{\delta_*^{\frac{3}{2}}\varrho_\infty}{3 \beta}\intom{\bra[\bigg]{\sum\limits_{k \in K^{\delta_*} }w_{k}}^3 } + \frac{\delta_*^{2- \gamma}\varrho_\infty |K^{\delta_*}| }{2\beta}  + o(\delta_*^{\frac{3}{2}})  \, .
\end{align}
Since $\gamma<1/2$ and $\delta_* \to 0 $ as $n \to \infty$, the result follows.
\end{proof}
To conclude our discussion of discontinuous transition points, we present the following corollary to provide some more intuition of  the types of interaction potentials that exhibit a discontinuous transition point.
\begin{cor}~\label{cor:delta}
Let $\{W^n\}_{n \in \N} $ be a sequence of interaction potentials with $ \norm{W^n}_1 =C>0$  for all $ n \in \N$ such that $W^n\to -C \delta_0$  in the sense of distributions as $n \to \infty$. Then for $n$ large enough, the associated free energy $\cF^n_\kappa(\varrho)$ possesses a discontinuous transition point at some $\kappa_c^n < \kappa_\sharp^n$. 
\end{cor}
\begin{proof}
Note first that we have not included the assumption $W^n \in \HH_\stable^c$ as eventually this must be the case if the potentials
converge to a negative Dirac measure. Now we just need to check that the other conditions of~\cref{cor:gamma} hold true. We have the following estimate
\begin{align}
\tilde{W^n}(k) \geq -C N_k  \implies \frac{\tilde{W^n}(k)}{\Theta(k)} \geq  -C L^{-d/2} \, \label{eq:minWb},
\end{align} 
for all $k \in \N^d \setminus \set{\mathbf{0}}$. From the convergence to the Dirac measure it follows that for 
any $\epsilon>0$ we can find an $N$ large enough such that $\frac{\tilde{W^n}(k)}{\Theta(k)},\frac{\tilde{W^n}(2 k)}{\Theta(2 k)} \in \bra*{-C L^{-d/2},-C L^{-d/2}+\epsilon}$ for all $n >N$, for some $k \in \N^d \setminus \set{\mathbf{0}}$. This and~\eqref{eq:minWb} tells us that $\delta_* \leq \epsilon$ and since $\epsilon $ is arbitrary $\delta_* \to 0$ as $n \to \infty$. From similar arguments we assert that for all $n>N$, $\bra[\bigg]{\min\limits_{k \in \N^d\setminus\set{\mathbf{0}}} \frac{\tilde{W^n}(k)}{\Theta(k)}}< -C L^{-d/2}+\epsilon$. Thus we have that $\abs*{\bra[\bigg]{\min\limits_{k \in \N^d\setminus\set{\mathbf{0}}} \frac{\tilde{W^n}(k)}{\Theta(k)}}}> C \frac{2^{d/2}}{L^{d/2}}-\epsilon$ for $n >N$. Since the conditions of~\cref{cor:gamma}  are satisfied, we have the desired result.
\end{proof}
\begin{rem}
% The restriction~\eqref{eq:ftlb} may seem artificial but it is automatically satisfied by all potentials with $W(x) \leq 0$ for all $x \in V$. It is precisely this assumption that enables us to separate the rate of decay of $\delta_*$ from the rate of decay(or growth) of the minimum of the Fourier modes. 
% 
As examples of potentials that satisfy the conditions of~\cref{cor:delta}
, we have the negative Dirichlet kernel $W^n(x)=-1 -2 \sum_{k=1}^{n} w_k(x)$, the negative F\'ejer kernel
$W^n(x)=-\frac{1}{n}\bra*{\frac{1- w_n(x)}{1- w_1(x)}}$, and any appropriately scaled negative mollifier. 
\end{rem}
\subsection{Continuous transition points}
We now present a couple of technical lemmas starting with a functional inequality that gives a bound on the defect in the Gibbs inequality from below by the size of individual Fourier modes. These will be useful for the characterisation of continuous transition points provided in~\cref{thm:spgap} and, in particular, in the proof of \cref{thm:m3}\ref{thm:m3b}.
\begin{lem} \label{lem:entdef}
Let $(\Omega, \Sigma,\mu)$ be a probability space and $\left\{  w_k\right\}_{k\in \N}$ be any orthonormal basis for $\Leb^2(\Omega,\mu)$. Assume that $f\in \Leb^2(\Omega, \mu)$ is a probability density with respect to $\mu$, that is $f$ is nonegative and $\int f \dx{\mu}=1$, then we have, for any $b\in \R$ and any $k\in \Z$, the following estimate,
\begin{align}\label{RelEnt:L2:bound}
\cH(f \mu | \mu ) \geq  -\log  \int_{\Omega}  \exp\bra[\big]{ b \skp{f,w_k }_\mu w_k(x)}\dx{\mu}  + b \abs{\skp{
f,w_k}_\mu}^2     \ ,
\end{align}
In particular, let $\Omega= U$, $\mu=\varrho_\infty$ and $w_k$ is as defined in~\eqref{e:def:wk}. Moreover, for any $k \in \Z^d\setminus\set{\mathbf{0}}$ let $n=n(k) = \abs*{\set{i: k_i \ne 0}}$ denote the number of nonzero entries. Then, there exists a strictly increasing function $\cG: \R^+\to \R^+$ with $\cG(0)=0$ such that it holds
\begin{align}\label{H:bound}
\mathcal{H}(\varrho| \varrho_\infty)- C(n(k)) \tfrac{L^d }{2}|\tilde{\varrho}(k)|^2 \geq \cG(|\tilde{\varrho}(k)|) \ ,
\end{align}
where the constant $C(n) > 0$ for is given by $C(1)=C(2)=1$ and for $n>2$ by
\begin{equation}
  C(n) = \frac{(n/2)^n}{(n-1)^{n-1}} < 1 \, . 
\end{equation}
\end{lem}
%
% \begin{rem}
% Since $1 \leq n \leq d$, for $d \leq 4$ the above result holds for all $k \in \Z^d$ not identically zero.
% \end{rem}
\begin{defn} \label{defn:astable}
Assume that $W \in \HH_\stable^c $ has one dominant negative mode, i.e., there exists a unique $ k^\sharp \in \N^d$ such that $\frac{\tilde{W}(k^\sharp)}{\Theta(k^\sharp)}=\min_{k \in \N^d}\frac{\tilde{W}(k)}{\Theta(k)}$(as defined in~\eqref{eq:poc}). We define the $\alpha$-stabilised potential $W_\alpha(x)$ as follows
\begin{align}
W_\alpha(x)=\skp*{W,w_{k^\sharp}}w_{k^\sharp}(x)+\alpha(W_\unstable(x)- \skp*{W,w_{k^\sharp}}w_{k^\sharp}(x))  + W_\stable(x)   \, ,
\end{align}
 where $\alpha \in [0,1]$, $W_\stable(x), W_\unstable(x)$ are as defined in~\cref{def:Hstable}, and $W_1(x)=W(x)$.  
\end{defn}
The above definition puts into context the discussion around Figure~\ref{fig:dcctp}(a) in~\autoref{S:intro}, i.e., the $\alpha$-stabilised potential $W_\alpha$ pushes all negative modes except the dominant one to some small
neighbourhood of $0$. We define the fixed point equation associated with the interaction potential $W_\alpha$ to be
\begin{align}
F_\kappa(\varrho,\alpha)= \varrho(x)- \frac{1}{Z}e^{-\beta \kappa W_\alpha \star \varrho} \,.
\end{align}
\begin{lem}\label{lem:lbfour}
Let $W_\alpha(x)$ be as in \cref{defn:astable} and let $\cC \subset \Prac^+(U)$ denote the set of nontrivial solutions of $F_{\kappa_\sharp}(\varrho,\alpha)=0$ for $\alpha \in [0,\alpha^*) \subset [0,1]$. Then, for $\alpha^*$ sufficiently small, we have the uniform lower bound $\sum\limits_{\sigma \in \Sym(\Lambda)}|\tilde{\varrho}(\sigma(k^\sharp))|^2 >c$ for all $\varrho \in \cC$ and for some $c>0$ independent of $\alpha \in [0,\alpha^*)$.
\end{lem}
We are now in the position to give the precise statement of \cref{thm:m3}\ref{thm:m3b} and prove it. We present the proofs of~\cref{lem:entdef} and~\cref{lem:lbfour} after the proof of~\cref{thm:spgap}.
\begin{thm}\label{thm:spgap}
Let $W_\alpha(x)$ be as in \cref{defn:astable} such that $ \Theta(k^\sharp) \leq 2$ where $\Theta(k)$ is as defined in~\eqref{e:def:thetak}. Assume further that $W_\unstable$ and $W_\stable$ are bounded below. Then, for $\alpha$ sufficiently small, the system exhibits a continuous transition point at $\kappa_c=\kappa_\sharp$.
\end{thm}
\begin{proof}
By~\cref{prop:CharactTP}~\ref{prop:CharactTP:cont}, it is sufficient to show that at the point of
critical stability $\kappa_\sharp$, i.e.,
\begin{align}
 \kappa_\sharp=\kappa_c=
-\frac{L^{\frac{d}{2}}\Theta(k)}{\beta \tilde{W_\alpha}(k^\sharp)}=-\frac{L^{\frac{d}{2}}\Theta(k)}{\beta \tilde{W}(k^\sharp)} \, ,
\end{align}
 the uniform state $\varrho_\infty$ is the unique minimiser, for $\alpha$ small enough. Let $\varrho$ be any solution of $F_{\kappa_\sharp}(\varrho,\alpha)=0$, i.e., a critical point of $\crF_{\kappa_\sharp}$ (cf.~\cref{prop:tfae}). Then we have
\begin{align}
\crF(\varrho)-\crF(\varrho_\infty)&= \beta^{-1}\mathcal{H(\varrho|\varrho_\infty)} + \frac{\kappa_\sharp}{2}\mathcal{E}(\varrho-\varrho_\infty,\varrho-\varrho_\infty) \\
&= \beta^{-1}\mathcal{H(\varrho|\varrho_\infty)}  + \frac{\kappa_\sharp}{2} L^{d/2} \frac{\tilde{W}(k^\sharp)}{\Theta(k^\sharp)}\left(\sum\limits_{\sigma \in \Sym(\Lambda)}|\tilde{\varrho}(\sigma(k^\sharp))|^2 \right) \\ \nonumber
&\qquad + \frac{\kappa_\sharp}{2} L^{d/2} \sum\limits_{k \in \N^d, k \neq k^\sharp }\frac{\tilde{W_\alpha}(k)}{\Theta(k)}\left(\sum\limits_{\sigma \in \Sym(\Lambda)}|\tilde{\varrho}(\sigma(k))|^2 \right) \, .
\end{align}
We can translate $\varrho$ w.l.o.g so that $\varrho(\sigma(k^\sharp))=0, \forall \sigma \in (\Sym(\Lambda)-e)$ and throw away all positive $\tilde{W_\alpha}(k)$. A consequence of this is that $|\tilde{\varrho}(k^\sharp)|^2 = \sum\limits_{\sigma \in \Sym(\Lambda)}|\tilde{\varrho}(\sigma(k^\sharp))|^2$. Thus we obtain
\begin{align}
\crF(\varrho)-\crF(\varrho_\infty)&\geq \beta^{-1}\left(\mathcal{H(\varrho|\varrho_\infty)}  - \frac{L^d}{2}|\tilde{\varrho}(k^\sharp)|^2\right) \\
&\qquad +\frac{\beta^{-1}L^d}{2} \sum\limits_{k \in \N^d, k \neq k^\sharp }\bra*{\frac{\tilde{W_\alpha}(k)\Theta(k^\sharp)}{\Theta(k)\tilde{W}(k^\sharp)}}_-\left(\sum\limits_{ \sigma \in \Sym(\Lambda)} |\tilde{\varrho}(\sigma(k))|^2 \right) \, . 
\end{align}
Since $\tilde{W}_\alpha(k)=\alpha\tilde{W}(k) $ for all $k \in \N^d, k \neq k^\sharp$ with $\tilde{W}(k)<0$ and
by definition $\tilde{W}(k)/\Theta(k) \geq \tilde{W}(k^\sharp)/\Theta(k^\sharp)$, we can obtain the estimate
\begin{align}
\crF(\varrho)-\crF(\varrho_\infty)&\geq  \beta^{-1} \left(\mathcal{H(\varrho|\varrho_\infty)}  - \frac{ L^d}{2}|\tilde{\varrho}(k^\sharp)|^2\right) -\frac{ \alpha\beta^{-1} L^d}{2} \sum\limits_{k \in \N^d, k \neq k^\sharp}\left(\sum\limits_{\sigma \in \Sym(\Lambda)}|\tilde{\varrho}(\sigma(k))|^2 \right) \, .
\end{align}
We apply \cref{lem:entdef} to the first term on the right hand side
\begin{align}
  \crF(\varrho)-\crF(\varrho_\infty)&> \beta^{-1} \bra*{ \cG(|\tilde\varrho(k^\sharp)|) - \frac{\alpha L^d}{2} \norm*{\varrho}_2^2} \, . 
\end{align}
Here, we use that the fact that the assumption that $ \Theta(k^\sharp) \leq 2$ is equivalent to $n(k^\sharp)\leq 2$, where $n(k^\sharp)$ is the number of nonzero components in $k^\sharp$ as defined in the statement of~\cref{lem:lbfour}. Now, we use the result of~\cref{lem:entdef} with the constant $c$ and the monotonicity of the function $\cG$ to further estimate
\begin{align}
\crF(\varrho)-\crF(\varrho_\infty)&> \beta^{-1}\left( \cG(c) -\frac{ \alpha L^d}{2}\norm*{\varrho}_2^2\right) \, ,
\end{align}
where $c$ is precisely the constant from \cref{lem:lbfour} for $\alpha \in [0,\alpha^*)$.
Since $\varrho$ is a zero of $F_{\kappa_\sharp}(\varrho,\alpha)=0$ we have the following estimate
\begin{align}
\norm*{\varrho}_2^2 \leq \norm*{\varrho}_\infty \stackrel{\eqref{eq:linfT}}{\leq} \exp \bra*{\beta \kappa \bra*{\norm*{W_{\alpha-}}_\infty + \norm*{W_{\alpha}}_1 }} \leq \exp \bra*{\beta \kappa \bra*{\norm*{W_{\alpha-}}_\infty + L^{-d}\norm*{W_{\alpha}}_2 }} \, .
\end{align}
If we restrict $\alpha$ to $[0,\alpha^*)$ as in~\cref{lem:lbfour}, we can obtain the following estimates on the norms of $W_\alpha$:
\begin{align}
\norm*{W_{\alpha -}}_\infty  &\leq \norm*{W_{\stable-}}_\infty + \norm*{W_{\unstable-}}_\infty +(\alpha +1 )|\tilde{W}(k^\sharp)| \\
& \leq \norm*{W_{\stable -}}_\infty + \norm*{W_{\unstable -}}_\infty +(\alpha^* +1 )|\tilde{W}(k^\sharp)| \,, \\
\text{and}\qquad 
\norm*{W_{\alpha}}_2^2 &= \norm*{W_\stable}_2^2 + \alpha^2\norm*{W_\unstable}_2^2 + (1-\alpha)^2 |\tilde{W}(k^\sharp)|^2 \\
&\leq \norm*{W_\stable}_2^2 + (\alpha^*)^2\norm*{W_\unstable}_2^2 +  |\tilde{W}(k^\sharp)|^2 \, .
\end{align}
Thus for $\alpha \in [0,\alpha^*)$ we have $\norm*{\varrho}_2^2<c_1$ for some positive constant $c_1$ independent of $\alpha$.  
Thus, for $\alpha < \frac{2\cG(c)}{L^dc_1}$, the result holds.
\end{proof}
\begin{proof}[Proof of~\cref{lem:entdef}]
Using its Fenchel dual, the relative entropy has the following formulation 
% is defined as follows,
% \begin{align}
% \Ent_\mu(f)= \int\limits_\Omega f \log f \dx{\mu}
% \end{align}
% We then have the following dual formulation. The proper convex function $\phi(r)=r \log r $ has the Fenchel dual
% %
% \begin{align}
% \phi^*(r)&= \sup\limits_{x \in \R} \{x r - \phi(x)\} = e^{r-1} .
% \end{align}
% %
% Thus, we have,
% %
% \begin{align}
% \int f g \dx\mu & \leq \int \phi(f) \dx\mu + \int \phi^*(g) \dx\mu = \Ent_\mu(\varrho) + \int e^{g-1} \dx\mu \ ,
% \end{align}
% where, $g \in C_c^0(\Omega; \R)$. Replacing $g$ by $g-1$, we obtain
% \begin{align}
% \Ent_\mu(\varrho) \geq \int f g \dx{\mu} + \int  \bra*{ e^{g}-1 } \dx{\mu}
% \end{align}
% Therefore, we can redefine the relative entropy as follows,
\begin{align}\label{eq:RelEntDual}
\cH(f \mu | \mu) = \sup_{g \in \Leb^2(\Omega,\mu)} \set*{ \int f g \dx\mu
:  \int e^{g} \dx\mu \leq 1 } \, . 
\end{align}
From here a lower bound is obtained by choosing, for $b\in \R$ arbitrary,
\[
 g(x) =b \skp{f,w_k}_\mu w_k(x) - \log \int \exp\bra[\big]{b \skp{f,w_k}_\mu w_k(x)} \dx{\mu}.
\]
It is easy to check that $\int e^{g} \dx\mu = 1$ and hence $g$ is admissible in~\eqref{eq:RelEntDual}. The estimate~\eqref{RelEnt:L2:bound} follows by plugging this specific choice of $g$ into~\eqref{eq:RelEntDual}
% \begin{align}
% \cH(f\mu | \mu)  \geq a + b |(\varrho,w_k)_\mu|^2, \qquad\text{ provided that }\qquad -\log \int\limits_\Omega \! \left(e^{b (\varrho,w_k)_\mu w_k(x)} \right) \dx{\mu} \leq a
% \end{align}
% Thus we have by eliminating $a$ in the lower bound
\begin{align}\label{RelEnt:L2:bound:p}
\cH(f \mu | \mu) & \geq -\log \int \exp\bra[\big]{b \skp{f,w_k}_\mu w_k(x)} \dx{\mu} + b |\skp{f,w_k}_\mu|^2 \,.
\end{align}
In the special case $\Omega=U$ and $\mu=\varrho_\infty$, setting $f= \frac{\varrho}{\varrho_\infty}$,  we obtain from~\eqref{RelEnt:L2:bound:p} the lower bound
\begin{align}
\mathcal{H}(\varrho|\varrho_\infty) \geq -\log \int \exp\bra*{b \tilde{\varrho}(k)w_k(x)} \varrho_{\infty}\dx{x} +  b |\tilde{\varrho}(k)|^2  \, . 
\end{align}
We can pick $b=\alpha L^d$ for some $\alpha>0$ and set $y=L^{d/2} 2^{n/2} \tilde{\varrho}(k)$. We thus obtain,
\begin{align}
\mathcal{H}(\varrho|\varrho_\infty) \geq \frac{\alpha y^2 }{2^n}  - \log\bra[\Bigg]{\varrho_\infty \intom{e^{\alpha y \prod_{i=1}^{n} \cos(2 \pi k_i x_i /L) }} } \ ,
\end{align}
where the $w_{k_i}(x_i)$ are as defined previously and $n \geq 1$ represents the number of $k_i \neq 0$.
 Setting $x_i=\frac{L}{2 \pi k_i} \theta_i$ for all $k_i \neq 0$,  we arrive at
\begin{equation}
\mathcal{H}(\varrho|\varrho_\infty) \geq \frac{\alpha y^2 }{2^n}  - \log \bra*{\frac{1}{2^n \pi^n} \int_{[0,2 \pi]^n}
\! \exp\bra*{\alpha y \prod_{i=1}^{n} \cos (\theta_i)} \prod\limits_{j=1}^{n} \dx{\theta_j}}
 \, . \label{H:bound:p0} 
\end{equation}
We introduce the function
\begin{align}
\mathcal{I}_n(z)&=\frac{1}{2^n \pi^n} \int_{[0,2 \pi]^n}
 \exp\bra*{z \prod_{i=1}^{n} \cos (\theta_i)} \prod\limits_{j=1}^{n} \dx{\theta_j} = \sum_{l=0}^\infty \frac{z^{2l}}{(2l)!} \bra*{\frac{1}{\pi} \int_0^{\pi} \cos(\theta)^{2l} \dx{\theta}}^n \\
 &= \sum_{l=0}^\infty z^{2l} \frac{((2l)!)^{n-1}}{(l!)^{2n} 2^{2 l n} }  \, .
\end{align}
We will show that
\begin{equation}\label{e:mathcalI:lb}
 \tilde\cG(z) =\frac{\lambda z^2}{2^{n+1}} - \log \mathcal{I}_n(z)  \qquad\text{with}\qquad \lambda = \lambda(n) =
 \begin{cases}
  1 &, n\in \set{1,2} \\
  \frac{(n-1)^{n-1}}{(n/2)^n} &, n >2
 \end{cases} \, , 
\end{equation}
is strictly increasing in $z$ with $\tilde\cG(0)=0$.
Once we have shown~\eqref{e:mathcalI:lb}, the proof concludes by combining this with~\eqref{H:bound:p0} to
deduce that
\[
  \mathcal{H}(\varrho|\varrho_\infty) -  \tilde\cG(\alpha y) \geq \frac{\alpha y^2 }{2^n} - \frac{\lambda \alpha^2 y^2}{2^{n+1}} = \bra*{2- \alpha \lambda}\alpha \frac{y^2}{2^{n+1}} \stackrel{\alpha=\lambda^{-1}}{=} \frac{y^2}{\lambda \, 2^{n+1}} ,
\]
from where the result~\eqref{H:bound} follows by setting $\cG(y) = \tilde\cG(y/\lambda)$. 

It is left now to show~\eqref{e:mathcalI:lb}. %, which we rewrite as
% \begin{equation*}
%  \frac{\exp\bra*{\frac{\lambda z^2}{2^{n+1}}}}{\log \mathcal{I}_n(z)}  \stackrel{!}{\geq} 1 . 
% \end{equation*}
For its validity, it is sufficient to note that $\cI_n(0)=1$ and to show that $\exp\bra*{\lambda z^2/(\lambda 2^{n+1})}/\mathcal{I}_n(z)$  is strictly increasing in $z$. A sufficient condition for the monotonicity of this quotient is that quotient of the coefficients of the individual power series expansion of numerator and denominator are also increasing (cf.~\cite[Theorem 4.4]{Heikkala2009}, ~\cite{Biernacki1955}). First of all, we observe that the odd coefficients are zero. We are left to investigate 
\begin{align}
 \frac{\bra*{\exp\bra*{\lambda z^2/2^{n+1}}}_{2l}}{(\mathcal{I}_n(z))_{2l}} \ &= \ \frac{(l!)^{2n} 2^{2ln} \lambda^{l}}{((2l)! )^{n-1} 2^{(n+1)l} l!} = \frac{(l!)^{2n-1} 2^{l(n-1)}\lambda^{l}}{((2l)!)^{n-1}}\\
 &= \begin{cases}
    \hspace{.5in}l! &, n= 1  \\
   \bra*{ \frac{(l!)^{1+\frac{n}{n-1}} 2^l \lambda^{l/(n-1)}}{(2l)!}}^{n-1} =: (a_l)^{n-1} &, n>1
  \end{cases}\, .
\end{align}
In the case $n=1$, the monotonicity follows by the above representation. For $n>1$, we consider
\begin{align}
 \frac{a_{l+1}}{a_l} = \frac{\lambda^{1/(n-1)} (l+1)^{1+\frac{n}{n-1}} 2}{(2l+2)(2l+1)} = \frac{\lambda^{1/(n-1)} (l+1)^{\frac{n}{n-1}}}{2l+1}.
\end{align}
We need to find a $\lambda$ such that the above expression is greater than or equal to $1$.
Hence, we obtain
\[
  \lambda^{1/(n-1)} = \sup_{l\geq 1} \frac{2l+1}{(l+1)^{\frac{n}{n-1}}} = \frac{n-1}{\bra*{n/2}^{\frac{n}{n-1}}},
\]
where we note that the $\sup$ is attained for $l = \frac{n-2}{2}$, hence proving~\eqref{e:mathcalI:lb}.
% \textbf{Old version of proof below!}

% Thus, we obtain,
% \begin{align}
% \frac{(\mathcal{I}_n(y))_l}{(e^{y^2/2^{n+1}})_l}=2^{l(n-1)}\frac{2^{2l}}{(2l)!} \left(\frac{1}{\pi}\int\limits_{0}^{\pi}
% \! \cos^{2l} \thet \dx{\theta} \right)^n  \ ,
% \end{align}
% where the subscript, $l \in \mathbb{Z}_{\geq 0}$ represents the coefficient of the $l$\textsuperscript{th} term in the power series expansion. We can now assert that the above ratio is strictly decreasing in $l$. Indeed,
% \begin{align}
% \frac{(\mathcal{I}_n(y))_l}{(e^{y^2/2^{n+1}})_l}&=2^{l(n-1)}\frac{2^{2l}}{(2l)!} \left(\frac{(2l)!}{(l!)^2 2^{2l}}\right)^n
% = \frac{1}{(l!)^2}\left(  \frac{(2l)!}{(l!)^2 2^{l}} \right)^{n-1}
% \end{align}
% For $n=1$, the result follows in a straightforward manner. For $n>1$ we rewrite the above expression as,
% \begin{align}
% \frac{(\mathcal{I}_n(y))_l}{(e^{y^2/2^{n+1}})_l}= \left(  \frac{(2l)!}{(l!)^{2 +\frac{2}{n-1}} 2^{l}} \right)^{n-1}
% \end{align}
% For the term under the exponent we have the following recurrence relation,
% \begin{align}
% a_{l+1}= \frac{2l+1}{(l+1)^{1 + \frac{2}{n-1}}} a_l
% \end{align}
% One can now check that $2l+1 < (l+1)^{1 + \frac{2}{n-1}}, \forall l\geq 1$ and $n \leq 4$ . This implies (cf. Lemma 2.1, ~\cite{baricz2010boundsv}) that $\frac{(\mathcal{I}(y))}{(e^{y^2/2^{n+1}})}$ is strictly decreasing in $y$, which gives us the desired result.
\end{proof}
\begin{proof}[Proof of~\cref{lem:lbfour}]
For the first part of the proof, we fix $\alpha \in [0,\alpha^*)$. Then, we know that $\kappa=\kappa_\sharp$ independent of $\alpha$ is a bifurcation point, i.e., it satisfies the conditions of \cref{thm:c1bif}. Then one can check that the same set of arguments can be applied in the larger space~$\Leb^2_{k^\sharp}(U)$ instead of $\Leb^2_s(U)$, where $\Leb^2_{k^\sharp}=\{f \in \Leb^2(U): \skp*{f,w_{\sigma(k^\sharp)}}=0, \forall \sigma \in \Sym(\Lambda), \sigma \neq e \}$, where $e$ represents the identity element. For fixed $\alpha$, we consider the map, $\overline{F} : \Leb^2_{k^\sharp}(U)  \times \R^+ \to \Leb^2(U)$, $(\varrho ,\kappa) \mapsto F_\kappa(\varrho,\alpha)$ and note that any $\varrho$ such that $\overline{F}(\varrho,\kappa)=0$ is obviously in $\Leb^2_{k^\sharp}(U)$. Additionally, any zero of $\overline{F}$ defined above is also a zero of $F^* :\Leb^2_{k^\sharp}(U) \times \R^+ \to \Leb^2_{k^\sharp}(U)$, which is defined as
\begin{align}
F^*(\varrho,\kappa)=\overline{F}(\varrho,\kappa) - \sum\limits_{\sigma \in \Sym(\Lambda),\sigma\neq e}\skp*{\overline{F}(\varrho,\kappa), w_{\sigma(k)}(x)}w_{\sigma(k)}(x)\, .
\end{align}
One can also notice that $F^*(\varrho)$ does not change any of the local properties of $\overline{F}(\varrho)$ near $\varrho_\infty$, i.e,
$D_\varrho F^*(\varrho_\infty,\kappa)=\left.D_\varrho \overline{F}(\varrho_\infty,\kappa)\right|_{\Leb^2_{k^\sharp}}$ and 
$D^2_{\varrho\kappa} F^*(\varrho_\infty,\kappa)=\left.D^2_{\varrho\kappa} \overline{F}(\varrho_\infty,\kappa)\right|_{\Leb^2_{k^\sharp}}$.
The advantage of defining $F^*$ in this way is that the Fr\'echet derivative of the map is then Fredholm with index zero, which is not the case with $\overline{F}$.  We also know from \cref{thm:c1bif} that $\overline{F}$ has at least one nontrivial solution $\varrho_\kappa \in \Leb^2_s(U)$ in a neighbourhood  of $(\varrho_\infty,\kappa_\sharp)$. We can now apply the same bifurcation argument to $F^*$ to obtain that $F^*$ has exactly one nontrivial solution in some neighbourhood of $(\varrho_\infty,\kappa_\sharp)$. Since every zero of $\overline{F}$ is a zero of $F^*$ it follows that $\varrho_\kappa$ is this nontrivial zero in some neighbourhood of $(\varrho_\infty,\kappa_\sharp)$ and that $\overline{F}$ has only one nontrivial solution in this neighbourhood. Thus the problem of studying bifurcations of $\overline{F}$ is reduced to that of studying bifurcations of $F^*$. This justifies our choice in~\autoref{S:lbt} to study the bifurcations of $\hat{F}$ in the space $\Leb^2_s(U)$ as all bifurcations from the trivial branch lie either in this space or its translates. 

Now, since we need a lower bound which is uniform in $\alpha$, we redefine $F^*$ to be a function of $\alpha$, i.e.,
$F^*: X \times \R^+ \to \Leb^2_{k^\sharp}(U)$, where $X:= \Leb^2_{k^\sharp}(U) \times \R $  is Banach space equipped with the norm $\norm*{\cdot}_2 + |\cdot|$ and $f=(\varrho,\alpha) \in X$ a typical element of the space. We will now show that due to the particular structure of the problem one can still apply a Crandall--Rabinowitz type argument and obtain existence of local 
bifurcations. What follows below is a description of the Lyapunov--Schmidt decomposition for the map $F^*$ and a slightly modified version of the proof of the Crandall--Rabinowitz theorem as presented in~\cite{kielhofer2006bifurcation}.

We recentre the map as in the proof of~\cref{thm:c1bif} and linearise the map $F^*$ about $((0,0),\kappa_\sharp)$. We also note that $F^*((0,\alpha),\kappa)=0$, for all $\kappa \in (0,\infty),\alpha \in[0,\alpha*)$ and it is precisely this fact that will help us apply a Crandall--Rabinowitz type argument. Before we start out analysis, we write out the exact form of $F^*$ for the convenience of the readers
\begin{align}
F^*(f,\kappa)= \varrho(x) +\varrho_\infty - \frac{1}{Z}e^{-\beta \kappa W_\alpha \star \varrho} - \sum\limits_{\sigma \in \Sym(\Lambda),\sigma\neq e}\skp*{\varrho(x) - \frac{1}{Z}e^{-\beta \kappa W_\alpha \star \varrho} , w_{\sigma(k^\sharp)}(x)}w_{\sigma(k^\sharp)}(x) \, .
\end{align} 
It is clear that $D_fF^*(f,\kappa)= \begin{pmatrix}D_\varrho F^* & D_\alpha F^* \end{pmatrix} \in L(X,L^2_{k^\sharp})$, the space of linear operators from $X$ to $\Leb^2_{k^\sharp}(U)$, with
\begin{align}
D_\varrho F^*((0,0),\kappa_\sharp)[w_1]	&= w_1 + \beta \kappa_\sharp \varrho_\infty (W_0 \star w_1) - \beta \kappa_\sharp \varrho_\infty^2 \intT{(W_0 \star w_1)(x)} \, ,\\
D_\alpha F^*((0,0),\kappa_\sharp) &=0 \, ,
\end{align} 
where $w_1 \in \Leb^2_{k^\sharp}(U)$. We will also need $D^2_{f \kappa}F^*(f,\kappa)= \begin{pmatrix}D_{\varrho\kappa} F^* & D_{\alpha \kappa}F^* \end{pmatrix}$, with
\begin{align}
D^2_{\varrho \kappa} F^*((0,0),\kappa_\sharp)[w_1]	&= \varrho_\infty (W_0 \star w_1) - \varrho_\infty^2 \intT{(W_0 \star w_1) (x)} \\
&\qquad -\varrho_\infty^2 W_0 \star D_\varrho ( F^*((0,0),\kappa_\sharp))[w_1] \, , \label{eq:drks}\\
D^2_{\alpha \kappa} F^*((0,0),\kappa_\sharp) &=0 \, .
\end{align}
Then by using the arguments of \cref{thm:c1bif}, we see that $N:=\ker (D_f F^*((0,0),\kappa_\sharp))=\mathrm{span}[w_{k^\sharp}] \times \R \tilde{=} \R^2$ and $Z_0 := R^\perp= (\Ima (D_f F^*((0,0),\kappa_\sharp)))^\perp= \mathrm{span}[w_{k^\sharp}]$. Thus,
$D_f F^*((0,0),\kappa_\sharp))$ is Fredholm and we have the following decompositions into complementary subspaces,
\begin{align}
X&= N \oplus X_0 \, ,\\
\Leb^2_{k^\sharp}(U)&= R \oplus Z_0  \, .
\end{align}
Given these decompositions, we define the following projection operators,
\begin{align}
P&:X \to N , & (\varrho,\alpha) &\mapsto \bra*{\tilde{\varrho}(k^\sharp)w_{k^\sharp}(x),\alpha}\, ,\\
Q&:\Leb^2_{k^\sharp}(U) \to Z_0 \ , & \varrho &\mapsto  \tilde{\varrho}(k^\sharp) w_{k^\sharp}(x) \, .
\end{align}
By introducing the splitting $v=Pf$, $w= (I -P)f$, we can solve $F^*(f,\kappa)=0$ individually on complementary subspaces
\begin{align}
G(v,w,\kappa)&:=(I -Q) F^*(v + w,\kappa)=0 \, ,\\
\Phi(v,w,\kappa)&:=Q F^*(v + w,\kappa)=0 \, .
\end{align}
As in \cref{thm:lyapschmid}, one can check that $D_wG((0,0),(0,0),\kappa_\sharp)= (I-Q)D_fF^*((0,0),\kappa_\sharp):
X_0 \to R$ is a homeomorphism. Thus, applying the implicit function theorem, there exist neighbourhoods $U$ of $((0,0),\kappa_\sharp)$ in $N \times \R$ and $V$ of $(0,0)$ in $X_0$ along with  a $C^1$ function $\Psi:U \to V$ such 
that every solution of $G(v,w,\kappa)=0$ in $U\times V$ is of the form $(v,\kappa,\Psi(v,\kappa))$ with $\Psi((0,0),\kappa_\sharp)=
(0,0)$. Thus in $U$ we are left to solve,
\begin{align}
\Phi(v,\kappa)&:=Q F^*(v + \Psi(v,\kappa),\kappa)=0 \, .
\end{align}
It is also straightforward to show that $D_\kappa \Psi((0,0),\kappa)=0$. Indeed,
\begin{align}
D_\kappa(I-Q) F^*(v + \Psi(v,\kappa),\kappa)&=0\\
(I-Q) (D_\kappa F^*(v + \Psi(v,\kappa),\kappa) + D_\varrho F^*(v + \Psi(v,\kappa),\kappa)D_\kappa \Psi(v,\kappa) ) &=0
\end{align}
Setting $v=(0,0)$ and $\kappa=\kappa_\sharp$ one can see that $D_\kappa F^*((0,0),\kappa_\sharp)=0$ and since $(D_\kappa \Psi((0,0),\kappa),0) \in X_0$ which is complementary to $N$ giving $D_\kappa \Psi((0,0),\kappa_\sharp),0)=0$. Using an argument similar to the above one, one can show that $D_{\tilde{\varrho}(k^\sharp)} \Psi((0,0),\kappa_\sharp)=0 \in L(N,X_0)$.

 Since a typical element of $N$ can be represented by $(\tilde{\varrho}(k^\sharp),\alpha)=(s,\alpha)$ we proceed by rewriting $\Phi$ as follows,
\begin{align}
\tilde{\Phi}((s,\alpha),\kappa)= \int_0^1 \! \frac{d}{dt} \Phi((t s w_{k^\sharp} ,\alpha),\kappa)\,\dx{t}  = \int_0^1 \! D_s \Phi((t s w_{k^\sharp} ,\alpha),\kappa) w_{k^\sharp}\,\dx{t} \,,
\end{align}
where we have used the fact that $\Phi((0,\alpha),\kappa)=0$, since $\varrho_\infty$ is always a trivial solution. Now,
$\tilde{\Phi}:\R^2 \times \R \to \R$ is the map, which we analyse in the neighbourhood $U$ and nontrivial solutions correspond to $s \neq0$. Let $\hat{v}= (ts w_{k^\sharp},\alpha) \in N$, then we compute,
\begin{align}
D_\kappa D_s\Phi(\hat{v},\kappa)w_{k^\sharp}&=D_\kappa \bra*{Q D_\varrho F^*(\hat{v}+\Psi(\hat{v},\kappa),\kappa)(w_{k^\sharp} + D_s \Psi(\hat{v},\kappa))w_{k^\sharp}} \\
&= QD^2_{\varrho\varrho} F^*(\hat{v}+\Psi(\hat{v},\kappa),\kappa) [w_{k^\sharp} + D_s \Psi(\hat{v},\kappa))w_{k^\sharp},D_\kappa \Psi(\hat{v},\kappa)] \\
&\qquad + Q D_\varrho F^*(\hat{v}+\Psi(\hat{v},\kappa),\kappa)D^2_{\varrho\kappa} \Psi(\hat{v},\kappa) w_{k^\sharp}
\\
&\qquad + Q D^2_{\varrho\kappa} F^*(\hat{v}+\Psi(\hat{v},\kappa),\kappa) (w_{k^\sharp} + D_s \Psi(\hat{v},\kappa))w_{k^\sharp}) \, .
\end{align}
Setting $\hat{v}=(0,0)$ and $\kappa=\kappa_\sharp$, we see that the first term of the above expression is zero because $D_\kappa \Psi((0,0),\kappa)=0$ and the second term is zero because $Q$ maps the range of $ D_\varrho F^*((0,0),\kappa_\sharp)$ to zero. Noting that $D_s \Psi((0,0),\kappa_\sharp))=D_{\tilde{\varrho}(k^\sharp)} \Psi((0,0),\kappa_\sharp)=0$, we finally have
\begin{align}
\frac{\dx{}}{\dx{}\kappa}\tilde{\Phi}((0,0),\kappa_\sharp) &= Q D^2_{\varrho\kappa} F^*((0,0),\kappa_\sharp) w_{k^\sharp} \neq 0 \, .
\end{align}
Thus we can apply the implicit function theorem to obtain a function $C^1(V_1;V_2)$, $\varphi(s,\alpha)$  such that  $\tilde{\Phi}((s,\alpha),\varphi(s,\alpha))=0$, where $V_1$ and $V_2$ are neighbourhoods of $(0,0)$ and $\kappa_\sharp$ respectively and 
$V_1 \times V_2 \subset U$. Additionally in $V_1 \times V_2$ every solution of $\tilde{\Phi}$ (and hence $\Phi$) is of the form
$((s,\alpha),\varphi(s,\alpha))$ and $\varphi((0,\alpha))=\kappa_\sharp$. We know however from \cref{thm:c1bif}, that we could apply the same set of arguments for fixed $\alpha \in [0,1]$ to obtain single locally increasing branches which, at least for some small neighbourhood around $0$, must coincide with $\varphi(s,\alpha)$. Thus, we now know that for each $\alpha \in[0,1]$, we can find $\epsilon_\alpha>0$ such that $\varphi(s,\alpha)>\kappa_\sharp$ for $0<|s|<\epsilon_\alpha$. Now, let $\alpha \in [0,\alpha^*)=A$. If we show that $\inf_A \epsilon_\alpha =\epsilon'>0$ for $\alpha^*$ small enough, we can conclude the proof.  To see this, set $V_1'= V_1 \cap (-\epsilon',\epsilon')\times[0,\alpha^*)$ and observe that   $((s,\alpha),\varphi(s,\alpha))$ are the only
solutions in $V_1' \times V_2$ and $\varphi(s,\alpha)=\kappa_\sharp$ implies $(s,\alpha)=(0,\alpha)$. Thus in $V_1'$, $(0,\alpha)$ is the only solution of the bifurcation equation which would provide the desired result. Assume now that there exists no $\alpha^*$, such that $\inf_A \epsilon_\alpha >0$. It is straightforward to check that this would violate the continuity of $\varphi$ since $\epsilon_0>0$.
% We know then from \cref{thm:cr} that there exists a neighbourhood of $(\varrho_\infty,\kappa_\sharp)$ in $\textrm{span}[w_{k^\sharp}] \times \R^+$ such that every nontrivial solution of $\overline{F}(\varrho,\kappa_\sharp)=0$ in $\Leb^2_{k^\sharp}$ is of the form $(sw_{k^\sharp},\psi(s))$, where $\psi(s)$ is a $C^1$ function with $\psi(0)=\kappa_\sharp$. Thus there exists some positive constant $c_\alpha>0$, such that for,
% \begin{align}
% |\skp*{\varrho-\varrho_\infty,w_{k^\sharp}}| + |\kappa-\kappa_\sharp| <c,
% \end{align}
% all nontrivial solutions are of the described form. For $\kappa=\kappa_\sharp$, this implies that for $|\skp*{\varrho-\varrho_\infty,w_{k^\sharp}}|<c$, there are no nontrivial solutions in $\Leb^2_{k^\sharp}(U)$. Since any $\varrho \in \mathscr{C}$ can
% be placed in $\Leb^2_{k^\sharp}(U)$ by translation, the result follows.
\end{proof}
As an immediate consequence of~\cref{thm:spgap} we have:
\begin{cor}\label{cor:unest}
Let $W_\alpha(x)$ be as in \cref{defn:astable} such that  $W_\unstable$ and $W_\stable$ are bounded below. Then, for $\alpha$ sufficiently small, $\varrho_\infty$ is the unique minimiser of the free energy $\mathcal{F}_\kappa(\varrho)$ for 
$\kappa \in (0, C(n)\kappa_\sharp]$, where $C(n)$ is as defined in~\cref{lem:entdef}.
\end{cor}  
\begin{proof}
The proof follows the same arguments as~\cref{thm:spgap} with $\kappa_\sharp$ replaced by $C(n)\kappa_\sharp$.
\end{proof}
A natural question to ask now is how the estimate from~\cref{cor:unest} compares to the one obtained in~\cref{prop:con}
by the convexity argument, i.e., how does $C(n)\kappa_\sharp$ compare to $\kappa_{\mathrm{con}}$. It is easier to make this comparison whenever we can explicitly compute $\norm*{W_{\unstable-}}_\infty$. Assume, $W=W_0$, i.e, $W$ has only one negative mode, say $w_{k^\sharp}$, then we have
\begin{align}
\frac{C(n)\kappa_\sharp}{\kappa_{\mathrm{con}}}=2^{n}C(n)=
\begin{cases}
2^n & n=1,2 \\
\frac{n^n}{(n-1)^{n-1}} & n>2
\end{cases}\, ,
\end{align}
with $n=n(k^\sharp)$ as defined in~\cref{lem:entdef}.  Thus, for all $n\geq 1$, we have that $C(n)\kappa_\sharp>\kappa_{\mathrm{con}}$. From this we conclude that, for this choice of $W$, ~\cref{cor:unest} provides a sharper estimate on the range of $\kappa$ for which the uniform state is a unique minimiser of the free energy.
\begin{rem}
\cref{thm:spgap} indicates that if the linearised McKean--Vlasov operator $\cL$, has a sufficiently large spectral gap $\lambda$, then (assuming all other conditions are satisfied) the system exhibits a continuous transition point. Indeed, the spectral gap of $\cL:\Leb^2_0(U) \to \Leb^2_0(U)$ at $\kappa=\kappa_\sharp$ associated with
the interaction potential $W_\alpha$ can be computed as 
\begin{align}
\lambda=\min_{k \in \N^d, k \neq k^\sharp}\bra*{-\beta^{-1}\bra*{\frac{2 \pi |k|}{L}}^2 - \kappa_\sharp L^{-d/2} \bra*{\frac{2 \pi |k|}{L}}^2 \frac{\tilde{W_\alpha}(k)}{\Theta(k)} } ,
\end{align}
Let us assume that $|\lambda|>C_1$ for some constant $C_1>0$. This implies that for all $k \in \N^d$ such that
$\tilde{W}(k)<0$ it must hold 
\begin{align}
\alpha <\frac{\bra*{ \beta^{-1} - C_1\frac{L^2}{4 \pi^2 |k|^2} }\Theta(k)}{\kappa_\sharp L^{-d/2} \abs{\tilde{W}(k)}}
\, .
\end{align}
It is easy to see then that $\lambda$ being sufficiently large is equivalent to $\alpha$ being sufficiently small.
\end{rem}
We conclude this section with the following useful proposition which provides us with a comparison principle
for interaction potentials to check if they possess continuous transition points.
\begin{prop}\label{lem:comp}
Let $W\in \HH_\stable^c$ be an interaction potential such that the associated free energy $\crF^W_\kappa(\varrho)$ has a continuous transition point. Additionally, assume that $G\in \HH_\stable^c$ is such that $\argmin_{k \in \N^d/ \set{\mathbf{0}}}\tilde{G}(k)=\argmin_{k \in \N^d/ \set{\mathbf{0}}}\tilde{W}(k)= k^\sharp$ and $\tilde{G}(k^\sharp)=\tilde{W}(k^\sharp)$ with $\tilde{G}(k)\geq\tilde{W}(k)$ for all $k \neq k^\sharp,k \in \N^d$. Then $\crF^G_{\kappa}(\varrho)$ exhibits a continuous transition point. 
\end{prop}
\begin{proof}
As in the proof of~\cref{thm:spgap}, it is sufficient to show that at $\kappa=\kappa_\sharp$, the free energy $\crF^G_{\kappa_\sharp}(\varrho)$ has $\varrho_\infty$ as its unique minimiser. Noting that given the assumptions on $G$, the value of $\kappa_\sharp$ is the same for $G$ and $W$, we have for $\varrho \neq \varrho_\infty, \varrho \in \Leb^2(U) \cap \Prac(U)$
\begin{align}
\crF^G_{\kappa_\sharp}(\varrho)-\crF^G_{\kappa_\sharp}(\varrho_\infty) &= \beta^{-1}\cH(\varrho|\varrho_\infty) +
\frac{\kappa_\sharp}{2}\cE^G(\varrho-\varrho_\infty,\varrho-\varrho_\infty) \\
&=\beta^{-1}\cH(\varrho|\varrho_\infty) +
\frac{\kappa_\sharp}{2}\cE^W(\varrho-\varrho_\infty,\varrho-\varrho_\infty) + \frac{\kappa_\sharp}{2}\cE^{G-W}(\varrho-\varrho_\infty,\varrho-\varrho_\infty) \\
&= \bra*{\crF^W_{\kappa_\sharp}(\varrho)-\crF^W_{\kappa_\sharp}(\varrho_\infty)}+ \frac{\kappa_\sharp}{2}\cE^{G-W}(\varrho-\varrho_\infty,\varrho-\varrho_\infty)  \, ,
\end{align}
where $\cE^W(\varrho,\varrho)=\iint W(x-y ) \varrho(x) \varrho(y) \dx{x} \dx{y}$. Using the fact that the term in the brackets must be strictly positive, since the free energy $\crF^W_{\kappa_\sharp}(\varrho)$ associated to $W$ possesses a continuous transition point, we obtain
\begin{align}
\crF^G_{\kappa_\sharp}(\varrho)-\crF^G_{\kappa_\sharp}(\varrho_\infty) &>\frac{\kappa_\sharp}{2}\cE^{G-W}(\varrho-\varrho_\infty,\varrho-\varrho_\infty) \\
&=\frac{\kappa_\sharp}{2} \sum\limits_{k \in \N^d, k \neq k^\sharp }\frac{\tilde{G}(k)-\tilde{W}(k)}{N_k}\left(\sum\limits_{\sigma \in \Sym(\Lambda)}|\tilde{\varrho}(\sigma(k))|^2 \right) \geq 0 \,.
\end{align}
In the above estimate we have used the fact that $\tilde{G}(k^\sharp)=\tilde{W}(k^\sharp)$ and that $\tilde{G}(k)\geq\tilde{W}(k)$
for all other $k \in \N^d$. Thus, we have the desired result.
\end{proof}
%
%%%%%%%%%%%%%%%%%%%%%%%%%%%%%%%%%
%
\section{Applications}\label{S:app}

\subsection{The generalised Kuramoto model}\label{ss:kura}
Let $W(x)=-w_k(x)$, for some $k \in \N, k \neq 0$, as defined in~\eqref{e:def:wk}. Then we refer to the corresponding McKean SDE given by
\begin{align}
d X_t^i&= \frac{\kappa}{N}\sum_{i=1}^N w_k'(X_t^i-X_t^j) + \sqrt{2 \beta^{-1}}dB_t^i \quad i=1, \dots,N \, ,
\end{align}
 as the generalised Kuramoto model. For $k=1$, it corresponds to the so-called noisy Kuramoto system (also referred to as the Kuramoto--Shinomoto--Sakaguchi model (cf. ~\cite{kuramoto1981rhythms,sakaguchi1988phase,ABPRS})) which models the synchronisation of noisy oscillators interacting through their phases.  For infinitely many oscillators, we obtain a mean field approximation of the underlying particle dynamics given precisely by the McKean--Vlasov equation with $W(x)=-w_1(x)$.  It is well known that this system exhibits a  phase transition for some critical, $\kappa_c$ (cf. ~\cite{bertini2010dynamical}). For $k=2$, it corresponds to the Maiers--Saupe system which is a model for the synchronization of liquid crystals (cf.~\cite{constantin2004remarks,constantin2005note}). Again, in the mean field limit we obtain the McKean--Vlasov equation with the effective interaction potential, $W(x)= -w_2(x)$.  The  system exhibits a continuous transition point which represents the nematic-isotropic phase transition as the temperature is lowered, i.e., as $\kappa$ is increased.

Finally, let us mention that there is a larger picture in the Kuramoto model when different frequency oscillators are allowed, see ~\cite{ABPRS} for a nice review of the subject and ~\cite{CCP} for recent numerical work on phase transitions for this problem.

Although it is possible to directly apply \cref{thm:spgap} to prove the existence of a continuous phase transition for this system, we employ an alternative approach that gives us more qualitative information about the structure of the nontrivial solutions.
 
\begin{prop}~\label{prop:kura}
The generalised Kuramoto model exhibits a continuous transition point at $\kappa_c=\kappa_\sharp$. Additionally, for $\kappa>\kappa_c$, the equation $F(\varrho,\kappa)=0$ has only two solutions in $\Leb^2(U)$ (up to translations). The nontrivial one, $\varrho_\kappa$ minimises $\crF_\kappa$ for $\kappa>\kappa_c$ and converges in the narrow topology as $\kappa \to \infty$ to a normalised linear sum of equally weighted Dirac measures centred at the minima of $W(x)$.
\end{prop} 
\begin{proof}
The strategy of proof is similar to that of \cref{thm:spgap}, i.e, we show that at $\kappa=\kappa_\sharp$, $\varrho_\infty$ is 
the unique minimiser of the free energy. We do this by showing that $F(\varrho,\kappa)=0$ has a unique solution at $\kappa=\kappa_\sharp$, which implies, by \cref{prop:tfae}(since $W$ satisfies Assumption~\eqref{ass:B}) , uniqueness of the minimiser. 

For $W(x)=-w_{k^\sharp}(x)$, we can explicitly compute,
\begin{align}
F(\varrho,\kappa)=\varrho - \frac{e^{\beta \kappa\sqrt{L/2}(\tilde{\varrho}(k^\sharp)w_{k^\sharp}+ \tilde{\varrho}(-k^\sharp)w_{-k^\sharp})  }}{\int_{-L/2}^{L/2} e^{\beta \kappa\sqrt{L/2}(\tilde{\varrho}(k^\sharp)w_{k^\sharp}+ \tilde{\varrho}(-k^\sharp)w_{-k^\sharp})  }}=0 
\end{align}
Since $F(\varrho,\kappa)$ is translation invariant, one can always translate $\varrho$ so that $\tilde{\varrho}(-k^\sharp)=0$.
Thus we obtain the following simplified equation, 
\begin{align}
F(\varrho,\kappa)=\varrho - \frac{e^{\beta \kappa\sqrt{L/2}\tilde{\varrho}(k^\sharp)w_{k^\sharp}  }}{\int_{-L/2}^{L/2} e^{\beta \kappa\sqrt{L/2}\tilde{\varrho}(k^\sharp)w_{k^\sharp}  }}=0 \label{eq:kmkm}
\end{align}
Taking the inner product with $w_{k^\sharp}(x)$ we obtain,
\begin{align}
\tilde{\varrho}(k^\sharp) - \frac{\int_{-L/2}^{L/2} e^{\beta \kappa\tilde{\varrho}(k^\sharp)\cos(2 \pi k^\sharp x /L)  }w_{k^\sharp} \dx{x} }{\int_{-L/2}^{L/2} e^{\beta \kappa\tilde{\varrho}(k^\sharp)\cos(2 \pi k^\sharp x /L)  } \dx{x}}=0
\end{align}
After a change of variables we obtain,
\begin{align}
\tilde{\varrho}(k^\sharp) - \sqrt{\frac{2}{L}}\frac{\int_{0}^{\pi} e^{\beta \kappa\tilde{\varrho}(k^\sharp)\cos(y)  }\cos(y) \dx{x} }{\int_{0}^{\pi} e^{\beta \kappa\tilde{\varrho}(k^\sharp)\cos(y)  } \dx{x} }=0
\end{align}
We can express the above equation in the following form,
\begin{align}
M(a,\kappa):=\sqrt{\frac{2}{L}}\beta \kappa \frac{I_1(a)}{I_0(a)}=\sqrt{\frac{2}{L}}\beta \kappa r_0(a)= a\label{eq:besfp} \ ,
\end{align}
where the $I_n$ represent modified Bessel functions of the first kind having order $n$, $r_n(a):=\frac{I_{n+1}(a)}{I_n(a)}$, and $a= \beta \kappa \tilde{\varrho}(k^\sharp)$. This equation is similar to the one derived in Section VI of~\cite{bavaud1991eq}(cf. ~\cite{messer1982statistical,battle1977phase}). It is also qualitatively similar to the self-consistency equation associated with the two-dimensional Ising model.

For $\varrho=\varrho_\infty$, we know that $\tilde{\varrho}(\kappa_\sharp)=0$. 
We argue that any nontrivial solution of $F(\varrho,\kappa)=0$ must have
$\tilde{\varrho}(k^\sharp) \neq0$. Assume this is not the case, i.e.,
there exists $ \varrho_\kappa \neq \varrho_\infty$ which satisfies $F(\varrho_\kappa,\kappa)=0$
and $\tilde{\varrho_\kappa}(k^\sharp) =0$, then from \eqref{eq:kmkm} we have that $\varrho=\varrho_\infty$. Thus $F(\varrho,\kappa)$ has non-trivial solutions if and only if \eqref{eq:besfp} has nonzero solutions. One should note 
that since $I_1$ is odd and $I_0$ is even, nonzero solutions 
to \eqref{eq:besfp} come in pairs, i.e, if $a$ is a solution
so is $-a$. However, these two solutions are simply translates of each other.

We now show that if $\kappa \leq \kappa_\sharp= \sqrt{2L}/\beta$, \eqref{eq:besfp}
has no nonzero solutions.  As mentioned earlier it is sufficient to study the 
problem on the half line. Note first, that for $a>0$, $r_0(a)$ is increasing,
i.e, $r_0'(a) >0$ (cf. ~\cite[(15)]{amos1974computation}). 
Additionally, we have that,
\begin{align}
 r_0'(a)&= \frac{1}{2} + \frac{I_0(a)I_2(a)-I_1(a)^2}{2I_0(a)^2}-\frac{r_0(a)^2}{2}  
\end{align}
and so $r_0'(0) = \frac{1}{2}$.
We can now use the so-called Turan-type inequalities (cf. ~\cite{thiruvenkatachar1951inequalities,baricz2013turan}) to assert that $I_0(a)I_2(a)-I_1(a)^2<0$ for $a>0$. This tells us that,
\begin{align}
r_0'(a)&< \frac{1}{2} -\frac{r_0(a)^2}{2}   \ ,
\end{align}
with $r_0(a)>0$ for $a>0$. Using the fact that $\kappa \leq \kappa_\sharp$, we obtain,
\begin{align}
\frac{\partial {M}}{\partial {a}}(a,\kappa)& < 1- r_0(a)^2 \, .
\end{align}
We know now that $M(a,\kappa)$ is increasing for $a>0$, $M(0,\kappa)=0$, $\frac{\partial {M}}{\partial {a}}(0,\kappa)=1$, and $\frac{\partial {M}}{\partial {a}}(a,\kappa)$ is bounded above by 1 for $a>0$. Thus the curve
$y=M(a,\kappa)$ cannot intersect $y=a$ for any $a>0$. Thus $\varrho_\infty$
is the unique minimiser for $\kappa\leq \kappa_\sharp$, which implies by \cref{prop:CharactTP}~\eqref{prop:CharactTP:cont} that $\kappa_c=\kappa_\sharp$ is a continuous transition point.

 We will now show that for  $\kappa>\kappa_\sharp$, $\eqref{eq:besfp}$ has at most one solution for $a>0$.  We know that
 \begin{align}
 \frac{\partial {M}}{\partial {a}}(0,\kappa)& >1 \, .
 \end{align}
 Also for $a$ large enough, $a>M(a,\kappa)$ (since $r_0(a) \to 1$, as $a \to \infty$, and is strictly increasing). Thus by the intermediate value theorem, there exists at least one positive $a$ such that \eqref{eq:besfp} holds for every $\kappa>\kappa_\sharp$. One can now show that $\frac{\partial {M}}{\partial {a}}(a,\kappa)$ is strictly decreasing for 
 $a>0$. This is equivalent to showing that $r_0''(a)$ is strictly negative. We have
 \begin{align}
-r_0''(a) &= \frac{3}{4}r_0 + \frac{3 }{2}r_0^2 r_1 -2 r_0^3 -\frac{1}{4}r_0r_1r_2 =r_0 \bra*{\frac{3}{4} + \frac{3 }{2}r_0 r_1 -2 r_0^2 -\frac{1}{4}r_1r_2 } \, ,
 \end{align} 
 where we have used the formula $\frac{\dx{}}{\dx{a}}I_n=\frac{1}{2}\bra*{I_{n+1}+I_{n-1}}, n\geq1$. The ratios $r_n$ enjoy the following monotonicity and separation properties (cf.~\cite[(10),(11)]{amos1974computation}):
  \begin{align}
r_n &\leq r_{n+1} \label{eq:rmon} \, ,\\
\text{and}\qquad  \frac{a}{n +1+ \sqrt{a^2 +(n+1)^2}}&\leq r_n \leq \frac{a}{n + \sqrt{a^2 +(n+2)^2}}, \qquad a \geq 0, n \geq0 \, . \label{eq:ruplo}
  \end{align}
 Using these we obtain
 \begin{align}
-r_0''(a) &\stackrel{\mathclap{\eqref{eq:rmon}}}{\ \geq\ } r_0 \bra*{\frac{3}{4} + \frac{3 }{2}r_0 r_1  -\frac{5}{4}r_0^2 } =r_0 \bra*{\frac{3}{4}-\frac{3}{4}r_0 + r_0\bra*{ \frac{3 }{2}r_1  -\frac{1}{2}r_0} } \\
&\stackrel{\mathclap{r_0<1}}{\ >\ } r_0 \bra*{ r_0\bra*{ \frac{3 }{2}r_1  -\frac{1}{2}r_0} }  \stackrel{\mathclap{\eqref{eq:ruplo}}}{\ \geq\ } \frac{r_0^2}{2} \bra*{  \frac{3a}{2+ \sqrt{a^2 +9}}  - \frac{a}{ \sqrt{a^2 +4}}} \\
&\ =\ \frac{r_0^2}{2} \bra*{\frac{ (\sqrt{9 a^2 +36} -\sqrt{a^2 +9} -2) a }{(2+ \sqrt{a^2 +9})\sqrt{a^2 +4}}} >0, \qquad\text{ for } a>0 \, .
 \end{align}
 This implies that $\frac{\partial }{\partial {a}} \left(a- M(a,\kappa)\right) =1 -\frac{\partial {M}}{\partial {a}}(a,\kappa)$ changes sign only once. Thus~\eqref{eq:besfp} has only one solution, $a_\kappa$ for $a>0$ and $\kappa>\kappa_\sharp$. Additionally, $a<M(a,\kappa)$ if and only if $0<a<a_\kappa$ and $a>M(a,\kappa)$ if and only if $a>a_\kappa$. Now let $\kappa_2>\kappa_1> \kappa_\sharp$ with $a_{\kappa_1}$ and $a_{\kappa_2}$ the solutions of~\eqref{eq:besfp} at $\kappa_1$ and $\kappa_2$ respectively. We then have
\begin{align}
\frac{\kappa_2}{\kappa_1}a_{\kappa_1} &=  \frac{\kappa_2}{\kappa_1}
M(a_{\kappa_1},\kappa_1)=
M(a_{\kappa_1},\kappa_2) 
< M\bra*{\frac{\kappa_2}{\kappa_1}a_{\kappa_1},\kappa_2} \, ,
\end{align} 
where we have used the fact that $\kappa_2>\kappa_1$, the linearity of $M(a,\kappa)$ in $\kappa$, and that $M(a,\kappa)$ is strictly increasing for positive $a$. Using previous arguments, the above inequality tells us that $0<\frac{\kappa_2}{\kappa_1}a_{\kappa_1}<a_{\kappa_2}$ which implies that $a_{\kappa} \to \infty$, as $\kappa \to \infty$. Finally, we have the following form for the solution
\begin{align}
\varrho(x,a_\kappa)=\frac{1}{L}\frac{e^{a_\kappa \cos(2 \pi k x /L)}}{ I_0(a_\kappa)}  \, .
\end{align}
Let us denote by $\varrho(\dx{x},a_\kappa)$ the measure associated to the density 
$\varrho(x,a_\kappa)$. We will now show that for $k=1$, $\varrho(\dx{x},a_\kappa)$ converges to $ \delta_0$ as $a_\kappa \to \infty$ in the narrow topology, i.e., tested against bounded, continuous functions. The argument for other $k \in \N$ is then simply an extension of the $k=1$ case. Let $A$ be a continuity set of $\delta_0$, then if $0 \notin A$ it follows that $0 \notin \partial A$. By a large deviations argument, Laplace's principle, we have that
\begin{align}
\lim_{a_\kappa \to \infty} \bra*{\frac{1}{a_\kappa}\log\bra*{\frac{\pi}{L}\frac{\int_A e^{a_\kappa \cos(2 \pi  x /L) }\dx{x}}{\int_0^\pi
e^{a_\kappa \cos(y)} \dx{y}}}} = \sup_A \cos(2 \pi x/L)-1 <0  \quad \textrm{ if } 0 \notin A \, .
\end{align}
Thus, $\varrho(\dx{x},a_\kappa)(A) \to 0$ for every Borel set not containing $0$ and thus  $\varrho(\dx{x},a_\kappa)(A) \to 1$ for $0 \in A$. By the portmanteau theorem(cf. ~\cite[Theorem 2.1]{billingsley1999convergence}), we have the desired convergence. For arbitrary $k$, one can apply the same argument on periods of the function $\cos(2 \pi k x /L)$, and due to the periodicity/symmetry of the solution the masses in each Dirac point are equal.
\end{proof}
\subsection{The noisy Hegselmann--Krause model for opinion dynamics}\label{S:Hegselmann}
The noisy Hegselmann--Krause system(cf. ~\cite{hegselmann2002opinion}) models the opinions of $N$ interacting agents such that each agent is only influenced by the opinions of its immediate neighbours. In the large $N$ limit, we obtain again the McKean--Vlasov PDE with the interaction potential $W_{\mathrm{hk}}(x)=-\frac{1}{2}\bra*{\bra*{|x|-\frac{R}{2}}_-}^2$ for some $R>0$. The ratio $R/L$ measures the range of influence of an individual agent with $R/L=1$ representing full influence, i.e., any one agent influences all others. In order to analyse this system further, we compute the Fourier transform of $W_{\mathrm{hk}}(x)$ given by
\begin{align}
\tilde{W}_{\mathrm{hk}}(k)=\frac{\left(-\pi^2 k^2 R^2+2 L^2\right) \sin \left(\frac{\pi  k R}{L}\right)-2 \pi  k L R \cos \left(\frac{\pi  k R}{L}\right)}{4 \sqrt{2} \pi
   ^3 k^3 \sqrt{\frac{1}{L}}}, \quad k \in \N, k \neq 0 \, .
\end{align}
A simple consequence of the above expression is that the model has infinitely many bifurcation points for $R/L=1$.
For the other values of $R/L$ the problem reduces to a computational one, namely checking that the conditions of~\cref{thm:c1bif} are satisfied. Also, $W_{\text{hk}}(x)$ is normalised and  decays to $0$ uniformly as $R \to 0$, i.e., as the range of influence of an agent decreases so does its corresponding strength. We could define a rescaled version of the potential, $W_{\mathrm{hk}}^R(x)=-\frac{1}{2R^3}\bra*{\bra*{|x|-\frac{R}{2}}_-}^2$ which does not lose mass as $R \to 0$.  We conclude this subsection with the following result.
\begin{prop}
For $R$ small enough, the rescaled noisy Hegselmann--Krause model possesses a discontinuous transition point.
\end{prop}
\begin{proof}
We define $C:=\norm{W_{\mathrm{hk}}^R}_1$ and note that it is independent of $R$. The proof follows from the observation that $W_{\mathrm{hk}}^R \to - C \delta_0$ as $R \to 0$ and applying~\cref{cor:delta}.
\end{proof}
% \begin{rem}
% One can also derive the above result from~\cref{cor:lm} by using the potential $W(x)=-\frac{1}{2}\bra*{\bra*{|x|-\frac{1}{2}}_-}^2$
% and noting that $W_{\mathrm{hk}}^R(x)=\frac{1}{R}W(x/R)$. We then have $l=m=-1$ which satisfy all the conditions of~\cref{cor:lm}
% for $d=1$. 
% \end{rem}
\subsection{The Onsager model for liquid crystals}
In~\autoref{ss:kura}, we discussed the Maiers--Saupe model as a special case of the generalised Kuramoto model. In this subsection we discuss another model for the alignment of liquid crystals, i.e., the Onsager model which has as its interaction potential, $W(x)=\abs*{\sin\bra*{\frac{2 \pi}{L} x}}$.  As discussed in~\cite{chenstationary2010}, one can also study the potential $W_\ell(x)=\abs*{\sin\bra*{\frac{2 \pi}{L} x}}^\ell \in \Leb^2_s(U) \cap C^\infty(\bar{U})$ with $\ell \in \N, \ell \geq 1$, so that the Onsager and Maiers--Saupe potential correspond to the cases $\ell=1$ and $\ell=2$, respectively. We have the following representation of $W_\ell(x)$ in Fourier space
\begin{align}
\tilde{W_\ell}(k)=\frac{\sqrt{\pi } 2^{\frac{1}{2}-\ell} \cos \left(\frac{\pi  k}{2}\right) \Gamma (\ell+1)}{\Gamma \left(\frac{1}{2} (-k+\ell+2)\right) \Gamma
   \left(\frac{1}{2} (k+\ell+2)\right)} \, .\label{eq:onsf}
\end{align}
Any nontrivial solutions to the stationary dynamics correspond to the so-called nematic phases of the liquid crystals. We can obtain the following characterisation of bifurcations associated to the $W_\ell(x)$ and thus of the Onsager model.

\begin{prop}
We have the following results:
\begin{tenumerate}
\item The trivial branch of the Onsager model, $W_1(x)$, has infinitely many bifurcation points.
\item The trivial branch of the Maiers--Saupe model, $W_2(x)$, has exactly one bifurcation point.
\item The trivial branch of the model $W_\ell(x)$ for $\ell$ even has at least $\frac{\ell}{4}$ bifurcation points if $\frac{\ell}{2}$ is even and $\frac{\ell}{4} + \frac{1}{2}$ bifurcation points if $\frac{\ell}{2}$ is odd.
\item The trivial branch of the model $W_\ell(x)$ for $\ell$ odd has infinitely many bifurcation points if $\frac{\ell-1}{2}$ is even and at least $\frac{\ell +1}{4}$ bifurcation points if $\frac{\ell-1}{2}$ is odd.
\end{tenumerate}
\end{prop}
\begin{proof}
The proof of (b) follows from~\cref{prop:kura} so we only need to show (a),(c), and (d). We start by noting that $\tilde{W}_\ell(0)\geq 0$ and $\tilde{W}_\ell(k)=0$ for all odd $k \in \N$. We also note that $\frac{1}{\Gamma(z)}$ is an entire function with zeroes at all nonpositive integers and $\frac{1}{\Gamma(-(2n+1)/2)}, n \in \N$ is negative for all even $n$ and positive otherwise. For the rest of the proof we will always assume that $k>0$. We will now attempt to show that all nonzero values of $\tilde{W}_\ell(k)$ for $k>0$ are distinct. Assume $l$ is even, then we have the following explicit form of $\tilde{W}_\ell(k)$:
\begin{align}
\tilde{W_\ell}(k)=\frac{\sqrt{\pi } 2^{\frac{1}{2}-\ell} \cos \left(\frac{\pi  k}{2}\right) \Gamma (\ell+1)}{ \left(\frac{1}{2} (-k+\ell)\right)! 
   \left(\frac{1}{2} (\ell+k)\right)!} \, ,
\end{align}
where $k$ is assumed to be even and $k < \ell+2$(since it is zero for $k$ odd or $k\geq l+2$). From the above expression one can check that the denominator is strictly increasing as $k$ increases from $2$ to $\ell$, thus $|\tilde{W}_\ell(k)|$ is strictly decreasing. Thus the nonzero values of $\tilde{W}_\ell(k)$ are distinct for $\ell$ even. For $\ell$ odd, we first note that by simple integration by parts we can derive the following recursion relation
\begin{align}
\tilde{W}_\ell(k)= -\frac{\ell(\ell-1)}{k^2 - \ell^2} \tilde{W}_{\ell-2}(k)\, \label{eq:onre},
\end{align}
 where again $k$ is even(and thus not equal to $\ell$). For $\ell=1$, we have the following alternative formula for $\tilde{W}_\ell(k)$ for even $k$:
 \begin{align}
\tilde{W}_1(k)=\sqrt{\frac{2}{\pi }}\frac{ (\cos (\pi  k)+1)}{1-k^2} \, \label{eq:onbif}.
 \end{align}
 It is clear now that for $\ell=1$, $\tilde{W}_1(k)$ has distinct(and in fact negative values) for $k$ even. From the recursion formula in~\eqref{eq:onre} it follows that this holds true for all odd $\ell$, i.e., $|\tilde{W}_\ell(k)|$ takes distinct values for $k$ even.

Assume now that $\ell=1$(i.e. the Onsager model), then as mentioned earlier we can deduce from~\eqref{eq:onbif} that $\tilde{W}_1(k)$ is distinct and negative for all $k$ even. It follows that $\tilde{W}_1(k)$ satisfies the conditions of~\cref{thm:c1bif} for all even $k$, thus completing the proof of (a).

Now let $\ell>2$ and even. It is clear from the expression in~\eqref{eq:onsf} that then $\tilde{W}_\ell(k)$ can be negative only if $\cos(k \pi/2)/\Gamma \left(\frac{1}{2} (-k+\ell+2)\right)$ is negative. This happens if and only if $\frac{k}{2}$ is odd and $k<\ell+2$ since if $k \geq \ell+2$, $\frac{1}{\Gamma \left(\frac{1}{2} (-k+\ell+2)\right)}$ is evaluated at a negative integer and thus $\tilde{W}_\ell(k)=0$. Since by the previous arguments each $\frac{k}{2}$ odd with $k<\ell+2$ corresponds to a distinct value of $\tilde{W}_\ell(k)$, we can apply ~\cref{thm:c1bif} to deduce that such $k$ correspond to bifurcation points. Given an $\ell>2$ and even, there are $\frac{\ell}{4} + \frac{1}{2}$ such $k$ if $\frac{\ell}{2}$ is odd and $\frac{\ell}{4}$ if $\frac{\ell}{2}$ is even.  This completes the proof of (c).

Now, we let $\ell>2$ and odd. One can check again that $\tilde{W}_\ell(k)$ is negative if and only if $\frac{k}{2}$ is odd and $k < \ell +2$ when $\frac{\ell-1}{2}$ is odd and if $k$ is even, but $\frac{k}{2}$ is odd if $k < \ell+2$, when $\frac{\ell-1}{2}$ is even. For $\frac{\ell-1}{2}$ odd there are $\frac{\ell +1}{4}$ such $k$, while for $\frac{\ell-1}{2}$ even there are infinitely many such $k$. Applying~\cref{thm:c1bif} again, gives us (d). 
\begin{comment}
Finally we let $\ell>2$ and odd. One can check again that $\tilde{W}_\ell(k)$ is negative if and only if $\frac{k}{2}$ is odd and $k < \ell +2$ when $\frac{\ell-1}{2}$ is odd and if $k$ is even, but $\frac{k}{2}$ is odd if $k < \ell+2$, when $\frac{\ell-1}{2}$ is even. For $\frac{\ell-1}{2}$ odd there are $\frac{\ell +1}{4}$ such $k$, while for $\frac{\ell-1}{2}$ even there are infinitely many such $k$. Applying~\cref{thm:c1bif} again, gives us (d).
\end{comment}
\end{proof}
The above result provides us with a finer analysis to that presented in~\cite{chenstationary2010}, as we are able to count the solutions for general odd and even $\ell$, instead of just proving the existence of nontrivial solutions. The above result also generalises the work in~\cite{luciaexact2010}  which studied a truncated version of the Onsager model with only a finite number of modes and proved the existence of nontrivial solutions. It also partially recovers results from~\cite[Theorem 2]{niksirat2015stationary} in which the non-truncated Onsager model is analysed. We refer the reader to~\cite{vollmer2017critical} for an analysis of the Onsager model
in 2 dimensions, i.e., for liquid crystals that live in 3 dimensions with two degrees of freedom.

\subsection{The Barr\'e--Degond--Zatorska model for interacting dynamical networks}
The Barr\'e--Degond--Zatorska system ~\cite{barre2017kinetic} models particles that interact through a dynamical network of links. Each particle interacts with its closest neighbours through cross-links modelled by springs which are randomly created and destroyed. Taking the combined mean field and overdamped limits one obtains the McKean--Vlasov equation with the interaction potential given by
\begin{align}
W(x)=
\begin{cases}
(|x|-\ell)^2 -(R -\ell)^2 & |x| <R \\
0 & |x|\geq R
\end{cases} \, ,
\end{align}
for two positive constants $0<\ell\leq R \leq L/2$. In~\cite[Theorem 6.1]{barre2017kinetic}, using  formal asymptotic analysis, it was shown (and later numerically verified in~\cite{barre2018particle}) that one can provide conditions for continuous and discontinuous transitions for the above potential based on the values of  the Fourier modes. We restate their result using our notation for the convenience of the reader.
\begin{prop}[{Sharp characterisation of transition point by formal asymptotics ~\cite[Theorem 6.1]{barre2017kinetic}}]
Consider the Barr\'e--Degond--Zatorska model with $\ell, R,L$ chosen such that $\beta \kappa \tilde{W}(1)+ \sqrt{2L} <0$ and $\beta \kappa \tilde{W}(k)+ \sqrt{2L} >0$ for all $k \neq 1, k \in \N$. Then
\begin{tenumerate}
\item If $2 \tilde{W}(2) - \tilde{W}(1) >0$, then the system exhibits a continuous transition point.
\item If $2 \tilde{W}(2) - \tilde{W}(1) <0$, then the system exhibits a discontinuous transition point.
\end{tenumerate}
\end{prop}
The assumptions in the proposition essentially imply a separation of the Fourier modes. It follows immediately under these assumptions that $k=1$ satisfies the conditions of~\cref{thm:c1bif} and thus  $\kappa_*=-\frac{ (2L)^{\frac{1}{2}} }{\beta \tilde{W}(1)} $ corresponds to a bifurcation point of the system. Additionally, looking at Figure ~\ref{fig:dcctp} one can see that the conditions (a) and (b) from the above proposition are consistent with our analysis for the existence of continuous and discontinuous transition points. If $\tilde{W}(1)$ and $\tilde{W}(2)$ are resonating/near-resonating then it follows that condition (b), i.e., $2 \tilde{W}(2) - \tilde{W}(1) <0$ must hold for $\delta_*$ small, where $\delta_*$ is as introduced in~\cref{def:del}. Indeed, let $k=1 ,2$ be elements
of the set $K^{\delta_*}$, then we have $2 \tilde{W}(2) - \tilde{W}(1) =\tilde{W}(1) + 2(\tilde{W}(2)- \tilde{W}(1) ) \leq \tilde{W}(1) +2\delta_* <0$, for $\delta_*$ sufficiently small. Similarly, using~\cref{lem:comp} and comparing with an $\alpha$-stabilised potential say $G_\alpha$, one can argue that  if $\tilde{W}(1)$ is the dominant mode then condition (a), i.e., $2 \tilde{W}(2) - \tilde{W}(1) >0$ must hold for $\alpha$ small, where $\alpha$ is as defined in~\cref{defn:astable}.
\subsection{The Keller--Segel model for bacterial chemotaxis}
The (elliptic-parabolic) Keller--Segel model is used to describe the motion of a group of bacteria under the effect of the concentration gradient of a chemical stimulus, whose distribution is determined by the density of the bacteria. This phenomenon is referred to as \emph{chemotaxis} in the biology literature~\cite{keller1971model}.  For this system, $\varrho(x,t)$ represents the particle density of the bacteria and $c(x,t)$ represents the availability of the chemical resource. The dynamics of the system are then described by the following system of coupled PDEs:
\begin{align}
\label{eq:ksev}
\begin{alignedat}{3}
\partial_t \varrho& =\dive \bra*{\beta^{-1}\nabla \varrho + \kappa \varrho \nabla c}  \qquad && (x,t) \in U \times (0,\infty)\, , \\
-(-\Delta)^{s} c &=\varrho \quad && (x,t) \in U \times [0,\infty) \, ,\\
\varrho(x,0) &= \varrho_0 \quad && x \in U \times \{0\} \, ,\\
\varrho(\cdot, t) &\in C^2(U) \quad && t \in [0,\infty),
\end{alignedat}
\end{align}
for $s \in(\frac{1}{2},1]$. The link between the model in~\eqref{eq:ksev} and the McKean--Vlasov equation is immediately noticed if one simply inverts $-(-\Delta)^s$ to obtain $c$. Thus, the stationary Keller--Segel equation is given by,
\begin{align}
\dive \bra*{\beta^{-1}\nabla \varrho + \kappa \varrho \nabla \Phi^s \star \varrho}=0  \quad & x \in U  \label{eq:sKS} \, ,
\end{align}
with $\varrho \in C^2(\bar{U})$ and where $\Phi^s$ is the fundamental solution of $-(-\Delta)^s$. Since $\Phi^s$ does not, in general, satisfy assumption~\eqref{ass:B}, \cref{thm:wellpsPDE} does not apply directly. However we can circumvent this issue to obtain the following result.
\begin{thm}
 Consider the stationary Keller--Segel equation~\eqref{eq:sKS}. For $d\leq 2$ and $s \in (\frac{1}{2},1]$, it has smooth solutions and its trivial branch $(\varrho_\infty,\kappa)$ has infinitely many bifurcation points.
\end{thm}
\begin{proof}
\begin{figure}[t]
\centering
\begin{minipage}[c]{0.45\textwidth}
\centering
    \includegraphics[width=\linewidth]{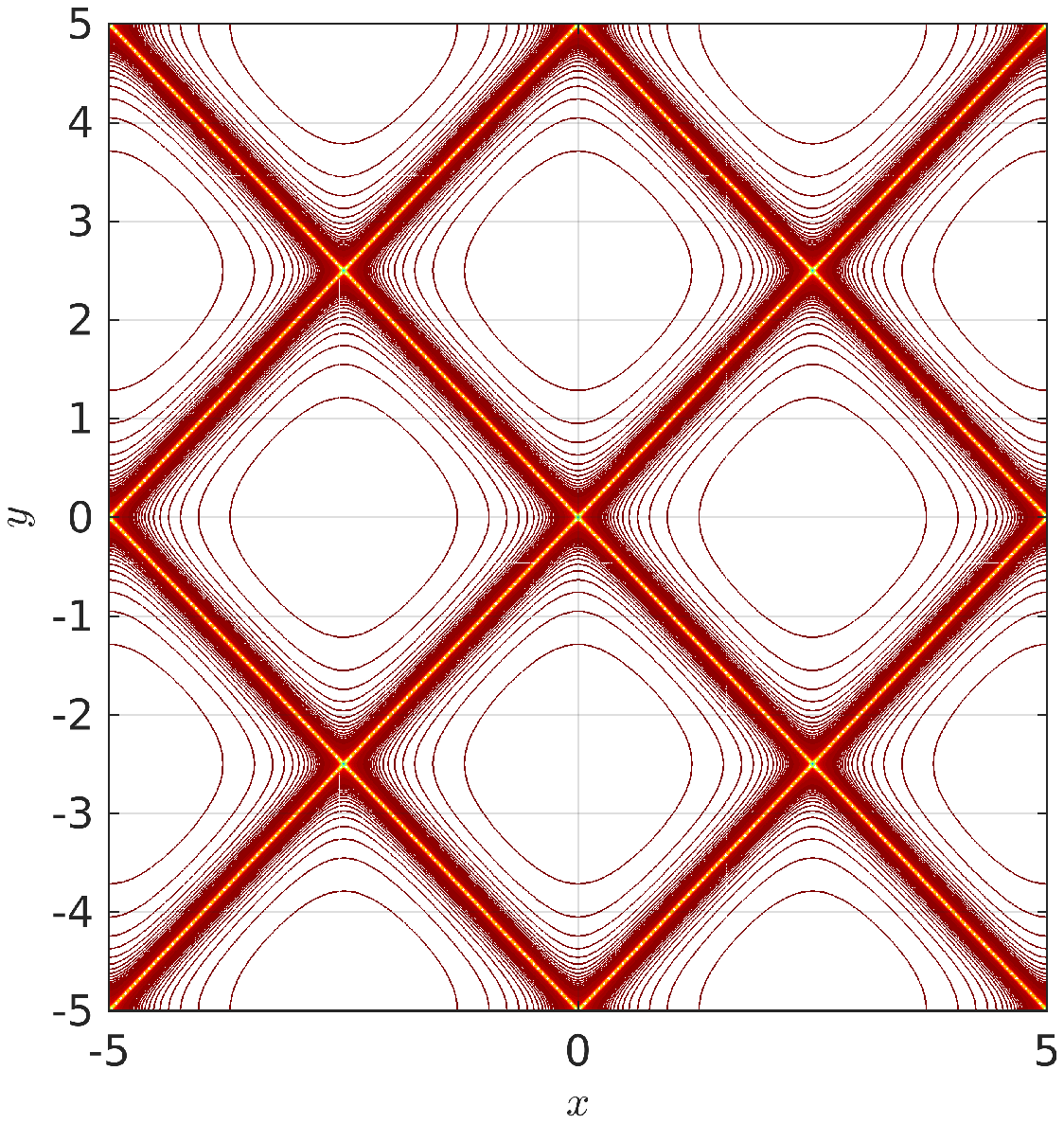}
    \caption*{(a)}
\end{minipage}
\begin{minipage}[c]{0.45\textwidth}
\centering
    \includegraphics[width=\linewidth]{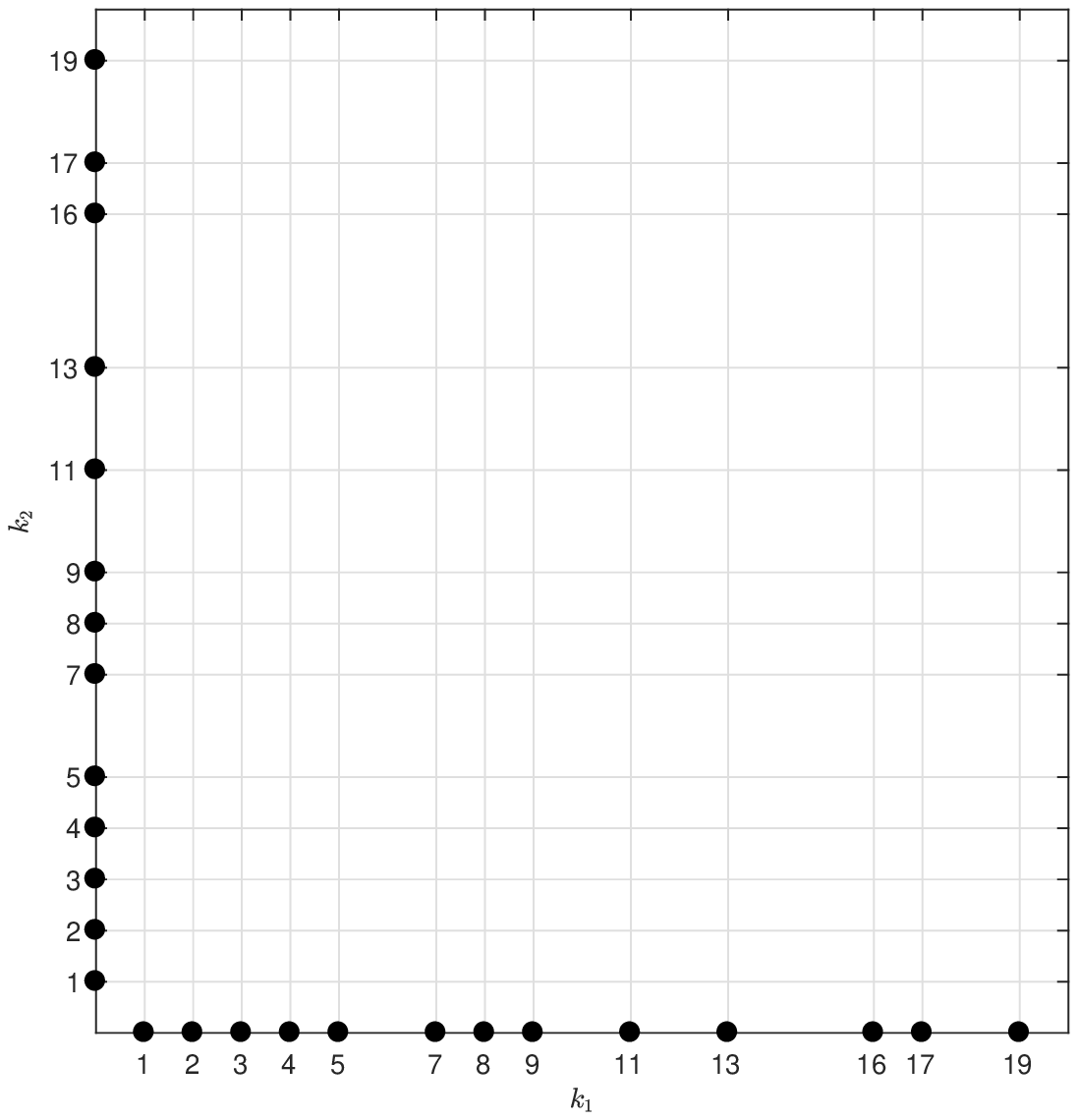}
     \caption*{(b)}
\end{minipage}
\caption{(a). Contour plot of the Keller--Segel interaction potential $\Phi^s$ for $d=2$ and $s=0.51$. The orange lines indicate the positions at which the potential  is singular (b). The associated wave numbers which correspond to bifurcation points of the stationary system}
\label{fig:ks2d}
\end{figure}
$\Phi^s$ is given by the following formal Fourier series,
\begin{align}
\Phi^s(x) = -\bra*{\frac{2 \pi}{L}}^{2s} \sum_{k \in \N^d \setminus\set{0}} \frac{N_k}{|k|^{2s}} w_k \in \cD(U)' \, .
\end{align}
The weak form of~\eqref{eq:sKS} is then given by
\begin{align}
-\beta^{-1}\intom{\nabla \varphi \cdot \nabla \varrho } -\kappa \intom{\varrho \nabla \varphi \cdot \nabla c}=0,
\quad \forall \varphi \in \SobH^1(U) \, , \label{eq:wsKS}
\end{align}
where we look for solutions $\varrho$ in $\SobH^1(U)\cap \Prac(U)$ and $c= \Phi^s \star \varrho$.
We start by noticing that any fixed point of $\cT^{ks}$ is a weak solution of~\eqref{eq:wsKS} where the map $\cT^{ks} : \Leb^2(U)\to \Leb^2(U)$ is defined as follows:
\begin{align}
\cT^{ks}\varrho=\frac{1}{Z(c,\kappa,\beta)}e^{-\beta \kappa c}, \quad\text{where}\quad Z(\varrho,\kappa,\beta)= \intom{e^{- \beta \kappa c}} \, .
\end{align}
Indeed, let $\varrho$ be such a fixed point and $0<\epsilon<s-\frac{1}{2}$, then
\begin{align}
 \sum\limits_{k \in \Z^d} |k|^{2 +2 \epsilon}|\tilde{c}(k)|^2&=\bra*{\frac{2 \pi}{L}}^{4}\sum\limits_{k \in \N^d\setminus\set{0}} \frac{1}{|k|^{4s-2 -2 \epsilon}}\sum\limits_{k \in \Sym(\Lambda)} |\tilde{\varrho}(\sigma(k))|^2 < \infty \, .
\end{align}
Thus $c \in \SobH^{1+ \epsilon}(U)$ which by the Sobolev embedding theorem for $d\leq 2$ implies that $c \in C^0(U)$. This tells us that $\varrho \in \SobH^1(U)\cap \Prac(U)$ with $\nabla \varrho  = -\beta \kappa Z^{-1} \, e^{-\beta \kappa c}\nabla c $. Plugging $\varrho$ into~\eqref{eq:wsKS}, we see immediately that it is a solution. The reverse implication follows by arguments identical to those in~\cref{thm:wellpsPDE}. 

Since $\varrho_\infty$ is a solution to $\varrho=\cT^{ks} \varrho$, for all $\kappa>0$, we need to check that any solution of the fixed point equation is smooth. Assume that $\varrho \in \SobH^\ell(U)$, i.e., $\sum_{k \in \Z^d}|k|^{2\ell} | \tilde{\varrho}(k)|^2 < \infty$. Then for $0<\epsilon<s-\frac{1}{2}$ we have that
\begin{align}
\sum\limits_{k \in \Z^d} |k|^{2\ell+2 + 2\epsilon}|\tilde{c}(k)|^2 &=\bra*{\frac{2 \pi}{L}}^{4s}\sum\limits_{k \in \N^d\setminus\set{0}} \frac{|k|^{2\ell }}{|k|^{4 s -2 -2 \epsilon}}\sum\limits_{k \in \Sym(\Lambda)} |\tilde{\varrho}(\sigma(k))|^2  \, , \nonumber \\
&<\bra*{\frac{2 \pi}{L}}^{4s}\sum\limits_{k \in \Z^d} |k|^{2\ell } |\tilde{\varrho}(k)|^2 < \infty \, . \label{e:Gamma:est}
\end{align}
 Thus $c \in \SobH^{\ell+1 +\epsilon}(U)$ and by the Sobolev embedding theorem we have that $\SobH^{\ell +1 +\epsilon}(U)$ is continuously embedded in  $C^{\ell}(U)$. Thus for all multiindices $\alpha$ such that  $|\alpha|\leq \ell$, we have that $\partial_\alpha c \in \Leb^ \infty(U)$\,. Since $\varrho = Z^{-1} e^{-\beta \kappa c}$, computing $\partial_\alpha \varrho$ with $|\alpha|=\ell+1$ gives us:
 \begin{align}
\partial_\alpha \varrho= Z^{-1} e^{-\beta \kappa c} \partial_\alpha c + F(Z^{-1},\beta \kappa,\partial_\xi c), \textrm{ for all } |\xi|\leq l \, .
 \end{align}
 Thus $\partial_\alpha c$ enters the expression for $\partial_\alpha \varrho$ linearly. Since all lower derivatives of $c(x)$ are bounded, one can then check that $\norm*{\partial_\alpha \varrho}_2 < \infty$ and thus $\varrho \in \SobH^{\ell+1}(U)$. We can then bootstrap to obtain smooth solutions. 

 Observe now that for $d \leq 2$ and $s \in (\frac{1}{2},1]$, $\Phi^s \in \Leb^2_s(U)$. For $d=1$, ~\cref{thm:c1bif} applies directly and the bifurcation points are given by:
 \begin{align}
\kappa_*= \bra*{\frac{L}{2 \pi}}^{2s} \frac{|k|^{2s} L}{\beta}, \textrm{ for } d=1 \,.
\end{align} 
For $d =2$ one can notice that $\Phi^1(x)=\Phi^1(\Pi(x))$ for any permutation $\Pi$ of the $d$ coordinates. Our
strategy will be to apply~\cref{thm:c1bif} after reducing the problem to the symmetrised space $\Leb^2_{\operatorname{ex}}(U)$  and then use the discussion in~\cref{rem:l2pi}. Then, showing that a particular $[k]$ corresponds to a bifurcation point reduces to the condition
\begin{align}
\mathrm{card}\set*{[k] : \frac{\tilde{W}([k])}{\Theta([k])}=\frac{\tilde{W}([k^*])}{\Theta([k^*])}}=\mathrm{card}\set*{ [k] : \frac{\tilde{W}([k])}{\Theta([k])}= -\bra*{\frac{2\pi}{L}}^{2s}\frac{1}{|k|^{2s}L}} =1  \ , 
\end{align}
which holds for example for  $[k]=\{(1,0),(0,1)\}$. We argue that $\kappa_*= -\frac{L^{\frac{d}{2}} \Theta(\pra{p^{n}})}{\beta \tilde{W}(\pra{p^{n}})} $, where $\pra{p^n}=\{(p^n,0),(0,p^n)\}$, $p$ is a prime, and $n \in \N$ satisfy the conditions of being a bifurcation point. We need to check that 
\begin{align}
\mathrm{card}\set*{ [p^n] : \frac{\tilde{W}([p])}{\Theta([p])}= -\bra*{\frac{2\pi}{L}}^{2s}\frac{1}{|p|^{2s}L}} =1 \, ,
\end{align}
which is equivalent to checking that given a prime $p$ there is a unique way (up to permutations) of expressing $p^{2n}$ as the sum of two squares and this is precisely $(p^n)^2 + 0^2$. Jacobi's two square theorem tells us that number of representations, $r(z)$, of a positive integer $z$ as the sum of two squares is given by the formula
\begin{align}
r(z)= (d_{1,4}(z)-d_{3,4}(z)) \, ,
\end{align}
where $d_{\ell,4}(z)$ is the number of divisors of $z$ of the form $4 k +\ell, k\in \N,\ell \geq 1$. If $p=2$, then $d_{1,4}(2^{2n})=1$ and $d_{3,4}(2^{2n})=0$ and thus $r(2^{2n})=1$. For any odd prime, $p$, we know that it is either of the form $4k+1$ or $4k+3$.
For either case, one can check that we  have $d_{1,4}(p^{2n})=1+n$ and $d_{3,4}(p^{2n})=n$ and thus $r(p^{2n})=1$. The expression for the bifurcation points then follows from the discussion in~\cref{rem:l2pi}.
\end{proof}
%
% It is also interesting to note that we can obtain a well-posed minimisation problem under certain conditions due to the regularising properties of $\Phi^s$.
% \begin{proof}
% We adapt the approach used in ~\cite{gajewski1998global}. Consider the the convex conjugate functions, $f(z)=z \log z$
% and $g(z)=e^{z-1}$ defined for $z\geq 0$. They obviously satisfy Young's inequality,
% \begin{align}
% zy \leq f(z) + g(y) \, . \label{eq:ye}
% \end{align}
% Now we should note that we can rewrite the free energy for the Keller--Segel system as follows,
% \begin{align}
% \crF_\kappa(\varrho)= \beta^{-1} \intom{\varrho \log \varrho} + \frac{\kappa}{2} \intom{c \varrho} \, .
% \end{align}
% Using~\eqref{eq:ye}, we obtain,
% \begin{align}
% \crF_\kappa(\varrho) &\geq \beta^{-1} \intom{\varrho \log \varrho} - \frac{\kappa}{2}\intom{|c|\varrho} \\
% & \geq \beta^{-1} \intom{\varrho \log \varrho}  -\frac{\kappa}{2}\bra*{\intom{\varrho \log \varrho} +\intom{e^{|c|-1}}} \\
% & \geq -\frac{\kappa}{2}\intom{e^{|c|-1}}
% \end{align}
% The above inequality can be thought of the analogue of the log-HLS inequality on $\R^2$ for the torus. (\todo{can only bound
% |c| for dynamics not for abritrary L2 functions})
% \end{proof}

\appendix

\section{Results from bifurcation theory} \label{S:lyapschmid}
Let $X$ be a separable Hilbert space and denote by $L(X)$ the set of bounded, linear, operators on $X$. For $F:X \times \R^+ \to X$ a twice Fr\'echet-differentiable mapping, we define $N = \ker D_x F (x_0,\kappa_*) $ and $R= \Ima D_x F (x_0,\kappa_*)$. Furthermore, we  assume that, $F(x_0,\kappa_*)=0$ for some $(x_0,\kappa_*) \in X \times \R^+$. We also assume that $D_x F (x_0,\kappa_*)$ is a Fredholm operator with index zero and that $\dim N =1$ from which follows that $\codim R =1$. Then, we have the following decompositions into complementary subspaces of~$X$:
\begin{equation}\label{eq:profSpaces}
X= N \oplus X_0 \qquad\text{and}\qquad X= R \oplus Z_0  \ ,
\end{equation}
where $N= \mathrm{span}[v_0]$ and $Z_0= \mathrm{span}[z_0]$ for some $v_0, z_0 \in X$. We can also pick $X_0$ to be orthogonal to $N$ and closed, i.e., $X_0=\{x \in X: \skp*{x,v_0}_X=0\}$, where $\skp{\cdot,\cdot}_X$ denotes the
inner product on $X$. This allows us to define the following canonical projection operators:
\begin{align}\label{eq:projOps}
P:X \to N \qquad\text{and}\qquad Q:X \to Z_0 \ ,
\end{align}
which, by the closed graph theorem, are continuous.
\begin{thm}\label{thm:lyapschmid}
There is a neighbourhood $U \times V$ of $(x_0,\kappa_*)$ in $X \times \R^+$ such that the implicit equation
\begin{align}\label{eq:Fimplicit}
F(x,\kappa)=0, \qquad (x,\kappa) \in U \times V \, ,
\end{align}
is equivalent to a finite-dimensional problem, i.e., there exists $\tilde{U}\subset N$ and $\Phi: \tilde{U} \times V \to Z_0$ continuous with $\Phi(v_0,\kappa_*) = 0$ for some $(v_0,\kappa_*) \in \tilde{U} \times V$ such that~\eqref{eq:Fimplicit} is equivalent to
\begin{align}
\Phi(v,\kappa)&=0 , \qquad (v,\kappa)\in \tilde{U} \times V \subset N \times \R^+ \, . 
\end{align}
The function $\Phi$ is referred to as the bifurcation function.
\end{thm}
\begin{proof}
Using the projection operators defined in~\eqref{eq:projOps}, we can restate the bifurcation problem~\eqref{eq:Fimplicit} as follows,
\begin{align}
Q F\bra[\big]{Px + (I-P)x,\kappa}=0 \quad\text{and}\quad (I -Q) F\bra[\big]{Px + (I-P)x,\kappa}&=0  \quad \text{for } (x,\kappa) \in U \times V \, . \label{eq:eqbif}
\end{align}
Let us recall the orthogonal splitting~\eqref{eq:profSpaces} from which we obtain two open neighbourhoods $\tilde U \subset U \cap N$ and $W \subset U \cap X_0$ such that $(v_0 , w_0) = (P x_0, (I-P)x_0) \in \tilde U \times W$. we now define the operator $G : \tilde{U}\times W \times V \to R$ by
\begin{align}
G(v,w,\kappa)&=(I -Q) F(v+w,\kappa), \quad\text{with} \quad v =Px \quad\text{and}\quad w= (I -P)x , 
\end{align}
We thus have that $G(v_0,w_0,\kappa_*)=0$. Since the projection operators are continuous, we can compute $D_w G(v_0,w_0,\kappa_*)
=(I -Q)D_x F(x_0,\kappa_*):X_0 \to R$ with $R$ defined in~\eqref{eq:profSpaces}. One can check that this mapping is a homeomorphism between $X_0$ and $R$. Applying
the implicit function theorem, we see that 
\begin{align}
G(v,w,\kappa)&=0 \textrm{ in } \tilde{U}\times W \times V \, ,
\end{align}
is equivalent to
\begin{align}
w&=\Psi(v,\kappa) \textrm{ for some } \Psi:\tilde{U}\times V \to W  \, ,
\end{align}
such that
\begin{align}
w_0&= \Psi(v_0,\kappa_*) \textrm{ and } D_\gamma\Psi(v,\kappa)=-(D_w G(\Psi(v,\kappa),v,\kappa))^{-1} D_\gamma G(\Psi(v,\kappa),v,\kappa) \label{eq:zerdiv} \ ,
\end{align}
where $\gamma=(v,\kappa) \in \tilde{U} \times V$ and $D_\gamma [\cdot]=\left(D_v [\cdot] \, D_\kappa [\cdot]  \right)$. Inserting the function $\Psi$ into \eqref{eq:eqbif} we obtain,
\begin{align}
\Phi(v,\kappa)=(I -Q) F (v + \Psi(v,\kappa),\kappa)&=0 \ ,
\end{align}
which is the desired result. Finally, the continuity of $\Psi$ and $Q$ gives us the desired continuity of~$\Phi$.
\end{proof}
Since we know that the function $\Psi$ is $C^1$ we can expand about $(v_0,\kappa_*)$  to obtain
\begin{align}
\Psi(\gamma_0+h)=w_0 + D_\gamma(\gamma_0)h + r_1(h) \ , \label{eq:expansion1} 
\end{align}
where $\gamma_0=(v_0,\kappa_*)$ and $\lim\limits_{h \to 0} \norm*{r_1(h)}/\norm*{h}=0$. It should also be noted that
\begin{align}
D_v G(v_0,w_0,\kappa_*) =(I -Q)D_x F(x_0,\kappa_*):N \to R=0 \in L(N,R) \, .
\end{align}
Thus, using \eqref{eq:zerdiv}, we  have that $D_v \Psi(v_0,\kappa)=0 \in L(N,X_0)$. We now state the Crandall--Rabinowitz theorem (cf. ~\cite{nirenberg2001topics,kielhofer2006bifurcation}) for bifurcations with a one-dimensional kernel.
\begin{thm}\label{thm:cr}
	Consider a separable Hilbert space $X$ with $U \subset X$ an open neighbourhood of 0, and a nonlinear $C^2$ map, $F: U \times V \to X$, where $V$ is an open subset of $\R^+$ such that $F(0,\kappa)=0$ for all $\kappa \in V$. Assume the following conditions are satisfied for some $\kappa_* \in V$:
	\begin{enumerate}
		\item $D_x(0,\kappa_*)$ is a Fredholm operator with index zero and has a one-dimensional kernel \, .
		\item $D^2_{x\kappa}(0,\kappa_*)[\hat{v_0}] \notin \Ima(D_x(0,\kappa_*))$, where $\hat{v_0} \in \ker (D_x(0,\kappa_*)), \norm*{\hat{v_0}}
			=1$ \, .
	\end{enumerate}
	Then, there exists a nontrivial $C^1$ curve through $(0,\kappa_*)$ such that for some $\delta>0$,
	\begin{align}
           \{(x(s),\kappa(s))\ : s\in (-\delta,\delta), x(0)=0, \kappa(0)=\kappa_* \} \ ,
	\end{align}
and $F(x(s),\kappa(s))=0$. Additionally, for some neighbourhood of $(0,\kappa_*)$,
 this is the only such solution (apart from the trivial solution) and it has the following form:
\begin{equation}
x(s)=s\hat{v_0} +\Psi(s\hat{v_0},\psi(s)) \,, \qquad \kappa(s)=\psi(s) \ ,
\end{equation}
where $\Psi$ is the implicit function previously described and $\psi:(-\delta,\delta ) \to V$ is a $C^1$ function such that $\psi(0)=\kappa_*$. Additionally, every nontrivial solution of $F$ in some neighbourhood of $(0,\kappa_*)$ in $N \times \R^+$ is of the form $(s \hat{v_0},\psi(s))$. Similarly, every nontrivial solution of $F$ in some neighbourhood of $(0,\kappa_*)$ in $N \times X_0 \times \R^+$ is of the form $(s \hat{v_0},\Psi(s\hat{v_0},\psi(s)),\psi(s))$.

\end{thm}
Since we have an entire branch of solutions, we can check that $D_\kappa \Psi(v_0,\kappa_*)=0$. Thus we obtain a simplified expression of the form
\begin{align}
x(s)=s\hat{v_0} + r_1(s\hat{v_0},\psi(s)) \, ,
\end{align} 
such that $\lim\limits_{|s| +|\psi(s)-\kappa_*| \to 0} \frac{\norm*{r_1(s\hat{v_0},\psi(s))}}{|s| +|\psi(s)-\kappa_*|}=0$. 

We finally present the following result from~\cite[Theorem 29.1]{deimling1985nonlinear}, often referred to as the Rabinowitz alternative (cf.~\cite{rabinowitz1971some}).
\begin{thm}\label{thm:rabalt}
Let $X$ be a real Banach space, $V \subset X \times \R$ a neighbourhood of $(0,\kappa_*)$, $G: \bar{V} \to X$ completely continuous,
and $G(x, \kappa)=o(|x|)$ as $x \to 0$ uniformly in $\kappa$ on compact subsets of $\R^+$. Let $K$ be a compact, linear operator on $X$ and $\kappa_*$ be a characteristic value of $K$ having odd algebraic multiplicity with $F(\varrho,\kappa)=x-\kappa K x + G(x,\kappa)$. If $\cC_V \subset V $ is the set of nontrivial solutions of $F(x,\kappa)=0$ in $V$ and $\cC_{V,\kappa_*}$ is the connected component of $\bar{\cC_V}$ containing $(0,\kappa_*)$, then $\cC_{V,\kappa_*}$ has at least one of the following two properties:
\begin{enumerate}
\item $\cC_{V,\kappa_*} \cap \partial V \neq \emptyset$.
\item $\cC_{V,\kappa_*}$ has an odd number of trivial zeros $(0,\kappa_i) \neq (0,\kappa_*)$, where the $\kappa_i$ are characteristic values of $K$ with odd algebraic multiplicity.
\end{enumerate}
\end{thm}

\bibliographystyle{myalpha}

\bibliography{refs}

\end{document}